\documentclass[10pt,oneside,leqno]{article}
\usepackage{amssymb,amsmath,amsthm,verbatim,titlesec}
\usepackage{mathrsfs,slashed}
\usepackage[all]{xy}
\usepackage{hyperref}

\titleformat{\section}[block]
  {\sc\large\filcenter}
  {\S\thesection.}{.5em}{}

\theoremstyle{plain}
\newtheorem{theorem}{Theorem}[section]

\theoremstyle{plain}
\newtheorem{corollary}[theorem]{Corollary}

\theoremstyle{plain}
\newtheorem{prop}[theorem]{Proposition}

\theoremstyle{plain}
\newtheorem{lemma}[theorem]{Lemma}

\theoremstyle{definition}
\newtheorem{defn}[theorem]{Definition}

\theoremstyle{definition}
\newtheorem{example}[theorem]{Example}

\theoremstyle{definition}
\newtheorem{subsec}[theorem]{\S\hspace{-0.03in}}

\def\t{\textrm}
\def\ds{\displaystyle}

\def\bs{\boldsymbol}
\def\bb{\mathbb}
\def\mb{\mathbf}
\def\mc{\mathcal}
\def\mf{\mathfrak}
\def\ms{\mathscr}
\def\beq{\begin{eqnarray}}
\def\eeq{\end{eqnarray}}
\def\beqq{\begin{eqnarray*}}
\def\eeqq{\end{eqnarray*}}

\def\vs{\vspace{0.01in}}
\def\vss{\vspace{0.03in}}
\def\ph{\phantom}
\def\d{\partial}
\def\T{\mc{T}}

\def\wt{\widetilde}
\def\Tr{\mathrm{Tr}\,}
\def\tnabla{\nabla^t}
\def\CDO{\mc{D}^{\mathrm{ch}}}
\def\ass{\Gamma^{\mathrm{ch}}}
\def\Killing{\lambda_{\t{\sf ad}}}
\def\n{\mf{n}}
\def\m{\mf{m}}
\def\p{\mf{p}}

\begin{document}

\begin{center}
{\large\bf CHIRAL DIFFERENTIAL OPERATORS:} \\
\vspace{0.05in}
{\large\bf FORMAL LOOP GROUP ACTIONS \& ASSOCIATED MODULES} \\
\vspace{0.2in}
POKMAN CHEUNG \\
\end{center}
\vspace{0.05in}

\begin{abstract}
Chiral differential operators (CDOs) are closely related to
string geometry and the quantum theory of $2$-dimensional
$\sigma$-models.
This paper investigates two topics about CDOs on smooth
manifolds.
In the first half, we study how a Lie group action on a smooth
manifold can be lifted to a ``formal loop group action'' on
an algebra of CDOs;
this turns out to be a condition on the equivariant first
Pontrjagin class.
The case of a principal bundle receives particular attention
and gives rise to a type of vertex algebras of great interest.
In the second half, we introduce a construction of modules over
CDOs using the said ``formal loop group actions'' and
semi-infinite cohomology.
Intuitively, these modules should have a geometric meaning
in terms of ``formal loop spaces''.
The first example we study leads to a new conceptual
construction of an arbitrary algebra of CDOs.
The other example, called the spinor module, may be useful
for a geometric theory of the Witten genus.
\end{abstract}

\section{Introduction}

The algebra of \emph{chiral differential operators (CDOs)}
on an affine space $\bb{A}^d$ is an elementary conformal
vertex algebra, and it has a holomorphic and a smooth
version as well.
The global construction of CDOs, first given by Gorbounov,
Malikov and Schechtman, is closely related to string geometry
and the quantum theory of $\sigma$-models.
In particular, the obstructions for a complex manifold $M$ to
admit a sheaf of holomorphic CDOs $\CDO_M$ are certain
refinements of $c_1(M)$ and $c_2(M)$;
and if $M$ is compact, then the genus-$1$ partition function
of the conformal vertex superalgebra $H^*(M;\CDO_M)$ is up
to a factor the Witten genus of $M$.
\cite{GMS1}
In the algebraic and holomorphic settings, Kapranov and
Vasserot have given a geometric interpretation of CDOs using
the notion of \emph{formal loop spaces} they introduced
\cite{KV1,KV};
see also \cite{BD}.
However, in the smooth setting, such an interpretation seems
to be still missing.
From the point of view of field theories, holomorphic CDOs are
known to physicists as a description of the large-volume limit
of the half-twisted $\sigma$-models \cite{Kapustin,Witten.CDO,Tan},
and they are also closely related to Costello's holomorphic
Chern-Simons theory. \cite{Costello1,Costello2}

This paper concerns CDOs on \emph{smooth} manifolds.
In the first half, we study how a Lie group action on a manifold
can be lifted to a ``formal loop group action'' on an algebra of
CDOs.
It turns out that the existence of such actions is a condition
on the equivariant first Pontrjagin class, and they are
parametrised by certain Cartan cochains.
The case of a principal bundle receives particular attention and
gives rise to a type of vertex algebras of great interest.
The study of ``formal loop group actions'' on CDOs is not
entirely new.
For example, it is well-known that the left and right
multiplications of a simple Lie group on itself can be lifted
to a commuting pair of affine Lie algebra actions on any algebra
of CDOs.
\cite{GMS4}
Also, the equivariant chiral de Rham complex of Lian and Linshaw
may be viewed as a particular differential graded version of
our construction.
\cite{LL}

In the second half of the paper, we introduce a construction of
modules over CDOs using the above ``formal loop group actions''
and semi-infinite cohomology.
This construction is quite analogous to that of associated vector
bundles with connection, and includes as a special case
a construction of Wakimoto modules.
\cite{FF,Voronov.Wakimoto}
The first example we study leads to a new conceptual description
of an arbitrary algebra of CDOs.
For some homogeneous spaces, this description of CDOs has 
been given before.
\cite{GMS4}
Another example of our construction, called the spinor module,
may be useful for understanding the geometric meaning of the
Witten genus for smooth string manifolds.

The following is a more detailed overview of the paper.

\begin{subsec}
{\bf Overview of the paper.}
\S\ref{sec.review} recalls the definition of an algebra of
CDOs on a smooth manifold (Definition \ref{defn.CDO}),
as well as its construction and classification (Theorems
\ref{thm.globalCDO} \& \ref{thm.globalCDO.iso}).
The explicit generators-and-relations description should be
compared to the more conceptual one in \S\ref{CDOP.CDOM}.

Let $G$ be a compact connected Lie group,
$\mf{g}$ its Lie algebra and
$\lambda$ an invariant symmetric bilinear form on $\mf{g}$.
Recall that $\lambda$ determines a centrally extended loop
algebra $\hat{\mf{g}}_\lambda$ (\S\ref{sec.Gmfld}) and also
a vertex algebra $V_\lambda(\mf{g})$ (Example \ref{VAoid.Lie}).
If $\mf{g}$ is simple and $\lambda$ is $k$ times the
normalized Killing form, then we will write the vertex algebra
more traditionally as $V_k(\mf{g})$.
Let us introduce some terminology:
a \emph{formal loop group action} of level $\lambda$ on
a vector space $W$ is a $\hat{\mf{g}}_\lambda$-action on $W$
whose restriction to $\mf{g}\subset\hat{\mf{g}}_\lambda$
integrates into a $G$-action;
and if $W$ is a vertex algebra then we also ask $\mf{g}$ to
act by inner derivations.
(In the main text this is called
an \emph{inner $(\hat{\mf{g}}_\lambda,G)$-action};
see Definition \ref{formalLG.action}.)
Notice that in the latter case the action is induced by a map of
vertex algebras $V_\lambda(\mf{g})\rightarrow W$.

Suppose $P$ is a smooth manifold with a smooth $G$-action
and $\CDO(P)$ is an algebra of CDOs on $P$.
Conjecturally, there should be an interpretation of $\CDO(P)$
in terms of the ``formal loop space of $P$''.
This motivates the main result in \S\ref{sec.LGaction}:

\vspace{0.05in}
\noindent
{\bf Theorem \ref{thm.formalLG.action}.}
{\it
The $G$-action on $P$ lifts to a formal loop group action on
$\CDO(P)$ of level $\lambda$ if and only if
\begin{align*}
8\pi^2p_1(P)_G=\lambda(P)
\end{align*}
where $p_1(P)_G$ is the equivariant first Pontrjagin class
and $\lambda(P)$ is the image of $\lambda$ under the
characteristic map $H^4(BG)\rightarrow H^4_G(M)$.
Moreover, $\CDO(P)$ has a conformal vector whose induced
Virasoro algebra action intertwines with any formal loop group
action as described.
}

\vspace{0.05in}
\noindent
The key ideas behind this result are:
the use of the Cartan model for equivariant de Rham cohomology
(recalled in \S\ref{sec.Cartan});
and the observation that the $G$-action on $P$ and the vertex
algebra structure of $\CDO(P)$ together determine a Cartan
cocycle for $p_1(P)_G$ (Lemma \ref{lemma.Cartan.closed}).
In fact, there is a more refined statement describing
a bijection between formal loop group actions and certain
Cartan cochains.

In the example where $G$ is simple and $P=G$ with the
$G\times G$-action by left and right multiplications, our
construction recovers a well-known result.
Namely, for any $k\in\bb{C}$, there is an algebra of CDOs
$\CDO_k(G)$ together with an embedding of vertex algebras
$V_{-k-h^\vee}(\mf{g})\otimes V_{k-h^\vee}(\mf{g})
\hookrightarrow\CDO_k(G)$, where $h^\vee$ is the dual Coxeter
number (Example \ref{CDOG} \& Proposition
\ref{prop.CDOG.commute}).
\cite{GMS4}
Moreover, for $k\notin\bb{Q}$, we give a new proof of
a ``chiral Peter-Weyl Theorem'' concerning the structure of
$\CDO_k(G)$ as a module over
$V_{-k-h^\vee}(\mf{g})\otimes V_{k-h^\vee}(\mf{g})$
(Proposition \ref{prop.PeterWeyl}).
\cite{FS.PeterWeyl}

From now on, $P$ is the total space of a principal $G$-bundle
$\pi:P\rightarrow M$ and $\CDO(P)$ is an algebra of CDOs
equipped with a formal loop group action
$V_\lambda(\mf{g})\hookrightarrow\CDO(P)$.
\S\ref{sec.CDOP} is a more detailed study of the structure of
$\CDO(P)$.
First we consider the centralizer subalgebra
\begin{align*}
\CDO(P)^{\hat{\mf{g}}}=C\big(\CDO(P),V_\lambda(\mf{g})\big)
\end{align*}
i.e.~the subalgebra invariant under the formal loop group action.
This is different from an algebra of CDOs on $P/G=M$.
In particular, part of the structure of $\CDO(P)^{\hat{\mf{g}}}$
is the Lie algebroid $(C^\infty(M),\T(P)^G)$ (\S\ref{sec.invCDO}),
while the corresponding part of the structure of an algebra of
CDOs on $M$ is the Lie algebroid $(C^\infty(M),\T(M))$.
For this reason $\CDO(P)^{\hat{\mf{g}}}$ has more interesting
modules (also see below).

In the case $\pi:P\rightarrow M$ is a principal frame bundle of
$TM$, we find a better description of $\CDO(P)$ such that (among
other things) the formal loop group action becomes manifest.
This occupies the second half of \S\ref{sec.CDOP} and
leads to the following result:

\vspace{0.05in}
\noindent
{\bf Theorem \ref{thm.CDOP}.}
{\it
Suppose we are given a principal $G$-bundle $\pi:P\rightarrow M$,
a representation $\rho:G\rightarrow SO(\bb{R}^d)$ and
an isomorphism $P\times_\rho\bb{R}^d\cong TM$;
a connection $1$-form $\Theta$ of $\pi$;
and a basic $3$-form $H$ on $P$ satisfying
$dH=(\lambda+\Killing+\lambda_\rho)(\Omega\wedge\Omega)$, where
$\Omega=d\Theta+\frac{1}{2}[\Theta\wedge\Theta]$ is the curvature
$2$-form (see also \S\ref{conventions} for notations).
These data determine an algebra of CDOs $\CDO_{\Theta,H}(P)$
with a formal loop group action of level $\lambda$.
For more details see the main text.
}

\vspace{0.05in}
\noindent
Even though the vertex algebras defined as above are particular
examples of algebras of CDOs, they can in fact be used to
recover \emph{all} algebras of CDOs as well as construct other
interesting objects (see below).
Also, they have an arguably more appealing
generators-and-relations description than an arbitrary algebra
of CDOs.

\S\ref{sec.ass} introduces the following construction:
given a formal loop group action of level $-\lambda-\Killing$
on a vector space $W$, one can define a module over the
invariant subalgebra $\CDO(P)^{\hat{\mf{g}}}$ by
\begin{align} \label{ass.module}
\ass(\pi,W)
:=H^{\frac{\infty}{2}+0}\big(\hat{\mf{g}}_{-\Killing}\,,\,
  \CDO(P)\otimes W\big)
\end{align}
(Definition \ref{defn.ass} \& Lemma \ref{lemma.ass}).
Here, $H^{\frac{\infty}{2}+*}(\hat{\mf{g}}_{-\Killing},-)$
is the \emph{semi-infinite (or BRST) cohomology} of the loop
algebra of $\mf{g}$, which is defined only at level $-\Killing$
(recalled in \S\ref{sec.Feigin}--\S\ref{sec.semiinf}).
This construction is analogous to that of associated
vector bundles.
In fact, the author believes there is an interpretation of
(\ref{ass.module}) in terms of a vector bundle over the
``formal loops of $M$'' equipped with some ``connection''.
The example where $\pi$ is the projection from a simple Lie
group to its flag manifold has been studied before as
a construction of Wakimoto modules.
\cite{FF,Voronov.Wakimoto}
Notice that if $W$ is a vertex algebra, then so is
(\ref{ass.module}).

For our first example of (\ref{ass.module}), $\pi$ is any
principal frame bundle of $TM$ and $W=\bb{C}$ (hence
$\lambda=-\Killing$).
By Theorem \ref{thm.CDOP}, the construction of
$\ass(\pi,\bb{C})$ requires the choice of some
$H\in\pi^*\Omega^3(M)$ that satisfies
$dH=\Tr\rho(\Omega)\wedge\rho(\Omega)$.
The main result in \S\ref{CDOP.CDOM} is a new description
of CDOs as alluded to earlier:

\vspace{0.05in}
\noindent
{\bf Theorem \ref{thm.CDO.semiinf}.}
{\it
$\ass(\pi,\bb{C})$ is an algebra of CDOs on $M$.
Moreover, up to isomorphism, every algebra of CDOs on $M$
arises in this fashion.
}

\vspace{0.05in}
\noindent
The proof consists of two parts:
first we identify the two lowest weights of $\ass(\pi,\bb{C})$
(\S\ref{sec.ass.C.0} \& \S\ref{sec.ass.C.1})
and work out their structure
(Proposition \ref{prop.ass.C.VAoid});
then we use a certain property of its conformal vector
to deduce that $\ass(\pi,\bb{C})$ is completely determined
by its two lowest weights
(Proposition \ref{prop.ass.C.conformal} \&
Corollary \ref{cor.ass.C.freeVA}).
For certain homogeneous spaces, this description of the algebras
of CDOs has been given before.
\cite{GMS4}
There is also an extension of the result to supermanifolds,
including a special case that recovers the chiral de Rham complex
(\S\ref{CDO.cs}).

For our second example of (\ref{ass.module}), $G=\t{Spin}_{2d'}$,
$\pi$ is a spin frame bundle of $M$ and $W=S$ is the spinor
representation of $\widehat{\mf{so}}_{2d'}$.
By Theorem \ref{thm.CDOP}, the construction of the
\emph{spinor module} $\ass(\pi,S)$ requires the choice of some
$H\in\pi^*\Omega^3(M)$ that satisfies
$dH=\frac{1}{2}\Tr(\Omega\wedge\Omega)$ (\S\ref{sec.ass.S}).
\S\ref{sec.FL.spinor} consists mainly of an analysis of the
spinor module that parallels the one of $\ass(\pi,\bb{C})$,
resulting in a more explicit description in terms of generating
data and relations (Theorem \ref{thm.ass.S}).
In fact, we regard the work in \S\ref{sec.FL.spinor} largely as
preparation for further study.
It is hoped that the spinor module will serve as an ingredient
in a geometric theory of the Witten genus.

The appendix reviews the notion of vertex algebroids
and their relations with vertex algebras.
Even though vertex algebroids have a rather complicated
definition, they are useful for handling the vertex algebras in
this paper.
\end{subsec}

\begin{subsec} \label{conventions}
{\bf Notations and conventions.}
In this paper, every vertex algebra $V$ is graded by
nonnegative integers which we call \emph{weights};
its component of weight $k$ is denoted by $V_k$ and
its weight operator by $L_0$, i.e. $L_0|_{V_k}=k$.
For any $u\in V$, we always write the Fourier modes of
its vertex operator as $u_k$, $k\in\bb{Z}$, such that
$u_k$ has weight $-k$.
For any conformal vector $\nu\in V$ we consider, $\nu_0=L_0$.

Given a smooth manifold $M$, we write
$C^\infty(M)$, $\T(M)$, $\Omega^n(M)$ for its spaces of
smooth $\bb{C}$-valued functions, vector fields and $n$-forms.
Also, all ordinary cohomology groups and their equivariant
versions have complex coefficients.
Given a Lie algebra $\mf{g}$ and a finite-dimensional
representation $\rho:\mf{g}\rightarrow\mf{gl}(V)$, we write
$\lambda_\rho$ for the invariant symmetric bilinear form on
$\mf{g}$ given by $\lambda_\rho(A,B)=\Tr\rho(A)\rho(B)$.
In particular, $\Killing$ is the Killing form.
Square brackets $[\;]$ are used for supercommutators between
operators of any parities, while curly brackets $\{\;\}$
are reserved for a different use (see \S\ref{VAoid}).
Repeated indices are always implicitly summed over all possible
values, unless a specific range is indicated.
\end{subsec}

\begin{subsec}
{\bf Acknowledgements.}
The author is currently supported by the ESPRC grant EP/H040692/1.
He would like to thank Matthew Ando for a suggestion that
initially motivated Theorem \ref{thm.formalLG.action}, and Dennis
Gaitsgory for explaining some aspects of semi-infinite cohomology.
He would also like to thank a reviewer for pointing out some
related works in the literature.
\end{subsec}

\newpage

\section{CDOs on Smooth Manifolds}
\label{sec.review}

A sheaf of CDOs is a sheaf of vertex algebras locally modelled
on an elementary vertex algebra known as a $\beta\gamma$-system,
and it provides an approximate mathematical formulation of the
quantum theory of $2$-dimensional $\sigma$-models.
\cite{GMS1,Kapustin,Witten.CDO} 
This section reviews the construction and classification of
sheaves of CDOs on a smooth manifold.
For the definition of a vertex algebra, see \cite{Kac,FB-Z}.

\begin{subsec} \label{algCDO} 
{\bf The algebra of CDOs on $\bb{A}^d$.}
Let $d$ be a positive integer.
Define a unital associative algebra $\mc{U}$ with
the following generators and relations
\begin{align} \label{Weyl}
b^i_n,\;a_{i,n},\;n\in\bb{Z},\;i=1,\ldots,d,\qquad
[a_{i,n},b^j_m]=\delta^j_i\delta_{n,-m},\qquad
[b^i_n,b^j_m]=0=[a_{i,n},a_{j,m}]\,.
\end{align}
The commutative subalgebra $\mc{U}_+$ generated
by $\{b^i_n\}_{n>0}$ and $\{a_{i,n}\}_{n\geq 0}$
has a trivial representation $\bb{C}$.
The induced $\mc{U}$-module
\begin{align*}
\CDO(\bb{A}^d):=\mc{U}\otimes_{\mc{U}_+}\bb{C}
\end{align*}
has the structure of a vertex algebra.
The vacuum is $\mathbf 1=1\otimes 1$.
The infinitesimal translation operator $T$ and weight
operator $L_0$ are determined by
\begin{align*}
\begin{array}{lll}
T\mathbf 1=0, &
[T,b^i_n]=(1-n)b^i_{n-1}, \ph{aa} &
[T,a_{i,n}]=-na_{i,n-1} \vss \\
L_0\mathbf 1=0, \ph{aa} &
[L_0,b^i_n]=-nb^i_n, &
[L_0,a_{i,n}]=-na_{i,n}
\end{array}
\end{align*}
The fields (or vertex operators) of $b^i_0\mathbf 1\in\CDO(\bb{A}^d)_0$
and $a_{i,-1}\mathbf 1\in\CDO(\bb{A}^d)_1$ are respectively
\begin{align*}
\begin{array}{l}
\sum_n b^i_n z^{-n},\qquad \sum_n a_{i,n}z^{-n-1}
\end{array}
\end{align*}
which determine the fields of other elements by
the Reconstruction Theorem \cite{FB-Z}.
This vertex algebra has a family of conformal vectors
of central charge $2d$, namely
\footnote{
This vertex algebra also has other conformal vectors that define
the weights differently.
\cite{Kac}
}
\begin{align*}
\qquad\qquad\qquad
a_{i,-1}b^i_{-1}\mathbf 1+T^2 f,\qquad
f\in\CDO(\bb{A}^d)_0=\bb{C}[b^1_0,\ldots,b^d_0]\cdot\mb{1}.
\end{align*}

The vertex algebra $\CDO(\bb{A}^d)$ is freely generated by
its associated vertex algebroid (see \S\ref{VA.VAoid} and
\S\ref{VAoid.VA}).
To describe the latter, consider the affine space
$\bb{A}^d=\t{Spec}\,\bb{C}[b^1,\cdots,b^d]$ and identify
the functions, $1$-forms and vector fields on $\bb{A}^d$
with the following subquotients of $\CDO(\bb{A}^d)$: \vss \\
\indent $\cdot$\;
$\mc{O}(\bb{A}^d)=\CDO(\bb{A}^d)_0$ via $b^i=b^i_0\bs{1}$, 
$b^i b^j=b^i_0 b^j_0\bs{1}$, etc. \vss \\
\indent $\cdot$\;
$\Omega^1(\bb{A}^d)\subset\CDO(\bb{A}^d)_1$ via 
$db^i=b^i_{-1}\bs{1}$ \vss \\
\indent $\cdot$\;
$\T(\bb{A}^d)=\CDO(\bb{A}^d)_1/\Omega^1(\bb{A}^d)$ via
$\d_i=\d/\d b^i=$ coset of $a_{i,-1}\bs{1}$ \vss \\
Under these identifications, the vertex algebroid associated to
$\CDO(\bb{A}^d)$ is of the form 
\begin{align*}
\big(\mc{O}(\bb{A}^d),\Omega^1(\bb{A}^d),\T(\bb{A}^d),
  \bullet,\{\;\},\{\;\}_{\Omega}\big).
\end{align*}
The extended Lie algebroid structure consists of
the usual differential on functions,
Lie bracket on vector fields, 
Lie derivations by vector fields on functions and $1$-forms,
and pairing between $1$-forms and vector fields.
Using the splitting
\begin{align} \label{algCDO.splitting}
s:\T(\bb{A}^d)\rightarrow\CDO(\bb{A}^d)_1,\qquad
X=X^i\d_i\;\mapsto\;a_{i,-1}X^i
\end{align} 
the rest of the vertex algebroid structure, according to
(\ref{VA.VAoid.456}), is given by
\begin{align} \label{CDO.VAoid456} 
X\bullet f=(\d_j X^i)(\d_i f)db^j,\quad
\{X,Y\}=-(\d_j X^i)(\d_i Y^j),\quad
\{X,Y\}_\Omega=-(\d_k\d_j X^i)(\d_i Y^j)db^k 
\end{align}
These expressions do not seem to have any obvious
global meaning;
however, see Theorem \ref{thm.globalCDO}.
\end{subsec}

\begin{subsec} \label{sec.CDO.Rd}
{\bf The sheaf of CDOs on $\bb{R}^d$.}
Now regard $b^1,\ldots,b^d$ as the standard coordinates of
$\bb{R}^d$.
The smooth functions, $1$-forms and vector fields on any open
set $W\subset\bb{R}^d$ form an extended Lie algebroid just as
in \S\ref{algCDO}, and the expressions in (\ref{CDO.VAoid456})
again define a vertex algebroid
\begin{align*}
\big(C^\infty(W),\Omega^1(W),\T(W),
  \bullet,\{\;\},\{\;\}_\Omega\big).
\end{align*}
The vertex algebra it freely generates
(see \S\ref{VAoid.VA}) will be denoted by $\CDO(W)$.
This vertex algebra also has a family of conformal
vectors of central charge $2d$, namely
\begin{align} \label{localCDO.conformal}
\d_{i,-1}db^i+\frac{1}{2}\,T\omega,\qquad
\omega\in\Omega^1(W),\;d\omega=0.
\end{align}
For any inclusion of open sets $W'\subset W$, there is
an obvious restriction map $\CDO(W)\rightarrow\CDO(W')$.
This defines a sheaf of conformal vertex algebras $\CDO$ on
$\bb{R}^d$.
\end{subsec}

\begin{defn} \label{defn.CDO}
A {\bf sheaf of CDOs} on a smooth $d$-dimensional manifold $M$
is a sheaf of vertex algebras $\mc{V}$ with the following
properties:
\vspace{-0.05in}
\begin{itemize}
\item[$\cdot$]
its weight-zero component is $\mc{V}_0=C^\infty_M$, and
\vspace{-0.08in}
\item[$\cdot$]
each point of $M$ has a neighborhood $U$ such that $(U,\mc{V}|_U)$
is isomorphic to $(W,\CDO|_W)$ for some open set $W\subset\bb{R}^d$.
\vspace{-0.05in}
\end{itemize}
A {\bf conformal structure} on $\mc{V}$ is an element
$\nu\in\mc{V}(M)$ such that each of the isomorphisms postulated
above takes $\nu|_U\in\mc{V}(U)$ to one of the conformal vectors
(\ref{localCDO.conformal}) of $\CDO(W)$.
\end{defn}

In order to state the results on the construction
and classification of sheaves of CDOs, let us
introduce a notation that will also appear
often in the sequel.

\begin{defn} \label{tnabla}
Let $M$ be a smooth manifold.
Given a connection $\nabla$ on $TM$ and $X\in\T(M)$, define
an operator $\tnabla X\in \Gamma(\t{End}\,TM)$ by
\begin{align*}
(\tnabla X)(Y):=\nabla_X Y-[X,Y],\qquad Y\in\T(M).
\end{align*}
Notice that if $\nabla$ is torsion-free, then
$\tnabla X=\nabla X$.
\end{defn}

\begin{theorem} \label{thm.globalCDO}
\cite{myCDO}
Let $M$ be a smooth $d$-dimensional manifold.

(a) Given a connection $\nabla$ on $TM$ with curvature $R$ and
any $H\in\Omega^3(M)$ satisfying $dH=\Tr(R\wedge R)$, we can
define a sheaf of vertex algebroids
$\big(C^\infty_M,\Omega^1_M,\T_M,\bullet,\{\;\},\{\;\}_\Omega\big)$
on $M$ by
\begin{align}
\nonumber 
X\bullet f\;\;\; &:= (\nabla X)f \\
\label{CDO.VAoid}
\{X,Y\}\;\; &:= -\Tr(\tnabla X\cdot\tnabla Y) \\
\nonumber
\{X,Y\}_\Omega &:=
  \Tr\bigg(-\nabla(\tnabla X)\cdot\tnabla Y
    +\tnabla X\cdot\iota_Y R
    -\iota_X R\cdot\tnabla Y\bigg)
  +\frac{1}{2}\iota_X\iota_Y H
\end{align}
The sheaf of vertex algebras it freely generates (see
\S\ref{VAoid.VA}) is a sheaf of CDOs on $M$, which will be
denoted by $\CDO_{M,\nabla,H}$.
Up to isomorphism, every sheaf of CDOs on $M$ is
of this form.

(b) There is a one-to-one correspondence between
$\omega\in\Omega^1(M)$ satisfying $d\omega=\Tr R$ and the conformal
structures on $\CDO_{M,\nabla,H}$.
For any such $\omega$, the corresponding conformal structure, which
will be denoted by $\nu^\omega$, has the following local expression:
\begin{align} \label{CDO.conformal}
\nu^\omega|_U
=\sigma_{i,-1}\sigma^i
+\frac{1}{2}\Tr\big(\Gamma^\sigma_{-1}\Gamma^\sigma
  -\Gamma^\sigma_{-2}\mb{1}\big)
+\sigma^i([\sigma_j,\sigma_k])
  \sigma^k_{-1}(\Gamma^\sigma)^j_{\ph{i}i}
+\frac{1}{2}\omega_{-2}\mb{1}
\end{align}
where $U\subset M$ is any sufficiently small open set;
$\sigma=(\sigma_1,\ldots,\sigma_d)$ is any basis of $\T(U)$ over
$C^\infty(U)$;
$(\sigma^1,\ldots,\sigma^d)$ is the dual basis of $\Omega^1(U)$;
and $\Gamma^\sigma\in\Omega^1(U)\otimes\mf{gl}_d$ is the connection
$1$-form of $\nabla$ with respect to $\sigma$,
i.e.~$\nabla\sigma_i=(\Gamma^\sigma)^j_{\ph{i}i}\otimes\sigma_j$.
Also, $\nu^\omega$ has central charge $2d$ and the property that
\begin{align} \label{CDO.L1}
\qquad
\nu^\omega_1\alpha=0\;\t{ for }\alpha\in\Omega^1(M),\qquad
\nu^\omega_1 X=\Tr\tnabla X-\omega(X)\;\t{ for }X\in\T(M).
\qquad\qedsymbol
\end{align}
\end{theorem}

{\it Remarks.}
(i) By Theorem \ref{thm.globalCDO}, a smooth manifold
$M$ admits sheaves of CDOs if and only if $p_1(M)$ is
trivial in de Rham cohomology, while conformal structures
always exist.
For example, if $\nabla$ is orthogonal with respect to
a Riemannian metric, then $\Tr R=0$ and a conformal
structure can be defined using, say, $\omega=0$.
However, this result generalizes to supermanifolds, in which case
the obstruction to conformal structures may well be nontrivial.
\cite{myCDO}

(ii) In \cite{myCDO}, we only obtained a local expression of
$\nu^\omega$ in terms of local coordinate vector fields,
but that implies the more general expression in
(\ref{CDO.conformal}) by a straightforward calculation.

(iii) In the original work \cite{GMS2} as well as in
\cite{myCDO}, sheaves of CDOs and conformal structures
were constructed by gluing local data.
For smooth manifolds, the result is the above description in
terms of generators and relations (or generating fields and
OPEs), but the expressions in (\ref{CDO.VAoid}) and
(\ref{CDO.conformal}) do not look very inspiring.
In \S\ref{CDOP.CDOM}, we will obtain a more conceptual
description of CDOs using semi-infinite cohomology.

\begin{theorem} \label{thm.globalCDO.iso}
\cite{myCDO}
Let $\CDO_{M,\nabla,H}$ and $\CDO_{M,\nabla,H'}$ be sheaves of CDOs
on a smooth manifold $M$ constructed as in Theorem
\ref{thm.globalCDO}a;
denote by $\bullet$, $\{\;\}$, $\{\;\}_\Omega$
(resp.~$\{\;\}'_\Omega$) the structure maps determined by
$\nabla$ and $H$ (resp.~$H'$) as in (\ref{CDO.VAoid}).

(a) There is a one-to-one correspondence between
$\beta\in\Omega^2(M)$ satisfying $d\beta=H'-H$ and isomorphisms
$\CDO_{M,\nabla,H}\stackrel{\sim}{\rightarrow}\CDO_{M,\nabla,H'}$
that restricts to the identity on $C^\infty_M$.
For any such $\beta$, the corresponding isomorphism is induced by
an isomorphism of sheaves of vertex algebroids
(see \S\ref{VAoid.VA.mor})
\begin{align*}
(\mathrm{id},\Delta_\beta):
\big(C^\infty_M,\Omega^1_M,\T_M,\bullet,\{\;\},\{\;\}_\Omega\big)
\rightarrow
\big(C^\infty_M,\Omega^1_M,\T_M,\bullet,\{\;\},\{\;\}'_\Omega\big)
\end{align*}
where $\Delta_\beta:\T_M\rightarrow\Omega^1_M$ is given by
$\Delta_\beta(X)=\frac{1}{2}\iota_X\beta$.

(b) Every isomorphism described above respects the correspondence
in Theorem \ref{thm.globalCDO}b.
$\qedsymbol$
\end{theorem}

{\it Remarks.}
(i) By Theorems \ref{thm.globalCDO}a and \ref{thm.globalCDO.iso}a,
if $M$ admits sheaves of CDOs, their isomorphism classes form
an $H^3(M)$-torsor.
(ii) Since every sheaf of CDOs $\CDO_{M,\nabla,H}$ is fine, for
most purposes it suffices to (and we will) work with the vertex
algebra of global sections $\CDO_{\nabla,H}(M)$, which will be
called an {\bf algebra of CDOs} on $M$.
The reader should keep in mind that by construction
\begin{align} \label{CDO.wt01}
\CDO_{\nabla,H}(M)_0=C^\infty(M),\qquad
\CDO_{\nabla,H}(M)_1=\Omega^1(M)\oplus\T(M).
\end{align}
For a description of the higher weights, see
\S\ref{freeVA.PBW}.

\begin{lemma} \label{lemma.1form.0mode}
Consider an algebra of CDOs $\CDO_{\nabla,H}(M)$ on a smooth
manifold $M$.
For any $\alpha\in\Omega^1(M)$ and $X\in\T(M)$ viewed as
elements of $\CDO_{\nabla,H}(M)$, we have
$\alpha_0 X=-\iota_X d\alpha$.
Moreover, $\alpha_0\equiv 0$ on $\CDO_{\nabla,H}(M)$ if
and only if $d\alpha=0$.
\end{lemma}

\begin{proof}
By definition of $\CDO_{\nabla,H}(M)$ and (\ref{freeVA.comm}),
we have $\alpha_0 f=0$ for $f\in C^\infty(M)$ and
\begin{align*}
\alpha_0 X
=-[X_{-1},\alpha_0]\mb{1}
=-L_X\alpha+d\alpha(X)
=-\iota_X d\alpha.
\end{align*}
Since $\alpha_0$ is a derivation, it is trivial on the entire
vertex algebra $\CDO_{\nabla,H}(M)$ if and only if it is trivial
on the generating subspaces $C^\infty(M)$ and $\T(M)$.
\end{proof}

\newpage
\setcounter{equation}{0}
\section{Formal Loop Group Actions on CDOs}
\label{sec.LGaction}

In this section, we study how a Lie group action on
a manifold can be lifted to a ``formal loop group action''
on an algebra of CDOs.
This turns out to be a condition on the equivariant first
Pontrjagin class of the manifold.

\begin{subsec} \label{sec.Gmfld}
{\bf Setting:~manifold with a Lie group action.}
Throughout this section, let $G$ be a compact connected
Lie group, $\mf{g}$ its Lie algebra, and $\lambda$
an invariant symmetric bilinear form on $\mf{g}$.
Recall the loop algebra
$L\mf{g}=\mf{g}\otimes\bb{C}[t,t^{-1}]$:
writing $A\otimes t^n$ as $A_n$, the Lie bracket is
given by $[A_n,B_m]=[A,B]_{n+m}$.
Also recall that $\lambda$ determines a central extension
$\hat{\mf{g}}_\lambda=L\mf{g}\oplus\bb{C}$ with
\begin{align*}
[A_n,B_m]=[A,B]_{n+m}+n\lambda(A,B)\delta_{n+m,0},\quad
A,B\in\mf{g},\;\,n,m\in\bb{Z}.
\end{align*}
\footnote{
In this paper we adopt the following habit:
the kernel of $\hat{\mf{g}}_\lambda\rightarrow L\mf{g}$
is always identified with $\bb{C}$ and, whenever
$\hat{\mf{g}}_\lambda$ acts on some vector space, $1\in\bb{C}$
always acts as identity.
}
Let $P$ be a smooth manifold with a smooth right $G$-action.
Later we will specialize to the case of a principal
bundle, but at the moment $P$ can be any right $G$-manifold.
The left $G$-action on $C^\infty(P)$ is completely determined
by the induced map of Lie algebras $\mf{g}\rightarrow\T(P)$.
The vector field generated by $A\in\mf{g}$ will be written
as $A^P\in\T(P)$.
\end{subsec}

\begin{subsec} \label{sec.Cartan}
{\bf The equivariant de Rham complex.}
Recall that $H^*_G(P)=H^*(EG\times_G P)$ can be computed
by the Cartan model $(\Omega^*_G(P),d_G)$.~\cite{GuiSte}
The graded algebra of Cartan cochains is given by
\begin{align*}
\Omega^*_G(P)=
  \bigoplus_{k\geq 0}\Omega^k_G(P),\qquad
\Omega^k_G(P)=
  \bigoplus_{2i+j=k}\big(\t{Sym}^i\mf{g}^\vee
    \otimes\Omega^j(P)\big)^G.
\end{align*}
Let us follow a usual convention:~regard any
$\xi\in\Omega^*_G(P)$ as a $G$-equivariant polynomial map
$\xi:\mf{g}\rightarrow\Omega^*(P)$ and write its value at
$A\in\mf{g}$ as $\xi_A$.
The Cartan differential then reads
\begin{align*}
(d_G\xi)_A=d\xi_A-\iota_{A^P}\xi_A.
\end{align*}
The characteristic map $H^*(BG)\rightarrow H^*_G(P)$ is
represented by the inclusion
$(\t{Sym}^*\mf{g}^\vee)^G\hookrightarrow\Omega^*_G(P)$,
and the image of any $\eta\in(\t{Sym}^*\mf{g}^\vee)^G=H^*(BG)$
will be denoted by $\eta(P)\in H^*_G(P)$.
\end{subsec}

\begin{subsec} \label{sec.inv.data}
{\bf CDOs with a Lie group action.}
Choose a $G$-invariant connection $\nabla$ on $TP$, with
curvature tensor $R$.
This means for $A\in\mf{g}$ we have $L_{A^P}\nabla=0$,
or equivalently
\begin{align} \label{conn.Ginv}
\nabla(\tnabla A^P)=-\iota_{A^P}R
\end{align}
(see Definition \ref{tnabla}).
\footnote{
It follows from a simple calculation that for any $X\in\T(P)$
we have $L_X\nabla=\nabla(\tnabla X)+\iota_X R$.
}
Assume that $p_1(P)=0$ and choose some $H\in\Omega^3(P)^G$
such that $dH=\Tr(R\wedge R)$.
By Theorem \ref{thm.globalCDO}a, $\nabla$ and $H$
determine an algebra of CDOs $\CDO_{\nabla,H}(P)$, which
is freely generated by a vertex algebroid
(see also \S\ref{VAoid.VA})
\begin{align*}
\big(C^\infty(P),\Omega^1(P),\T(P),
  \bullet,\{\;\},\{\;\}_\Omega\big);
\end{align*}
when there is no risk of confusion, we simply write
$\CDO(P)$.
Clearly, the $G$-invariance of $\nabla$ and $H$ implies the
$G$-equivariance of the structure maps $\bullet$, $\{\;\}$,
$\{\;\}_\Omega$, so that the
$G$-action on $C^\infty(P)=\CDO(P)_0$ extends to a $G$-action
on $\CDO(P)$.

Without loss of generality, we may assume that $\Tr R=0$.
(For example, this is true if $\nabla$ is orthogonal
with respect to a Riemannian metric.)
Choose some $\omega\in\Omega^1(P)^G$ such that $d\omega=0$.
By Theorem \ref{thm.globalCDO}b, $\omega$ determines
a $G$-invariant conformal vector $\nu^\omega$ in
$\CDO(P)$.
The Fourier modes of the associated Virasoro field 
will be denoted by $L_n^\omega$, $n\in\bb{Z}$.
Soon we will make a more specific choice of $\omega$.
\end{subsec}

{\it Remark.}
The $G$-invariance of $\nabla$ and $H$, as well as that
of other geometric data to appear later, can always be
achieved by averaging over $G$ with respect to the
Haar measure.

\begin{defn} \label{formalLG.action}
An {\bf inner $(\hat{\mf{g}}_\lambda,G)$-action} on a vertex
algebra $V$ is a map of vertex algebras from $V_\lambda(\mf{g})$
such that the induced $\mf{g}$-action on $V$ integrates into
a $G$-action. 
An inner $(\hat{\mf{g}}_\lambda,G)$-action on a conformal vertex
algebra $V$ is {\bf primary} if the image of any
$A\in V_\lambda(\mf{g})_1$ is primary, or equivalently, if the
induced $\hat{\mf{g}}_\lambda$-action on $V$ is intertwined by
the Virasoro algebra action.
\end{defn}

{\it Remarks.}
(i) By definition of $V_\lambda(\mf{g})$ (see Example
\ref{VAoid.Lie}), any map $V_\lambda(\mf{g})\rightarrow V$
is determined by its component of weight one, i.e.~a linear map
$\mf{g}=V_\lambda(\mf{g})_1\rightarrow V_1$.
Taking the zeroth modes (resp.~Fourier modes) then yields a map of
Lie algebras from $\mf{g}$ to the inner derivations of $V$
(resp.~from $\hat{\mf{g}}_\lambda$ to the endomorphisms of $V$).
These are the induced actions mentioned above.
(ii) Now we can state the goal of this section more precisely:
find the condition under which the given $G$-action on
$\CDO(P)_0=C^\infty(P)$ extends to an inner
$(\hat{\mf{g}}_\lambda,G)$-action on $\CDO(P)$.

\begin{subsec} \label{CDO.Cartan}
{\bf Cartan cochains associated to CDOs.}
The $G$-action on $P$ and the vertex algebroid structure associated
to $\CDO(P)$ together determine two Cartan cochains of degree $4$,
namely
\begin{align} \label{Cartan4.VAoid}
\begin{array}{ll}
\chi^{2,2}\in\big(\mf{g}^\vee\otimes\Omega^2(P)\big)^G,
&\chi^{2,2}_A:=\{A^P,-\}_\Omega=
  \Tr(\tnabla A^P\cdot R)-\frac{1}{2}\iota_{A^P}H
\vss \\
\chi^{4,0}\in\big(\t{Sym}^2\mf{g}^\vee\otimes C^\infty(P)\big)^G,
\quad
&\chi^{4,0}_A:=\{A^P,A^P\}
  =-\Tr(\tnabla A^P\cdot\tnabla A^P)
\end{array}
\end{align}
Indeed, by (\ref{CDO.VAoid}) and (\ref{conn.Ginv}),
the operator $\{A^P,-\}_\Omega:\T(P)\rightarrow\Omega^1(P)$
may be viewed as the indicated $2$-form, and
the $G$-invariance of $\nabla$, $H$ implies the $G$-equivariance
of $\chi^{2,2}$, $\chi^{4,0}$.
The $G$-action on $P$ and the conformal vector $\nu^\omega$ of
$\CDO(P)$ together determine a Cartan cochain of degree $2$,
namely
\begin{align} \label{Cartan2.conformal}
\chi^{2,0}\in\big(\mf{g}^\vee\otimes C^\infty(P)\big)^G,\qquad
\chi^{2,0}_A:=L_1^\omega A^P=\Tr\tnabla A^P-\omega(A^P).
\end{align}
Indeed, the $G$-invariance of $\nabla$, $\omega$ implies the
$G$-equivariance of $\chi^{2,0}$.
\end{subsec}

\begin{lemma} \label{lemma.Cartan.closed}
(a) The Cartan cochain $2\chi^{2,2}+\chi^{4,0}$ is closed and
represents $8\pi^2 p_1(P)_G\in H^4_G(P)$.
(b) The Cartan cochain $\chi^{2,0}$ is exact.
In fact, it is trivial with a suitable choice of conformal
structure $\nu^\omega$.
\end{lemma}

\begin{proof}
(a) According to \cite{BGV}, the Cartan cochain
$A\mapsto\Tr(R-\tnabla A^P)^2$ is closed and represents
the class $-8\pi^2 p_1(P)_G$.
Then the claim follows from the calculation
\begin{align*}
\Tr(R-\tnabla A^P)^2
=dH-2\,\Tr(\tnabla A^P\cdot R)
  +\Tr(\tnabla A^P\cdot\tnabla A^P)
=(d_G H)_A-2\chi^{2,2}_A-\chi^{4,0}_A
\end{align*}
where we have used $dH=\Tr(R\wedge R)$ and
(\ref{Cartan4.VAoid}).

(b) According to \cite{BGV} again, the Cartan cochain
$A\mapsto\Tr(R-\tnabla A^P)=-\Tr\tnabla A^P$ is exact.
Also we have $(d_G\omega)_A=d\omega-\omega(A^P)=-\omega(A^P)$.
This proves the first part of the claim.
On the other hand, there exists some $\omega'\in\Omega^1(P)^G$
such that
\begin{align*}
-\Tr\tnabla A^P
=(d_G\omega')_A
=d\omega'-\omega'(A^P)
\end{align*}
which is equivalent to two equations:
firstly $d\omega'=0$, so that $\omega'$ also determines
a conformal vector $\nu^{\omega'}$;
and secondly $L_1^{\omega'}A^P=0$.
This proves the rest of the claim.
\end{proof}

{\it Remarks.}
(i) The closedness of $2\chi^{2,2}+\chi^{4,0}$ is
equivalent to a pair of equations
\begin{align} \label{Cartan4.closed}
\phantom{WWWW}
d\chi^{2,2}_A=0,\qquad
d\chi^{4,0}_A=2\iota_{A^P}\chi^{2,2}_A,\qquad
\t{for }A\in\mf{g}
\end{align}
which can also be verified directly from our
assumptions on $\nabla$ and $H$.
(ii) From now on we assume that the conformal vector
$\nu^\omega$ has been chosen such that $L_1^\omega A^P=0$
for $A\in\mf{g}$, and write the associated Virasoro
operators simply as $L_n$ for $n\in\bb{Z}$.

\vspace{0.08in}
{\it Preparation.}
The following sequence of results all concern lifting the map
of Lie algebras $\mf{g}\rightarrow\T(P)$ to a linear map
of the form
\begin{align} \label{lift.CDO1}
\mf{g}\,\rightarrow\,
  \CDO(P)_1=\T(P)\oplus\Omega^1(P),\qquad
A\mapsto A^P+h^{2,1}_A
\end{align}
where $h^{2,1}:\mf{g}\rightarrow\Omega^1(P)$ is
assumed to be $G$-equivariant.
In other words, $h^{2,1}$ is a Cartan cochain.

\begin{prop} \label{prop.lift.Lie}
The linear map (\ref{lift.CDO1}) gives rise to a map of
Lie algebras
\begin{align} \label{lift.Lie}
\mf{g}\,\rightarrow\,\mathrm{Der}\,\CDO(P),\qquad
A\mapsto\big(A^P+h^{2,1}_A\big)_0
\end{align}
if and only if the closed $2$-form $\chi^{2,2}_A-dh^{2,1}_A$
is $G$-invariant for any $A\in\mf{g}$.
\end{prop}

\begin{proof}
For $A,B\in\mf{g}$, consider the calculation
\begin{align*}
\big[\big(A^P+h^{2,1}_A\big)_0,\big(B^P+h^{2,1}_B\big)_0\big]
&=\Big([A^P,B^P]+\{A^P,B^P\}_\Omega+L_{A^P}h^{2,1}_B
  -L_{B^P}h^{2,1}_A\Big)_0 \\
&=\Big([A,B]^P+h^{2,1}_{[A,B]}\Big)_0
  +\Big(\iota_{B^P}\chi^{2,2}_A-L_{B^P}h^{2,1}_A\Big)_0 
\end{align*}
where we have used (\ref{freeVA.comm}), (\ref{Cartan4.VAoid})
and the $G$-equivariance of $h^{2,1}$.
By Lemma \ref{lemma.1form.0mode}, the second term is trivial
if and only if
\begin{align*}
0=d\big(\iota_{B^P}\chi^{2,2}_A-L_{B^P}h^{2,1}_A\big)
=L_{B^P}\big(\chi^{2,2}_A-dh^{2,1}_A\big)
\end{align*}
where we have used (\ref{Cartan4.closed}).
This proves our claim.
\end{proof}

The map of Lie algebras $\mf{g}\rightarrow\T(P)$ extends
in an obvious way to a map of extended Lie algebroids
$i:(\bb{C},0,\mf{g})\rightarrow(C^\infty(P),\Omega^1(P),\T(P))$.
Now we would like to extend it further to a map of vertex
algebroids and hence a map of vertex algebras
(see Example \ref{VAoid.Lie} and \S\ref{VAoid.VA.mor}).

\begin{prop} \label{prop.lift.VA}
The linear map (\ref{lift.CDO1}) determines a map of
vertex algebroids
\begin{align*}
(i,h^{2,1}):(\bb{C},0,\mf{g},0,\lambda,0)\rightarrow
  \big(C^\infty(P),\Omega^1(P),\T(P),
  \bullet,\{\;\},\{\;\}_\Omega\big),
\end{align*}
and hence a map of vertex algebras
$V_\lambda(\mf{g})\rightarrow\CDO(P)$, if and only if for
any $A\in\mf{g}$ the closed $2$-form
$\chi^{2,2}_A-dh^{2,1}_A$ is $G$-horizontal and
\begin{align} \label{lambda.Cartan40}
\chi^{4,0}_A+2h^{2,1}_A(A^P)
=\lambda(A,A).
\end{align}
\end{prop}

\begin{proof}
Let $A,B\in\mf{g}$.
According to Definition \ref{VAoid.mor}, the linear map
$h^{2,1}:\mf{g}\rightarrow\Omega^1(P)$ defines a map
between the vertex algebroids in question if and only if
\begin{align*}
\{A^P,B^P\}\;\;
&=\lambda(A,B)-h^{2,1}_A(B^P)-h^{2,1}_B(A^P) \\
\t{and}\quad
\{A^P,B^P\}_\Omega
&=-L_{A^P}h^{2,1}_B+L_{B^P}h^{2,1}_A-dh^{2,1}_A(B^P)
  +h^{2,1}_{[A,B]}
\end{align*}
By (\ref{Cartan4.VAoid}) the first equation is equivalent
to (\ref{lambda.Cartan40}).
By (\ref{Cartan4.VAoid}) again and the $G$-equivariance of
$h^{2,1}$, the second equation can be rewritten as
$\iota_{B^P}(\chi^{2,2}_A-dh^{2,1}_A)=0$.
This proves our claim.
In fact, the $G$-horizontality of $\chi^{2,2}_A-dh^{2,1}_A$
always implies that the left side of (\ref{lambda.Cartan40})
is locally constant, since
\begin{align*}
d\big(\chi^{4,0}_A+2\iota_{A^P}h^{2,1}_A\big)
=2\iota_{A^P}\chi^{2,2}_A-2\iota_{A^P}dh^{2,1}_A
  +2L_{A^P}h^{2,1}_A
=2\iota_{A^P}(\chi^{2,2}_A-dh^{2,1}_A)
\end{align*}
by virtue of (\ref{Cartan4.closed}) and the $G$-equivariance
of $h^{2,1}$.
Therefore if $P$ is connected, the horizontal condition
guarantees that (\ref{lambda.Cartan40}) must hold for some
$\lambda\in(\t{Sym}^2\mf{g}^\vee)^G$.
\end{proof}

\begin{prop} \label{prop.Gaction}
The map of Lie algebras (\ref{lift.Lie}) integrates into
a homomorphism $G\rightarrow\mathrm{Aut}\,\CDO(P)$
if and only if the closed $2$-form $\chi^{2,2}_A-dh^{2,1}_A$
is trivial for any $A\in\mf{g}$.
\end{prop}

\begin{proof}
First we compute the one-parameter subgroup of automorphisms
generated by the inner derivation $(A^P+h^{2,1}_A)_0$.
Since $\CDO(P)$ is freely generated by a vertex
algebroid, it suffices to compute the automorphisms
at weights $0$ and $1$.
For the following computations, keep (\ref{freeVA.comm})
in mind.
For $f\in C^\infty(P)$ and $\alpha\in\Omega^1(P)$ we
simply have
\begin{align}
\label{Gaction.fcn}
\big(A^P+h^{2,1}_A\big)_0^n f=(A^P)^n f
\qquad&\Rightarrow\qquad
e^{t(A^P+h^{2,1}_A)_0}f=e^{tA}\cdot f \\
\label{Gaction.1form}
\big(A^P+h^{2,1}_A\big)_0^n\alpha=L_{A^P}^n\alpha\;\;
\qquad&\Rightarrow\qquad
e^{t(A^P+h^{2,1}_A)_0}\alpha=e^{tA}\cdot\alpha
\end{align}
where $\cdot$ refers to the given $G$-actions.
For $X\in\T(P)$ we first have
\begin{align*}
\big(A^P+h^{2,1}_A\big)_0 X
=[A^P,X]+\{A^P,X\}_\Omega-\iota_X dh^{2,1}_A
=[A^P,X]+\iota_X(\chi^{2,2}_A-dh^{2,1}_A)
\end{align*}
by Lemma \ref{lemma.1form.0mode} and
(\ref{Cartan4.VAoid}).
Then it follows by induction and the $G$-equivariance
of $\chi^{2,2}$, $h^{2,1}$ that
\begin{align}
\big(A^P+h^{2,1}_A\big)_0^n X
&=L_{A^P}^n X
  +nL_{A^P}^{n-1}\iota_X(\chi^{2,2}_A-dh^{2,1}_A),\qquad
  n\geq 1 \nonumber\\
\Rightarrow\qquad
e^{t(A^P+h^{2,1}_A)_0}X
&=e^{tA}\cdot X
  +te^{tA}\cdot\iota_X(\chi^{2,2}_A-dh^{2,1}_A)
\label{Gaction.vf}
\end{align}
This finishes the computation of the automorphisms.

If $\chi^{2,2}_A-dh^{2,1}_A$ vanishes for all $A\in\mf{g}$,
it follows from (\ref{Gaction.fcn})--(\ref{Gaction.vf})
that (\ref{lift.Lie}) integrates into the $G$-action described
in \S\ref{sec.inv.data}.
Conversely, assume that (\ref{lift.Lie}) integrates into
a $G$-action.
By (\ref{Gaction.vf}), $\chi^{2,2}_A-dh^{2,1}_A$ must vanish
whenever $e^A=1$.
Since every element of $G$ lies in a torus, the subset
$\{A\in\mf{g}\,|\,e^A=1\}$ spans $\mf{g}$, so that 
$\chi^{2,2}_A-dh^{2,1}_A$ must in fact vanish for all
$A\in\mf{g}$.
\end{proof}

{\it Remark.}
It follows from the proof that whenever (\ref{lift.Lie}) is
integrable, the resulting $G$-action on $\CDO(P)$ must be
the one described in \S\ref{sec.inv.data}.

\begin{subsec}
{\bf Comparing the three conditions.}
Recall the conditions encountered respectively in
Propositions \ref{prop.lift.Lie}, \ref{prop.lift.VA} and
\ref{prop.Gaction}:
\begin{itemize}
\item[(i)\ph{ii}]
$\chi^{2,2}_A-dh^{2,1}_A$ is $G$-invariant for $A\in\mf{g}$
\vspace{-0.1in}
\item[(ii)\ph{i}]
$\chi^{2,2}_A-dh^{2,1}_A$ is $G$-horizontal for $A\in\mf{g}$
\vspace{-0.1in}
\item[(iii)]
$\chi^{2,2}_A-dh^{2,1}_A=0$ for $A\in\mf{g}$
\end{itemize}
In general, (iii) $\Rightarrow$ (ii) $\Rightarrow$ (i);
the second implication follows from (\ref{Cartan4.closed}).
From another point of view, a map of vertex algebras as in
Proposition \ref{prop.lift.VA} always induces a map of Lie
algebras as in Proposition \ref{prop.lift.Lie},
but the other implication may seem somewhat
surprising.
In the case $\mf{g}$ is semisimple, so that
$[\mf{g},\mf{g}]=\mf{g}$, all three conditions are
equivalent.
\end{subsec}

Here is the main result of this section.

\begin{theorem} \label{thm.formalLG.action}
The $G$-action on $P$ lifts to an inner
$(\hat{\mf{g}}_\lambda,G)$-action on $\CDO(P)$
if and only if
\begin{align*}
8\pi^2 p_1(P)_G=\lambda(P).
\end{align*}
Moreover, this action is primary with respect to the chosen
conformal vector $\nu^\omega$.
\end{theorem}

\begin{proof}
Recall Definition \ref{formalLG.action} and Example
\ref{VAoid.Lie}.
Any inner $(\hat{\mf{g}}_\lambda,G)$-action as described is
determined by a linear map $\mf{g}\rightarrow\CDO(P)_1$ of
the form $A\mapsto A^P+h^{2,1}_A$ for some
$h^{2,1}\in(\mf{g}^\vee\otimes\Omega^1(P))^G$.
By Propositions \ref{prop.lift.VA} and
\ref{prop.Gaction}, the precise conditions on
$h^{2,1}$ are
\begin{align} \label{formalLG.h}
\left\{\begin{array}{l}
  \chi^{2,2}_A=dh^{2,1}_A \vss \\
  \chi^{4,0}_A=\lambda(A,A)-2h^{2,1}_A(A^P)
\end{array}\right.
\quad\iff\quad
2\chi^{2,2}+\chi^{4,0}=\lambda+2d_G h^{2,1}.
\end{align}
By Lemma \ref{lemma.Cartan.closed}a, this proves the
first claim.
The second claim is simply a restatement of Lemma
\ref{lemma.Cartan.closed}b and the subsequent remark.
\end{proof}

{\it Remark.}
In the rest of the paper, $h^{2,1}$ will be referred to as
the {\bf associated Cartan cochain} of the inner
$(\hat{\mf{g}}_\lambda,G)$-action.

\begin{example} \label{CDOG}
{\bf CDOs on a Lie group.}
Let $G$ be a simple compact Lie group.
In this discussion, $A,B$ will always mean general elements
of $\mf{g}=T_e G$ and we will use the following notations:
\vspace{-0.1in}
\begin{itemize}
\item[$\cdot$]
$A^\ell$ (resp.~$A^r$) is the left-invariant
(resp.~right-invariant) vector field on $G$ that extends $A$;
\vspace{-0.1in}
\item[$\cdot$]
$\theta^\ell$ (resp.~$\theta^r$) is the left-invariant
(resp.~right-invariant) Maurer-Cartan form on $G$, i.e.~the
$\mf{g}$-valued $1$-form with $\theta^\ell(A^\ell)=A$
(resp.~$\theta^r(A^r)=A$);
\vspace{-0.1in}
\item[$\cdot$]
$\Killing$ and $\lambda_0=(2h^\vee)^{-1}\cdot\Killing$ are
respectively the Killing form and normalized Killing
form on $\mf{g}$, where $h^\vee$ is the dual Coxeter number.
\vspace{-0.1in}
\end{itemize}
The transgression isomorphism $H^4(BG)\cong H^3(G)$, where
both spaces are one-dimensional, is represented by the
correspondence
\begin{align*}
\textstyle
\lambda\in(\t{Sym}^2\mf{g}^\vee)^G\quad
\longmapsto\quad
-\frac{1}{6}\,\lambda(\theta^\ell\wedge[\theta^\ell\wedge\theta^\ell])
=-\frac{1}{6}\,\lambda(\theta^r\wedge[\theta^r\wedge\theta^r])
  \in\Omega^3(G)\,.
\end{align*}
In particular, the closed $3$-form corresponding to
$\lambda_0$ will be denoted by $H_0$.

Let us construct an algebra of CDOs together with a conformal
vector.
Let $\nabla$ be the connection on $TG$ defined by
$\nabla_{A^\ell}B^\ell=\frac{1}{2}[A,B]^\ell$, or equivalently,
by $\nabla_{A^r}B^r=-\frac{1}{2}[A,B]^r$.
\footnote{
Indeed, both define the \emph{unique} torsion-free connection
that is compatible with any bi-invariant symmetric bilinear form.
}
Notice that $p_1(\nabla)=0$.
By Theorems \ref{thm.globalCDO}a and \ref{thm.globalCDO.iso}a,
$\nabla$ and $kH_0$ for any $k\in\bb{C}$ together determine
an algebra of CDOs $\CDO_k(G)=\CDO_{\nabla,kH_0}(G)$, and
every algebra of CDOs on $G$ is up to isomorphism of this form.
In the vertex algebroid structure of $\CDO_k(G)$, we find that
\begin{align}
\label{CDOG.VAoidl}
&\textstyle
\{A^\ell,B^\ell\}=-\frac{1}{2}h^\vee\lambda_0(A,B)
&&\textstyle
\{A^\ell,B^\ell\}_\Omega
=\big(\frac{1}{2}k+\frac{1}{4}h^\vee\big)
  \lambda_0([A,B],\theta^\ell) \\
\label{CDOG.VAoidr}
&\textstyle
\{A^r,B^r\}=-\frac{1}{2}h^\vee\lambda_0(A,B)
&&\textstyle
\{A^r,B^r\}_\Omega
=\big(\frac{1}{2}k-\frac{1}{4}h^\vee\big)
  \lambda_0([A,B],\theta^r) \\
\label{CDOG.VAoidlr}
&\textstyle
\{A^\ell,B^r\}=\frac{1}{2}h^\vee\lambda_0(A,\theta^\ell(B^r))
&&\textstyle
\{A^\ell,B^r\}_\Omega
=\big(\frac{1}{2}k+\frac{1}{4}h^\vee\big)
  \lambda_0([A,\theta^\ell(B^r)],\theta^\ell)
\end{align}
By Theorem \ref{thm.globalCDO}b, the trivial $1$-form determines
a conformal vector $\nu$ of central charge $2\dim\mf{g}$.
If we choose a basis $t_1,t_2,\ldots$ of $\mf{g}$ and denote
the corresponding components of $\theta^\ell$ (resp.~$\theta^r$)
by $\theta^{1,\ell},\theta^{2,\ell},\ldots$
(resp.~$\theta^{1,r},\theta^{2,r},\ldots$), then we can write
\begin{align} \label{CDOG.conformal}
\textstyle
\nu
=t^\ell_{a,-1}\theta^{a,\ell}
  -\frac{3}{4}h^\vee\lambda_0(\theta^\ell_{-1}\theta^\ell)
=t^r_{a,-1}\theta^{a,r}
  -\frac{3}{4}h^\vee\lambda_0(\theta^r_{-1}\theta^r)\,.
\end{align}
Its Virasoro field has the property that
$L_1 A^\ell=L_1 A^r=0$.
\footnote{
The expressions in (\ref{CDOG.VAoidl})--(\ref{CDOG.conformal})
are obtained by some computations that we have omitted.
}

Consider the action of $G$ on itself by right multiplication.
The induced map of Lie algebras $\mf{g}\rightarrow\T(G)$
takes $A$ to $A^\ell$.
Both $\nabla$ and $H_0$ are invariant under this action.
In view of (\ref{CDOG.VAoidl}), the equations in
(\ref{formalLG.h}) are satisfied by
\begin{align} \label{hl}
\textstyle
h^\ell\in(\mf{g}^\vee\otimes\Omega^1(G))^G,\qquad
h^\ell_A:=-\big(\frac{1}{2}k+\frac{1}{4}h^\vee\big)
  \lambda_0(A,\theta^\ell)
\end{align}
and $\lambda=(-k-h^\vee)\lambda_0$.
Therefore we have an inner $(\hat{\mf{g}},G)$-action
\begin{align} \label{CDOG.l}
V_{-k-h^\vee}(\mf{g})\hookrightarrow\CDO_k(G)
\end{align}
with associated Cartan cochain $h^\ell$;
its image will be denoted by $V_{-k-h^\vee}(\mf{g})^\ell$.
This action is primary with respect to $\nu$.

Consider also the action of $G$ on itself by inverse
left multiplication.
The induced map of Lie algebras $\mf{g}\rightarrow\T(G)$
takes $A$ to $-A^r$.
Both $\nabla$ and $H_0$ are invariant under this action
as well.
In view of (\ref{CDOG.VAoidr}), the equations in
(\ref{formalLG.h}) are now satisfied by
\begin{align} \label{hr}
\textstyle
h^r\in(\mf{g}^\vee\otimes\Omega^1(G))^G,\qquad
h^r_A:=\big(-\frac{1}{2}k+\frac{1}{4}h^\vee\big)
  \lambda_0(A,\theta^r)
\end{align}
and $\lambda=(k-h^\vee)\lambda_0$.
Therefore we have another inner $(\hat{\mf{g}},G)$-action
\begin{align} \label{CDOG.r}
V_{k-h^\vee}(\mf{g})\hookrightarrow\CDO_k(G)
\end{align}
with associated Cartan cochain $h^r$;
its image will be denoted by $V_{k-h^\vee}(\mf{g})^r$.
This action is also primary with respect to $\nu$.
\end{example}

The next few propositions provide more details on
these two $(\hat{\mf{g}},G)$-actions on $\CDO_k(G)$.

\begin{prop} \label{prop.CDOG.commute}
The vertex subalgebras $V_{-k-h^\vee}(\mf{g})^\ell$ and
$V_{k-h^\vee}(\mf{g})^r$ commute with each other.
\end{prop}

\begin{proof}
This amounts to showing that
$(A^\ell+h^\ell_A)_i(-B^r+h^r_B)=0$ for $A,B\in\mf{g}$
and $i=0,1$.
For $i=1$, here is the calculation
\begin{align*}
(A^\ell+\,&h^\ell_A)_1(-B^r+h^r_B) \\
&=-\{A^\ell,B^r\}-h^\ell_A(B^r)+h^r_B(A^\ell) \\
&\textstyle
=\frac{1}{2}h^\vee\lambda_0(A,\theta^\ell(B^r))
+\big(\frac{1}{2}k+\frac{1}{4}h^\vee\big)\lambda_0(A,\theta^\ell(B^r))
+\big(-\frac{1}{2}k+\frac{1}{4}h^\vee\big)\lambda_0(\theta^\ell(B^r),A)
=0
\end{align*}
which follows from (\ref{freeVA.comm}),
(\ref{CDOG.VAoidlr}), (\ref{hl}) and (\ref{hr}).
For $i=0$, we have a similar calculation
\begin{align*}
(A^\ell+\,&h^\ell_A)_0(-B^r+h^r_B) \\
&=-[A^\ell,B^r]-\{A^\ell,B^r\}_\Omega
  +\iota_{B^r}dh^\ell_A+L_{A^\ell}h^r_B \\
&\textstyle
=0-\big(\frac{1}{2}k+\frac{1}{4}h^\vee\big)
  \lambda_0([A,\theta^\ell(B^r)],\theta^\ell)
+\big(\frac{1}{2}k+\frac{1}{4}h^\vee\big)
  \lambda_0(A,[\theta^\ell(B^r),\theta^\ell])
+0=0\qquad
\end{align*}
which also utilizes Lemma \ref{lemma.1form.0mode}
and the commutativity between left and right
multiplications.
\end{proof}

{\it Remark.}
This recovers the well-known fact that every algebra of CDOs
on $G$ provides a realization of a commuting pair of affine
Lie algebras of dual levels.
\cite{GMS4}

\begin{prop} \label{prop.CDOG.conformal}
For any $k\neq 0$, the sum of the Sugawara vectors of
$V_{-k-h^\vee}(\mf{g})^\ell$ and $V_{k-h^\vee}(\mf{g})^r$
equals the conformal vector (\ref{CDOG.conformal}) of
$\CDO_k(G)$.
\end{prop}

\begin{proof}
Let us outline the calculation without going through all
the details.
The assumption $k\neq 0$ is precisely the condition for either
affine vertex algebra to admit a Sugawara vector.~\cite{Kac,FB-Z}
In addition to the notations introduced for the expression in
(\ref{CDOG.conformal}), also let $t^1,t^2,\ldots$ denote the
dual basis of $\mf{g}$ with respect to $\lambda_0$ and
$\theta_1^\ell,\theta_2^\ell,\ldots$
(resp.~$\theta_1^r,\theta_2^r,\ldots$) the corresponding
components of $\theta^\ell$ (resp.~$\theta^r$).
By definition, the Sugawara vectors of
$V_{-k-h^\vee}(\mf{g})^\ell$ and $V_{k-h^\vee}(\mf{g})^r$
are respectively
\begin{align*}
\nu^\ell
&\textstyle
=-\frac{1}{2k}\big(t_a^\ell+h^\ell_{t_a}\big)_{-1}
  \big(t^{a,\ell}+h^\ell_{t^a}\big)
=-\frac{1}{2k}t_{a,-1}^\ell t^{a,\ell}
+\frac{2k+h^\vee}{4k}\,t_{a,-1}^\ell\theta^{a,\ell}
-\frac{(2k+h^\vee)^2}{32k}\,\theta^\ell_{a,-1}\theta^{a,\ell} \\
\nu^r
&\textstyle
=\frac{1}{2k}\big(-t_a^r+h^r_{t_a}\big)_{-1}
  \big(-t^{a,r}+h^r_{t^a}\big)
=\frac{1}{2k}t_{a,-1}^r t^{a,r}
+\frac{2k-h^\vee}{4k}\,t_{a,-1}^r\theta^{a,r}
+\frac{(2k-h^\vee)^2}{32k}\,\theta^r_{a,-1}\theta^{a,r}
\end{align*}
In order to express the quantities with superscript $r$ in
terms of those with superscript $\ell$, let
$\rho_a^b=\theta^{b,\ell}(t_a^r)$.
By (\ref{freeVA.NOP}) and the definition of $\nabla$,
\begin{align*}
t_a^r
=\rho_a^b\,t_b^\ell
=t^\ell_{b,-1}\rho_a^b-t_b^\ell\bullet\rho_a^b
=t^\ell_{b,-1}\rho_a^b-(\nabla t_b^\ell)\rho_a^b
=t^\ell_{b,-1}\rho_a^b-h^\vee\theta_a^r\,.
\end{align*}
Then we have the following calculations:
\begin{align*}
\theta^r_{a,-1}\theta^{a,r}
&=\big(\rho_a^b\,\theta^\ell_b\big)_{-1}
  \big((\rho^{-1})^a_c\,\theta^{c,\ell}\big)
=\theta^\ell_{b,-1}\theta^{b,\ell} \\
t^r_{a,-1}\theta^{a,r}
&=\big(t^\ell_{b,-1}\rho_a^b-h^\vee\theta_a^r\big)_{-1}
  \theta^{a,r} \\
&=\big(t^\ell_{b,-1}\rho_{a,0}^b
  +\rho_{a,-1}^b\,t^\ell_{b,0}
  +\rho_{a,-2}^b\,t^\ell_{b,1}\big)\theta^{a,r}
-h^\vee\theta_{a,-1}^\ell\theta^{a,\ell} \\
&=t^\ell_{b,-1}\theta^{b,\ell}+0
+h^\vee\theta_{c,-1}^\ell\theta^{c,\ell}
-h^\vee\theta_{a,-1}^\ell\theta^{a,\ell} \\
&=t^\ell_{b,-1}\theta^{b,\ell} \\
t^r_{a,-1}t^{a,r}
&=\big(t^\ell_{b,-1}\rho_a^b-h^\vee\theta_a^r\big)_{-1}
  t^{a,r} \\
&=\big(t^\ell_{b,-2}\rho_{a,1}^b
  +t^\ell_{b,-1}\rho_{a,0}^b
  +\rho_{a,-1}^b\,t^\ell_{b,0}
  +\rho_{a,-2}^b\,t^\ell_{b,1}\big)\,t^{a,r}
-h^\vee t^r_{a,-1}\theta^{a,r} \\
&\textstyle
=0+t_{b,-1}^\ell(t^{b,\ell}+h^\vee\theta^{b,\ell})
-\big(k+\frac{1}{2}h^\vee\big)h^\vee
  \theta_{c,-1}^\ell\theta^{c,\ell}
+\frac{1}{2}(h^\vee)^2
  \theta_{c,-1}^\ell\theta^{c,\ell}
-h^\vee t^\ell_{a,-1}\theta^{a,\ell} \\
&=t_{b,-1}^\ell t^{b,\ell}
-kh^\vee\theta_{c,-1}^\ell\theta^{c,\ell}
\end{align*}
It is now easy to see that $\nu^\ell+\nu^r$ coincides
with either expression in (\ref{CDOG.conformal}).
\end{proof}

{\it Preparation.}
(i) For $k'\in\bb{C}$, we will write $\hat{\mf{g}}_{k'}$ for
$\hat{\mf{g}}_{k'\lambda_0}$ (see \S\ref{sec.Gmfld}) and
$U(\hat{\mf{g}})_{k'}$ for its universal enveloping algebra.
Let $\hat{\mf{g}}_+=\mf{g}[t]$ and
$\hat{\mf{g}}_{++}=t\cdot\mf{g}[t]$.
(ii) Let $\mb{P}$ denote the set of dominant weights of $G$
(with respect to a choice of maximal torus and Weyl chamber).
For $\mu\in\mb{P}$ and $k'\in\bb{C}$, let us write $\mu^*$
for $-w_{\t{max}}(\mu)$, where $w_{\t{max}}$ is the longest
element of the Weyl group;
$\bb{M}_\mu$ for the irreducible $G$-module of highest
weight $\mu$;
and $\bb{M}_{k',\mu}$ for the irreducible positive-energy
$(\hat{\mf{g}}_{k'},G)$-module of highest weight $\mu$.
(iii) To avoid confusion, eigenvalues of $L_0$ will be
referred to as \emph{conformal} weights here.

\begin{prop} \label{prop.PeterWeyl}
For any $k\notin\bb{Q}$, there is a $(\hat{\mf{g}}_{-k-h^\vee}
\oplus\hat{\mf{g}}_{k-h^\vee},G\times G)$-equivariant embedding
\begin{align} \label{PeterWeyl}
\bigoplus_{\mu\in\mb{P}}
  \bb{M}_{-k-h^\vee,\mu}\otimes
  \bb{M}_{k-h^\vee,\mu^*}
\;\hookrightarrow\;
\CDO_k(G)
\end{align}
and its image is ``dense'' in the sense described below.
Notice that the summand with $\mu=0$ is the tensor product of
(\ref{CDOG.l}) and (\ref{CDOG.r}).
\end{prop}

\begin{proof} 
By Example \ref{CDOG} and Proposition
\ref{prop.CDOG.commute}, the $G\times G$-action on $C^\infty(G)$
coming from left and right multiplications extends to
a $(\hat{\mf{g}}_{-k-h^\vee}\oplus\hat{\mf{g}}_{k-h^\vee},
G\times G)$-action on $\CDO_k(G)$.
According to the Peter-Weyl theorem, there is a canonical
$G\times G$-equivariant embedding
\begin{align*}
\bigoplus_{\mu\in\mb{P}}\bb{M}_\mu\otimes\bb{M}_{\mu^*}
\cong
\bigoplus_{\mu\in\mb{P}}\bb{M}_\mu\otimes\bb{M}_{\mu}^\vee
\;\hookrightarrow\;
C^\infty(G)
\end{align*}
whose image is dense in the $L^\infty$-topology.
This induces
a $(\hat{\mf{g}}_{-k-h^\vee}\oplus\hat{\mf{g}}_{k-h^\vee},
G\times G)$-equivariant map
\begin{align*}
\bigg(U(\hat{\mf{g}})_{-k-h^\vee}\otimes
  U(\hat{\mf{g}})_{k-h^\vee}\bigg)
\otimes_{U(\hat{\mf{g}}_+)\otimes
  U(\hat{\mf{g}}_+)}
\bigg(
  \bigoplus_{\mu\in\mb{P}}\bb{M}_\mu\otimes\bb{M}_{\mu^*}
\bigg)
\;\rightarrow\;
\CDO_k(G)\,.
\end{align*}
If $k'\notin\bb{Q}$, then it follows from a consideration of
the Casimir operator that any $\hat{\mf{g}}_{k'}$-module of the
form $U(\hat{\mf{g}})_{k'}\otimes_{U(\hat{\mf{g}}_+)}\bb{M}_\mu$
contains no proper submodule.
Therefore the above induced map is equivalent to a map of the
form (\ref{PeterWeyl}).

For $\CDO_k(G)$, the filtration described in \S\ref{freeVA.PBW}
splits $G\times G$-equivariantly with the use of (say)
left-invariant vector fields and $1$-forms.
This results in a $G\times G$-equivariant isomorphism
\begin{align*}
\bigoplus_{w\geq 0}q^i\CDO_k(G)_i
\cong
C^\infty(G)\otimes
\bigg(
  \bigotimes_{r\geq 1}
  \t{Sym}_{q^r}(\mf{g}\oplus\mf{g}^\vee)
\bigg)
\end{align*}
where the first $G$ acts on the symmetric powers in
the obvious way and the second $G$ acts trivially there.
Using this isomorphism, the $L^\infty$-topology on
$C^\infty(G)$ induces a topology on $\CDO_k(G)$ that is
compatible with the vertex algebra structure.
If the image of (\ref{PeterWeyl}) is not dense, then we
can find some nontrivial $G\times G$-submodule
$\bb{M}\subset\CDO_k(G)$ with no intersection with the said
image.
Also, we may assume that $\bb{M}$ is irreducible and, among
all such submodules, has the lowest conformal weight,
which is necessarily positive.
Notice that $\hat{\mf{g}}_{++}$ acts trivially on $\bb{M}$.
By Proposition \ref{prop.CDOG.conformal},
\begin{align*}
\textstyle
L_0
=\nu^\ell_0+\nu^r_0
=\frac{1}{2k}(-\Omega^\ell_0+\Omega^r_0)
\quad\t{on}\;\;\bb{M}
\end{align*}
where $\Omega^\ell_0$, $\Omega^r_0$ are the Casimir numbers
of the two $\mf{g}$-actions.
Since $\Omega^\ell_0,\Omega^r_0\in\bb{Q}$ but $k\notin\bb{Q}$,
the conformal weight of $\bb{M}$ must be zero, giving
a contradiction.
This proves that (\ref{PeterWeyl}) has a dense image.
\end{proof}

{\it Remark.}
This is a reproduction of one of the
``chiral Peter-Weyl theorems'' in \cite{FS.PeterWeyl}.
It will be very interesting to understand the representation-theoretic
meaning of the vertex algebra structure on (\ref{PeterWeyl}),
as well as to find an appropriate extension of the result for
$k\in\bb{Q}$;
see \emph{op.~cit.} for some conjectures.

\newpage
\setcounter{equation}{0}
\section{CDOs on Principal Bundles}
\label{sec.CDOP}

In this section, our object of study is an algebra of CDOs
$\CDO(P)$ on the total space of a principal bundle
$P\rightarrow M$ when it is equipped with
a ``formal loop group action'' in the sense of
\S\ref{sec.LGaction}.
The first goal is to understand the subalgebra of $\CDO(P)$
invariant under the action.
In subsequent sections we will construct and study modules
over this invariant subalgebra.
The second (and more technical) goal is to give, when
$P\rightarrow M$ is a principal frame bundle, an alternative
description of $\CDO(P)$ that makes the
``formal loop group action'' manifest.
This type of vertex algebras will play a central role in
the rest of the paper.

\begin{subsec} \label{sec.prin.bdl}
{\bf Setting:~principal bundle.}
Let $G$ be again a compact connected Lie group and
$\pi:P\rightarrow M$ a smooth principal $G$-bundle.
Identify $H^*_G(P)$ with $H^*(M)$ via $\pi^*$.
Choose a connection $\Theta$ on $\pi$,
i.e. some $\Theta\in(\Omega^1(P)\otimes\mf{g})^G$
such that $\Theta(A^P)=A$ for $A\in\mf{g}$;
its curvature is
$\Omega=d\Theta+\frac{1}{2}[\Theta\wedge\Theta]$.
Keep in mind that $\Theta$ is equivalent to
a $G$-equivariant vector bundle decomposition
$TP=T_h P\oplus T_v P$, where $T_h P=\ker\Theta$ and
$T_v P=\ker\pi_*$;
and $\Omega$ measures the non-integrability of
the subbundle $T_h P$.
Given $X\in\T(M)$, denote its horizontal lift by
$\wt X\in\T_h(P)^G$.
Let us extend the notation $(-)^P$ (see \S\ref{sec.Gmfld})
$C^\infty(P)$-linearly to denote the isomorphism
$C^\infty(P)\otimes\mf{g}\cong\T_v(P)$.
For $X,Y\in\T(M)$ notice that
\begin{align} \label{horlift.nonint}
[\wt X,\wt Y]-\wt{[X,Y]}=-\Omega(\wt X,\wt Y)^P.
\end{align}
Also recall the notations introduced in \S\ref{sec.Cartan}.

Consider an algebra of CDOs $\CDO(P)=\CDO_{\nabla,H}(P)$
defined as in \S\ref{sec.inv.data}.
Let $\lambda\in(\t{Sym}^2\mf{g}^\vee)^G$.
Recall from Definition \ref{formalLG.action} and the proof
of Theorem \ref{thm.formalLG.action} the meaning of an inner
$(\hat{\mf{g}}_\lambda,G)$-action on $\CDO(P)$ and its
associated Cartan cochain.
It will be understood without further comment that any inner
$(\hat{\mf{g}}_\lambda,G)$-action on $\CDO(P)$ considered
below extends the given $G$-action on $\CDO(P)_0=C^\infty(P)$.
\end{subsec}

\begin{lemma} \label{lemma.Gp1.p1}
$8\pi^2 p_1(P)_G=8\pi^2 p_1(M)-\Killing(P)$.
\end{lemma}

\begin{proof}
Recall the $G$-equivariant decomposition
$TP=T_h P\oplus T_v P$.
Since $T_h P\cong\pi^*TM$ and $\pi^*$ identifies $H^*_G(P)$
with $H^*(M)$, we have $p_1(T_h P)_G=p_1(TM)$.
Since $T_v P\cong P\times\mf{g}$ with $G$ acting on $\mf{g}$
in the adjoint representation, we also have
$-8\pi^2 p_1(T_v P)_G=\Killing(P)$ (see \S\ref{conventions}).
This proves the lemma.
\end{proof}

By this lemma, Theorem \ref{thm.formalLG.action} specializes
to our current setting as follows.

\begin{corollary} \label{cor.formalLG.P}
The $G$-action on $P$ lifts to an inner
$(\hat{\mf{g}}_\lambda,G)$-action on $\CDO(P)$ if
and only if
\begin{align} \label{formalLG.condition2}
\qquad
8\pi^2 p_1(M)=(\lambda+\Killing)(P).
\qquad\qedsymbol
\end{align}
\end{corollary}

{\it Remark.}
For convenience, when $\CDO(P)$ is equipped with an inner
$(\hat{\mf{g}}_\lambda,G)$-action, we will refer to it as
a {\bf principal $(\hat{\mf{g}}_\lambda,G)$-algebra}.

\begin{subsec} \label{sec.invCDO}
{\bf The invariant subalgebra.}
Suppose $\CDO(P)$ is a principal
$(\hat{\mf{g}}_\lambda,G)$-algebra, i.e.~it is given
an inner $(\hat{\mf{g}}_\lambda,G)$-action 
$V_\lambda(\mf{g})\hookrightarrow\CDO(P)$, defined by some
associated Cartan cochain
$h\in(\mf{g}^\vee\otimes\Omega^1(P))^G$.
Consider the centralizer subalgebra
\begin{align*}
\CDO(P)^{\hat{\mf{g}}}
:=C\big(\CDO(P),V_\lambda(\mf{g})\big)
\subset\CDO(P).
\end{align*}
\cite{Kac,FB-Z}
This is the subalgebra whose fields are
$\hat{\mf{g}}_\lambda$-invariant.
The weight-zero component consists of $f\in C^\infty(P)$
satisfying $(A^P+h_A)_0 f=A^P f=0$ for $A\in\mf{g}$, i.e.
\begin{align} \label{invCDO.0}
\CDO(P)^{\hat{\mf{g}}}_0
=C^\infty(P)^G
=\pi^*C^\infty(M).
\end{align}
The weight-one component consists of
$\alpha+X\in\Omega^1(P)\oplus\T(P)$ satisfying
\begin{align*}
\left\{\begin{array}{l}
(A^P+h_A)_0(\alpha+X)=L_{A^P}\alpha+[A^P,X]=0 \vss \\
(A^P+h_A)_1(\alpha+X)=\alpha(A^P)+\{A^P,X\}+h_A(X)=0
\end{array}\right.\qquad
\t{for }A\in\mf{g},
\end{align*}
where we have used (\ref{freeVA.comm}), Lemma
\ref{lemma.1form.0mode} and (\ref{formalLG.h}).
It follows that there is a pullback square
\begin{align*}
\xymatrix{
  \CDO(P)^{\hat{\mf{g}}}_1\ar[r]\ar[d] &
  \T(P)^G\ar[d]^q \\
  \Omega^1(P)^G\ar[r]^{p\ph{WW}} &
  (\mf{g}^\vee\otimes C^\infty(P))^G
}
\end{align*}
where $p$, $q$ are suitable adjoints to
$(\alpha,A)\mapsto-\alpha(A^P)$ and
$(X,A)\mapsto\{A^P,X\}+h_A(X)$ respectively.
Observe that
(i) $p$ is surjective, as contraction with $\Theta$ provides
a right inverse, and
(ii) $\ker p$ is the space of basic $1$-forms on $P$, 
i.e.~$\pi^*\Omega^1(M)$.
This implies the short exact sequence
\begin{align}\label{invCDO.1}
\xymatrix{
  0\ar[r] &
  \pi^*\Omega^1(M)\ar[r] &
  \CDO(P)^{\hat{\mf{g}}}_1\ar[r] &
  \T(P)^G\ar[r] & 0
}
\end{align}
In view of (\ref{invCDO.0}) and (\ref{invCDO.1}),
the vertex algebroid associated to
$\CDO(P)^{\hat{\mf{g}}}$ (see \S\ref{VA.VAoid}) is
of the form
\begin{align} \label{invCDO.VAoid}
\big(C^\infty(M),\Omega^1(M),\T(P)^G,\cdots\big).
\end{align}
Since $\CDO(P)$ is freely generated by its associated
vertex algebroid (see \S\ref{VAoid.VA}),
$\CDO(P)^{\hat{\mf{g}}}$ must contain at least
a subalgebra that is freely generated by
(\ref{invCDO.VAoid}).
In general, it is not clear whether
$\CDO(P)^{\hat{\mf{g}}}$ may contain other elements
or not.
(However, see Example \ref{invCDOG}.)
\end{subsec}

{\it Remarks.}
(i) By the above discussion, any module over
$\CDO(P)^{\hat{\mf{g}}}$ contains in its lowest weight a module
over the Atiyah algebroid $(C^\infty(M),\T(P)^G)$, such as the
space of sections of an associated vector bundle.
In subsequent sections we will construct and study such modules
over $\CDO(P)^{\hat{\mf{g}}}$.
(ii) It is noteworthy that $\CDO(P)^{\hat{\mf{g}}}$ is different
from an algebra of CDOs on the base manifold $M$ --- compare
(\ref{invCDO.1}) and (\ref{CDO.wt01}) --- and admits more
interesting modules.

\begin{example} \label{invCDOG}
{\bf CDOs on a Lie group.}
This is a continuation of Example \ref{CDOG} and
the same notations will be used.
Consider the centralizers of the inner
$(\hat{\mf{g}},G)$-actions (\ref{CDOG.l}) and
(\ref{CDOG.r}):
\begin{align*}
\CDO_k(G)^{\hat{\mf{g}},\ell}
=C\big(\CDO_k(G),V_{-k-h^\vee}(\mf{g})^\ell\big),\qquad
\CDO_k(G)^{\hat{\mf{g}},r}
=C\big(\CDO_k(G),V_{k-h^\vee}(\mf{g})^r\big).
\end{align*}
By Proposition \ref{prop.CDOG.commute}, they contain the
following subalgebras
\begin{align} \label{aff.invCDOG}
V_{k-h^\vee}(\mf{g})^r\subset
  \CDO_k(G)^{\hat{\mf{g}},\ell},\qquad
V_{-k-h^\vee}(\mf{g})^\ell\subset
  \CDO_k(G)^{\hat{\mf{g}},r}\,.
\end{align}
For $k\neq 0$, both affine vertex algebras admit Sugawara
vectors.
By the coset construction and Proposition
\ref{prop.CDOG.conformal}, the Sugawara vectors are also
conformal for the respective centralizers.
\cite{Kac,FB-Z}
In fact, for $k\notin\bb{Q}$, it follows from Proposition
\ref{prop.PeterWeyl} that
\begin{align*}
V_{k-h^\vee}(\mf{g})^r=\CDO_k(G)^{\hat{\mf{g}},\ell},\qquad
V_{-k-h^\vee}(\mf{g})^\ell=\CDO_k(G)^{\hat{\mf{g}},r}\,.
\end{align*}
\footnote{
The author is grateful to a reviewer for pointing this out to
him.
}
These may be viewed as examples of a version of the Borel-Weil
construction.
\end{example}

\begin{subsec} \label{sec.prin.bdl2}
{\bf Setting refined:~principal frame bundle.}
In addition to the data described in \S\ref{sec.prin.bdl},
suppose we also have a representation
$\rho:G\rightarrow SO(\bb{R}^d)$ together with an isomorphism
$P\times_\rho\bb{R}^d\cong TM$.
This induces a Riemannian metric on $M$ and an orthogonal
connection on $TM$, which (for simplicity) is assumed to be
torsion-free.
Let $[p,v]\in TM$ denote the coset of
$(p,v)\in P\times\bb{R}^d$.
For $i=1,\ldots,d$, let $\mb{e}_i\in\bb{R}^d$ be the standard
basis vectors and $\tau_i\in\T_h(P)$ the tautological vector
fields defined by
\begin{align*}
\tau_i|_p:=\t{horizontal lift of }[p,\mb{e}_i],
\end{align*}
which constitute a framing of $T_h P$.
For $A\in\mf{g}$ and $i,j=1,\ldots,d$, we have the Lie brackets
\begin{align} \label{tau.brackets}
[A^P,\tau_i]=\rho(A)_{ji}\tau_j,
\qquad
[\tau_i,\tau_j]=-\Omega(\tau_i,\tau_j)^P
\end{align}
where we use $\rho$ to also denote the induced map of Lie
algebras $\mf{g}\rightarrow\mf{so}_d$.
\footnote{
The absence of a horizontal component in $[\tau_i,\tau_j]$
is equivalent to the torsion-free assumption.
Moreover, the Jacobi identity for any
$\tau_i,\tau_j,\tau_k$ is equivalent to the two Bianchi
identities.
}

Let us define an algebra of CDOs on $P$ as in
\S\ref{sec.inv.data} using some more specific data (and
slightly different notations).
Let $\nabla'$ be the $G$-invariant flat connection on $TP$ with
respect to which all $A^P$ and $\tau_i$ are parallel, and $H'$
be a $G$-invariant closed $3$-form on $P$.
By Theorem \ref{thm.globalCDO}a, $\nabla'$ and $H'$ determine
an algebra of CDOs $\CDO(P)'=\CDO_{\nabla',H'}(P)$, which is
freely generated by a vertex algebroid
\begin{align} \label{CDOP.VAoid}
\big(C^\infty(P),\Omega^1(P),\T(P),
  \bullet',\{\;\}',\{\;\}'_\Omega\big)
\end{align}
(see \S\ref{VAoid.VA}).
The structure of this vertex algebroid is, in view of the
axioms (see Definition \ref{VAoid}), entirely determined by the
portion displayed in the ensuing lemma.
By Theorem \ref{thm.globalCDO}b, the trivial $1$-form determines
a conformal vector in $\CDO(P)'$, namely
\begin{align} \label{CDOP.conformal}
\nu'=t^P_{a,-1}\Theta^a+\tau_{i,-1}\tau^i
\end{align}
where $t_1,t_2,\cdots$ are a basis of $\mf{g}$;
$\Theta^1,\Theta^2,\cdots$ the corresponding components of
$\Theta$, so that $\Theta=\Theta^a\otimes t_a$;
$\tau_1,\cdots,\tau_d$ the horizontal vector fields defined above;
and $\tau^1,\cdots,\tau^d$ the dual horizontal $1$-forms.
In the rest of this section, we will provide a more detailed
description of this conformal vertex algebra in the presence of
an inner $(\hat{\mf{g}}_\lambda,G)$-action.
\end{subsec}

\begin{lemma} \label{lemma.CDOP.VAoid}
For $f\in C^\infty(P)$, $A,B\in\mf{g}$ and
$i,j=1,\ldots,d$, we have \\
\begin{tabular}{cl}
$\cdot$ &
$A^P\bullet'f=\tau_i\bullet'f=0$ \\
$\cdot$ &
$\{A^P,B^P\}'=-(\Killing+\lambda_\rho)(A,B)$;\;
$\{A^P,\tau_i\}'=0$;\;
$\{\tau_i,\tau_j\}'=2\mathrm{Ric}_{ij}$  \\
$\cdot$ &
$\{A^P,B^P\}'_\Omega
=\frac{1}{2}\iota_{A^P}\iota_{B^P}H'$;\;
$\{A^P,\tau_i\}'_\Omega
=\frac{1}{2}\iota_{A^P}\iota_{\tau_i}H'$;\;
$\{\tau_i,\tau_j\}'_\Omega=d\,\mathrm{Ric}_{ij}
+\frac{1}{2}\iota_{\tau_i}\iota_{\tau_j}H'$
\end{tabular} \\
where $\mathrm{Ric}_{ij}=\rho(\Omega(\tau_i,\tau_k))_{jk}$.
Moreover, we also have $\nu'_1 A^P=\nu'_1\tau_i=0$.
\end{lemma}

\begin{proof}
The structure maps are defined by (\ref{CDO.VAoid}) and
(\ref{CDO.L1}).
Then the calculations follow easily from our definition
of $\nabla'$, the Lie brackets (\ref{tau.brackets}) and
the symmetries of the Riemannian curvature tensor
$\rho(\Omega)$.
\end{proof}

Theorem \ref{thm.formalLG.action} specializes
to our current setting as follows.

\begin{corollary} \label{cor.formalLG.P2}
The $G$-action on $P$ lifts to an inner
$(\hat{\mf{g}}_\lambda,G)$-action on $\CDO(P)'$
if and only if
\begin{align*}
(\lambda+\Killing+\lambda_\rho)(P)=0.
\end{align*}
Moreover, this action is primary with respect to
the conformal vector $\nu'$.
\end{corollary}

\begin{proof}
By assumption, $\pi:P\rightarrow M$ is the lifting of the
special orthogonal frame bundle $F_{SO}(TM)\rightarrow M$
along $\rho:G\rightarrow SO_d$.
This gives us a commutative diagram
\begin{align*}
\xymatrix{
  &
  H^4(BSO_d)\ar[r]^{\;\;\rho^*}\ar[d] &
  H^4(BG)\ar[d] \\
  H^4(M)\ar[r]^{=\phantom{WWJ}} &
  H^4_{SO_d}(F_{SO}(TM))\ar[r]^{\phantom{WWI}=} &
  H^4_G(P)
}
\end{align*}
By definition, $\lambda_\rho\in H^4(BG)$ is the image of
$-8\pi^2 p_1\in H^4(BSO_d)$ under $\rho^*$ (see
\ref{conventions}), which implies that
$-8\pi^2p_1(M)=\lambda_\rho(P)$.
Now the first claim becomes a special case of Corollary
\ref{cor.formalLG.P}.
The other claim is true by Lemma \ref{lemma.CDOP.VAoid}.
\end{proof}

{\it Remark.}
In the sequel we will write
$\lambda^*=\lambda+\Killing+\lambda_\rho$.

\begin{subsec} \label{CDOP.FLGaction}
{\bf Formal loop group action.}
Suppose $\CDO(P)'$ is a principal
$(\hat{\mf{g}}_\lambda,G)$-algebra, i.e.~it is given an inner
$(\hat{\mf{g}}_\lambda,G)$-action
$V_\lambda(\mf{g})\hookrightarrow\CDO(P)'$, defined by some
associated Cartan cochain
$h\in(\mf{g}^\vee\otimes\Omega^1(P))^G$.
Notice that the Cartan cochains in 
(\ref{Cartan4.VAoid}) may now be written as
$\chi^{2,2}=\frac{1}{2}d_G H'$ and
$\chi^{4,0}=-\Killing-\lambda_\rho$ by Lemma
\ref{lemma.CDOP.VAoid}.
Hence the condition (\ref{formalLG.h}) on $h$
is now equivalent to the equation
\begin{align} \label{CDOP.h}
d_G(H'-2h)=\lambda^*.
\end{align}
This of course agrees with Corollary
\ref{cor.formalLG.P2}.
The purpose of the next two lemmas is to
modify the description of the vertex algebra
$\CDO(P)'$ in such a way that the inner
$(\hat{\mf{g}}_\lambda,G)$-action becomes
manifest and the data $(H',h)$ involved are
replaced by a basic $3$-form that trivializes
$\lambda^*(\Omega\wedge\Omega)$.
\end{subsec}

\begin{lemma} \label{lemma.basic.3form}
There is a natural bijection between the
following two sets of data:

(i) $(H',h)\in\Omega^3(P)^G\oplus
(\mf{g}^\vee\otimes\Omega^1(P))^G$ such
that $d_G(H'-2h)=\lambda^*$, and

(ii) $(H,\beta)\in\pi^*\Omega^3(M)\oplus\Omega^2(P)^G$
such that $dH=\lambda^*(\Omega\wedge\Omega)$ and
$\beta|_{T_h P}=0$.
\end{lemma}

\begin{proof}
Once the maps in both directions are described below,
it will be clear that they are inverse to each other.
Let $\mathrm{CS}_{\lambda^*}(\Theta)
=\lambda^*(\Theta\wedge\Omega)
-\frac{1}{6}\lambda^*(\Theta\wedge[\Theta\wedge\Theta])$.
Notice that as Cartan cochains
$\mathrm{CS}_{\lambda^*}(\Theta)\in\Omega^3(P)^G$ and
$\lambda^*(-,\Theta)\in(\mf{g}^\vee\otimes\Omega^1(P))^G$
have the following differentials
\begin{align} \label{dCS}
d_G\,\mathrm{CS}_{\lambda^*}(\Theta)
=\lambda^*(\Omega\wedge\Omega)
  -\lambda^*(-,d\Theta),\qquad
d_G\,\lambda^*(-,\Theta)
=\lambda^*(-,d\Theta)-\lambda^*.
\end{align}

{\it From (i) to (ii).}
Let $(H',h)$ be a pair as in (i).
First define a $G$-invariant $2$-form on $P$ by
\begin{align} \label{beta}
\beta(A^P,B^P)=h_A(B^P)-h_B(A^P),\qquad
\beta(A^P,\tau_i)=2h_A(\tau_i),\qquad
\beta(\tau_i,\tau_j)=0
\end{align}
for $A,B\in\mf{g}$ and $i,j=1,\ldots,d$.
Indeed, the $G$-equivariance of $h$ and (\ref{tau.brackets})
imply the $G$-invariance of $\beta$.
It follows from $d_G(H'-2h)=\lambda^*$ that
\begin{align} \label{iA.beta}
\iota_{A^P}\beta=2h_A-\lambda^*(A,\Theta)
\qquad\Rightarrow\qquad
d_G\beta=d\beta-2h+\lambda^*(-,\Theta).
\end{align}
Then define a $G$-invariant $3$-form on $P$ by
\begin{align} 
\label{basic.3form}
H=H'-d\beta+\mathrm{CS}_{\lambda^*}(\Theta).
\end{align}
It follows from $d_G(H'-2h)=\lambda^*$, (\ref{iA.beta}) and
(\ref{dCS}) that $d_G H=\lambda^*(\Omega\wedge\Omega)$.
Equivalently, $H$ is horizontal (hence basic) and satisfies
$dH=\lambda^*(\Omega\wedge\Omega)$.

{\it From (ii) to (i).}
Given $(H,\beta)$ as in (ii), define $h$ by (\ref{iA.beta})
and $H'$ by (\ref{basic.3form}).
\end{proof}

{\it Preparation.}
(i) Let $\beta\in\Omega^2(P)^G$ and $H\in\pi^*\Omega^3(M)$
be given by applying the construction in Lemma
\ref{lemma.basic.3form} to the data $(H',h)$ described in
\S\ref{sec.prin.bdl2} and \S\ref{CDOP.FLGaction}.
(ii) Define a $C^\infty(P)$-linear map 
$\Delta:\T(P)\rightarrow\Omega^1(P)$ as follows:
$\Delta(A^P)=-h_A$ for $A\in\mf{g}$ and
$\Delta(\tau_i)=-\frac{1}{2}\iota_{\tau_i}\beta$ for
$i=1,\ldots,d$.
By Lemma \ref{lemma.newVAoid}, $\Delta$ determines an isomorphism
$(\t{id},\Delta)$ from the vertex algebroid (\ref{CDOP.VAoid}) to
a new vertex algebroid
\begin{align} \label{CDOP.VAoidnew}
\big(C^\infty(P),\Omega^1(P),\T(P),
  \bullet,\{\;\},\{\;\}_\Omega\big).
\end{align}
This in turn induces an isomorphism from $\CDO(P)'$ to the vertex
algebra freely generated by (\ref{CDOP.VAoidnew}), which will be
(temporarily) denoted by $\CDO(P)$.
By composition, it has an inner $(\hat{\mf{g}}_\lambda,G)$-action
$V_\lambda(\mf{g})\hookrightarrow\CDO(P)$ whose associated Cartan
cochain is trivial, i.e.~it simply takes any
$A\in\mf{g}=V_\lambda(\mf{g})_1$ to $A^P$.

\begin{lemma} \label{lemma.CDOP.VAoidnew}
For $f\in C^\infty(P)$, $A,B\in\mf{g}$ and
$i,j=1,\ldots,d$, we have \\
\begin{tabular}{cl}
$\cdot$ &
$A^P\bullet f=\tau_i\bullet f=0$ \\
$\cdot$ &
$\{A^P,B^P\}=\lambda(A,B)$;\;
$\{A^P,\tau_i\}=0$;\;
$\{\tau_i,\tau_j\}=2\mathrm{Ric}_{ij}$ \\
$\cdot$ &
$\{A^P,B^P\}_\Omega=\{A^P,\tau_i\}_\Omega=0$;\;
$\{\tau_i,\tau_j\}_\Omega=d\,\mathrm{Ric}_{ij}
+\frac{1}{2}\iota_{\tau_i}\iota_{\tau_j}H
+\lambda^*(\Omega(\tau_i,\tau_j),\Theta)$
\end{tabular}
\end{lemma}

\begin{proof}
The structure of (\ref{CDOP.VAoidnew}) is determined by
the structure of (\ref{CDOP.VAoid}) (see Lemma
\ref{lemma.CDOP.VAoid}) together with the isomorphism
$(\t{id},\Delta):(\ref{CDOP.VAoid})\rightarrow
(\ref{CDOP.VAoidnew})$ (see Definition \ref{VAoid.mor}
and the definition of $\Delta$).
The claim for $\bullet$ is easy.
Recall the map $(i,h)$ in Proposition \ref{prop.lift.VA}.
Since $(\t{id},\Delta)\circ(i,h)=(i,0)$, we have
$\{A^P,B^P\}=\lambda(A,B)$ and $\{A^P,B^P\}_\Omega=0$.
The other values of $\{\;\}$ are computed as follows
\begin{align*}
\{A^P,\tau_i\}
&\textstyle
=\{A^P,\tau_i\}'+h_A(\tau_i)+\frac{1}{2}\beta(\tau_i,A^P)=0 \\
\{\tau_i,\tau_j\}\;
&\textstyle
=\{\tau_i,\tau_j\}'+\frac{1}{2}\beta(\tau_i,\tau_j)
  +\frac{1}{2}\beta(\tau_j,\tau_i)
=2\t{Ric}_{ij}
\end{align*}
using the definition of $\beta$ (\ref{beta}).
The other values of $\{\;\}_\Omega$ are computed as
follows
\begin{align*}
\{A^P,\tau_i\}_\Omega
&\textstyle
=\{A^P,\tau_i\}'_\Omega
+\frac{1}{2}L_{A^P}\iota_{\tau_i}\beta
-L_{\tau_i}h_A
+dh_A(\tau_i)
-\frac{1}{2}\iota_{[A^P,\tau_i]}\beta \\
&\textstyle
=\frac{1}{2}\iota_{A^P}\iota_{\tau_i}H'
-\iota_{\tau_i}dh_A=0 \\
\{\tau_i,\tau_j\}_\Omega\;
&\textstyle
=\{\tau_i,\tau_j\}'_\Omega
+\frac{1}{2}L_{\tau_i}\iota_{\tau_j}\beta
-\frac{1}{2}L_{\tau_j}\iota_{\tau_i}\beta
+\frac{1}{2}d\beta(\tau_i,\tau_j)
+h_{\Omega(\tau_i,\tau_j)}\\
&\textstyle
=d\,\t{Ric}_{ij}
+\frac{1}{2}\iota_{\tau_i}\iota_{\tau_j}H'
-\frac{1}{2}\iota_{\Omega(\tau_i,\tau_j)}\beta
-\frac{1}{2}\iota_{\tau_i}\iota_{\tau_j}d\beta
+\frac{1}{2}\iota_{\Omega(\tau_i,\tau_j)}\beta
+\frac{1}{2}\lambda^*(\Omega(\tau_i,\tau_j),\Theta) \\
&\textstyle
=d\,\t{Ric}_{ij}
+\frac{1}{2}\iota_{\tau_i}\iota_{\tau_j}H
+\lambda^*(\Omega(\tau_i,\tau_j),\Theta)
\end{align*}
using, in order:
the $G$-invariance of $\beta$;
the component of (\ref{CDOP.h}) in
$(\mf{g}^\vee\otimes\Omega^2(P))^G$;
the identity
\begin{align*}
L_{\tau_i}\iota_{\tau_j}
  -L_{\tau_j}\iota_{\tau_i}
  +d\iota_{\tau_j}\iota_{\tau_i}
=L_{\tau_i}\iota_{\tau_j}
  -\iota_{\tau_j}d\iota_{\tau_i}
=L_{\tau_i}\iota_{\tau_j}
  -\iota_{\tau_j}L_{\tau_i}
  +\iota_{\tau_j}\iota_{\tau_i}d
=\iota_{[\tau_i,\tau_j]}
  -\iota_{\tau_i}\iota_{\tau_j}d;
\end{align*}
the Lie brackets (\ref{tau.brackets});
equation (\ref{iA.beta});
and finally the definition of $H$ (\ref{basic.3form}).
\end{proof} 

{\it Remark.}
The bijection in Lemma \ref{lemma.basic.3form} also means
that \emph{any} $H\in\pi^*\Omega^3(M)$ satisfying
$dH=\lambda^*(\Omega\wedge\Omega)$ determines
a principal $(\hat{\mf{g}}_\lambda,G)$-algebra whose
associated vertex algebroid structure is as in Lemma
\ref{lemma.CDOP.VAoidnew}.

\begin{corollary} \label{cor.fund.normal.comm}
In the vertex algebra $\CDO(P)$ we have the following
normal-ordered expansions and commutation relations
\begin{align*}
&(fA^P)_n
=\sum_{k\geq 0}f_{n-k}A^P_k
+\sum_{k<0}A^P_k f_{n-k},\qquad
(f\tau_i)_n
=\sum_{k\geq 0}f_{n-k}\tau_{i,k}
+\sum_{k<0}\tau_{i,k}f_{n-k} \\
&[A^P_n,B^P_m]=[A,B]^P_{n+m}+n\lambda(A,B)\delta_{n+m,0},
\qquad
[A^P_n,\tau_{i,m}]=\rho(A)_{ji}\tau_{j,n+m}
  \ph{\rule[-0.08in]{.1pt}{.1pt}} \\
&\textstyle
[\tau_{i,n},\tau_{j,m}]
=-\Omega(\tau_i,\tau_j)^P_{n+m}
+(n-m)(\mathrm{Ric}_{ij})_{n+m}
+\Big(\frac{1}{2}\iota_{\tau_i}\iota_{\tau_j}H
  +\lambda^*(\Omega(\tau_i,\tau_j),\Theta)\Big)_{n+m}
\end{align*}
for $f\in C^\infty(P)$; $A,B\in\mf{g}$; $i,j=1,\ldots,d$;
and $n,m\in\bb{Z}$.
\end{corollary}

\begin{proof}
This follows immediately from Lemma \ref{lemma.CDOP.VAoidnew}
and Definition \ref{VAoid.VA}.
\end{proof}

\begin{lemma} \label{lemma.CDOP.conformal}
The vertex algebra $\CDO(P)$ has a conformal vector
of central charge $2\dim P$:
\begin{align*}
\textstyle
\nu=t^P_{a,-1}\Theta^a+\tau_{i,-1}\tau^i
  -\frac{1}{2}\lambda^*(\Theta_{-1}\Theta).
\end{align*}
Moreover, we have $\nu_1 A^P=\nu_1 \tau_i=0$ for
$A\in\mf{g}$ and $i=1,\ldots,d$.
\end{lemma}

\begin{proof}
By construction, the isomorphism from $\CDO(P)'$ to $\CDO(P)$
(see the paragraph preceding Lemma \ref{lemma.CDOP.VAoidnew})
takes the conformal vector $\nu'$ in (\ref{CDOP.conformal}) to
\begin{align*}
\nu
&\textstyle
=(t^P_a-h_{t_a})_{-1}\Theta^a
  +(\tau_i-\frac{1}{2}\iota_{\tau_i}\beta)_{-1}\tau^i \\
&\textstyle
=t^P_{a,-1}\Theta^a+\tau_{i,-1}\tau^i
  -\Big(h_{t_a}(t_b^P)\Theta^b
  +h_{t_a}(\tau_i)\tau^i\Big)_{-1}\Theta^a
  -\frac{1}{2}\Big(\beta(\tau_i,t^P_a)\Theta^a
  +\beta(\tau_i,\tau_j)\tau^j\Big)_{-1}\tau^i \\
&\textstyle
=t^P_{a,-1}\Theta^a+\tau_{i,-1}\tau^i
  -\frac{1}{2}\lambda^*(t_a,t_b)\Theta^b_{-1}\Theta^a
  -h_{t_a}(\tau_i)\tau^i_{-1}\Theta^a
  +h_{t_a}(\tau_i)\tau^i_{-1}\Theta^a+0 \\
&\textstyle
=t^P_{a,-1}\Theta^a+\tau_{i,-1}\tau^i
  -\frac{1}{2}\lambda^*(\Theta_{-1}\Theta)
\end{align*}
where we have used, in order:
the component of (\ref{CDOP.h}) in
$(\t{Sym}^2\mf{g}^\vee\otimes C^\infty(P))^G$;
the symmetry $\alpha_{-1}\alpha'=\alpha'_{-1}\alpha$
for $1$-forms $\alpha,\alpha'$;
and the definition of $\beta$ (\ref{beta}).
The said isomorphism also takes $\nu'_1(A^P+h_A)$ to
$\nu_1 A^P$ and $\nu'_1(\tau_i+\iota_{\tau_i}\beta)$ to
$\nu_1\tau_i$.
This proves our second claim in view of Lemma
\ref{lemma.CDOP.VAoid} and (\ref{CDO.L1}).
\end{proof}

The following summarizes all the discussions since
\S\ref{sec.prin.bdl2}.

\begin{theorem} \label{thm.CDOP} 
Suppose we have:
a principal $G$-bundle $\pi:P\rightarrow M$,
a representation $\rho:G\rightarrow SO(\bb{R}^d)$
and an isomorphism $P\times_\rho\bb{R}^d\cong TM$;
a connection $\Theta$ on $\pi$ that (for simplicity) induces
a torsion-free connection on $TM$;
an invariant symmetric bilinear form $\lambda$ on $\mf{g}$;
and a basic $3$-form $H$ on $P$ satisfying
$dH=\lambda^*(\Omega\wedge\Omega)$, where $\Omega$ is the
curvature of $\Theta$ and
$\lambda^*=\lambda+\Killing+\lambda_\rho$ (see
\S\ref{conventions}).

These data determine a vertex algebra $\CDO_{\Theta,H}(P)$
as follows.
There are generating fields
\begin{align*}
\Phi_f(z)\t{ for }f\in C^\infty(P),\qquad
\Phi_{A^P}(z)\t{ for }A\in\mf{g},\qquad
\Phi_{\tau_i}(z)\t{ for }i=1,\ldots,d
\end{align*}
(see \S\ref{sec.Gmfld} and \S\ref{sec.prin.bdl2})
whose weights are $0$, $1$ and $1$ respectively.
The vacuum is $1\in C^\infty(P)$.
The OPEs with $\Phi_f(z)$ have the following leading terms
\begin{align*}
\Phi_f(z)\Phi_g(w)\;\,
&=\Phi_{fg}(w)+\Phi_{fdg}(w)(z-w)+O((z-w)^2),\quad
g\in C^\infty(P)\ph{\Big(} \\
\Phi_{A^P}(z)\Phi_f(w)
&=\frac{\Phi_{A^P f}(w)}{z-w}+\Phi_{fA^P}(w)+O(z-w) \\
\Phi_{\tau_i}(z)\Phi_f(w)\;
&=\frac{\Phi_{\tau_i f}(w)}{z-w}+\Phi_{f\tau_i}(w)+O(z-w)
\end{align*}
where the three second terms (plus linearity) are meant to be
the definition of fields associated to arbitrary $1$-forms
and vector fields.
The OPEs of the other generating fields have the
following singular parts
\begin{align*}
\Phi_{A^P}(z)\Phi_{B^P}(w)
&\sim
  \frac{\lambda(A,B)}{(z-w)^2}
  +\frac{\Phi_{[A,B]^P}(w)}{z-w} \\
\Phi_{A^P}(z)\Phi_{\tau_i}(w)\;
&\sim
  \frac{\rho(A)_{ji}\Phi_{\tau_j}(w)}{z-w} \\
\Phi_{\tau_i}(z)\Phi_{\tau_j}(w)\;\,
&\sim
  \frac{2\Phi_{\mathrm{Ric}_{ij}}(\frac{z+w}{2})}{(z-w)^2}
  -\frac{\Phi_{\Omega(\tau_i,\tau_j)^P}(w)}{z-w}
  +\frac{\frac{1}{2}\Phi_{\iota_{\tau_i}\iota_{\tau_j}H}(w)
    +\Phi_{\lambda^*(\Omega(\tau_i,\tau_j),\Theta)}(w)}{z-w}
\end{align*}
where $\mathrm{Ric}_{ij}=\rho(\Omega(\tau_i,\tau_k))_{jk}$.
In particular, the fields $\Phi_{A^P}(z)$ for $A\in\mf{g}$ represent
an inner $(\hat{\mf{g}}_\lambda,G)$-action (see Definition
\ref{formalLG.action}).
Moreover, a Virasoro field of central charge $2\dim P$ is given by
\begin{align*}
\textstyle
:\Phi_{t^P_a}(z)\Phi_{\Theta^a}(z):
+:\Phi_{\tau_i}(z)\Phi_{\tau^i}(z):
-\frac{1}{2}\,\lambda^*(\Phi_{\Theta}(z),\Phi_{\Theta}(z))
\end{align*}
(see \S\ref{sec.prin.bdl2}) with respect to which all the
generating fields are primary.
$\qedsymbol$
\end{theorem}

{\it Remark.}
For convenience (and lack of imagination), we will
refer to $\CDO_{\Theta,H}(P)$ as
a {\bf principal frame $(\hat{\mf{g}}_\lambda,G)$-algebra}.
While this type of vertex algebras are particular examples of
algebras of CDOs, in \S\ref{CDOP.CDOM} we will see that all
algebras of CDOs can be recovered in a natural way from
principal frame $(\hat{\mf{g}}_\lambda,G)$-algebras with
$\lambda=-\Killing$.
Moreover, that is only a special case of a construction
associated to general principal
$(\hat{\mf{g}}_\lambda,G)$-algebras (to be introduced in
\S\ref{sec.ass}).
\footnote{
The author also thinks that, compared to an arbitrary algebra
of CDOs, a principal frame $(\hat{\mf{g}},G)$-algebra has
a more appealing description in terms of generating fields
and OPEs (compare Theorem \ref{thm.globalCDO} with the above
theorem).
}

\begin{subsec} \label{sec.invCDO.can}
{\bf The invariant subalgebra.}
This discussion is a more detailed version of
\S\ref{sec.invCDO} specifically for a principal frame
$(\hat{\mf{g}}_\lambda,G)$-algebra $\CDO_{\Theta,H}(P)$.
Consider the centralizer subalgebra 
\begin{align*}
\CDO_{\Theta,H}(P)^{\hat{\mf{g}}}
:=C\big(\CDO_{\Theta,H}(P),V_\lambda(\mf{g})\big).
\end{align*}
The weight-zero component is again $C^\infty(P)^G$.
Let us examine the weight-one component below.

By definition, $\CDO_{\Theta,H}(P)^{\hat{\mf{g}}}_1$ consists
of those elements
$\alpha+\ms{X}+\ms{Y}\in\Omega^1(P)\oplus\T_h(P)\oplus\T_v(P)$
that are annihilated by $A^P_0$ and $A^P_1$ for $A\in\mf{g}$.
Let $t_1,t_2,\cdots$ be a basis of $\mf{g}$ and
$t^1,t^2,\cdots$ the dual basis of $\mf{g}^\vee$.
Let us write $\ms{X}=\ms{X}^i\tau_i$ and
$\ms{Y}=\ms{Y}^a t^P_a$.
By Lemma \ref{lemma.CDOP.VAoidnew} and Definition
\ref{VAoid}, both $\{A^P,\ms{X}^i\tau_i\}_\Omega$ and
$\{A^P,\ms{Y}^a t^P_a\}_\Omega$ vanish.
Then by (\ref{freeVA.comm}) the first condition on
$\alpha+\ms{X}+\ms{Y}$ reads
\begin{align*}
A^P_0(\alpha+\ms{X}+\ms{Y})
=L_{A^P}\alpha+[A^P,\ms{X}]+[A^P,\ms{Y}]=0
\quad\;\;\t{for }A\in\mf{g}
\end{align*}
i.e.~$\alpha$, $\ms{X}$, $\ms{Y}$ are all $G$-invariant.
By Lemma \ref{lemma.CDOP.VAoidnew} and Definition
\ref{VAoid} again, we find that
\begin{align*}
\qquad\qquad
\{A^P,\ms{X}^i\tau_i\}\,
&=\rho(A)_{ji}\tau_j\ms{X}^i \\
\{A^P,\ms{Y}^a t^P_a\}
&=\lambda(A,t_a)\ms{Y}^a+[A,t_a]^P\ms{Y}^a \\
&=\lambda(A,t_a)\ms{Y}^a-t^a([[A,t_a],t_b])\ms{Y}^b
\qquad \because [A^P,\ms{Y}]=0 \\
&=(\lambda+\Killing)(A,\ms{Y})
\end{align*}
Then by (\ref{freeVA.comm}) the second condition on
$\alpha+\ms{X}+\ms{Y}$ reads
\begin{align*}
A^P_1(\alpha+\ms{X}+\ms{Y})
=\alpha(A^P)+\rho(A)_{ij}\tau_i\ms{X}^j
  +(\lambda+\Killing)(A,\ms{Y})
=0
\quad\;\;\t{for }A\in\mf{g}.
\end{align*}
The two conditions together imply that
$\alpha+(\tau_i\ms{X}^j)\rho(\Theta)_{ij}
+(\lambda+\Killing)(\Theta,\ms{Y})$ is $G$-invariant
and horizontal (i.e.~basic).
Therefore $\CDO_{\Theta,H}(P)^{\hat{\mf{g}}}_1$ is
the direct sum of the following subspaces of
$\Omega^1(P)\oplus\T(P)$:
\begin{align} \label{invCDOP.1}
\pi^*\Omega^1(M),\;\;
\big\{\ms{X}-(\tau_i\ms{X}^j)\rho(\Theta)_{ij}:
  \ms{X}\in\T_h(P)^G\big\},\;\;
\big\{\ms{Y}-(\lambda+\Killing)(\Theta,\ms{Y}):
  \ms{Y}\in\T_v(P)^G\big\}.
\end{align}
This provides an explicit description of the short exact
sequence (\ref{invCDO.1}) together with a splitting.
\end{subsec}

\newpage
\setcounter{equation}{0}
\section{Associated Modules and Algebras over CDOs}
\label{sec.ass}

Given a principal algebra $\CDO(P)$ in the sense of
\S\ref{sec.CDOP}, this section introduces a construction
of modules over the invariant subalgebra of $\CDO(P)$ using
semi-infinite cohomology.
This construction is analogous to that of associated vector
bundles.
Examples will be studied in \S\ref{CDOP.CDOM} and
\S\ref{sec.FL.spinor}.

The semi-infinite cohomology of a centrally extended loop
algebra can be defined using the Feigin complex, which
is a vertex algebraic analogue of the Chevalley-Eilenberg
complex, but with some important differences.
For more about semi-infinite cohomology, see
e.g.~\cite{Feigin,Voronov,BD}.

\begin{subsec} \label{sec.CE} 
{\bf The Chevalley-Eilenberg complex.}
Consider a finite-dimensional Lie algebra $\mf{g}$.
Let $t_1,t_2,\ldots$ be a basis of $\mf{g}$;
$t^1,t^2,\ldots$ the dual basis of $\mf{g}^\vee$; and
$\phi^1,\phi^2,\ldots$ the corresponding coordinates of
the supermanifold $\Pi\mf{g}$.
By definition $\mc{O}(\Pi\mf{g})=\wedge^*\mf{g}^\vee$ and
$\T(\Pi\mf{g})$ consists of the derivations on
$\wedge^*\mf{g}^\vee$.
Suppose
\begin{align*}
J,\,q\in\T(\Pi\mf{g}),\qquad
\theta\in\mc{O}(\Pi\mf{g})\otimes\mf{g}
\end{align*}
are the elements corresponding respectively to the exterior
degree on $\wedge^*\mf{g}^\vee$, the Chevalley-Eilenberg
differential on $\wedge^*\mf{g}^\vee$, and the Maurer-Cartan
form on $\mf{g}$;
or in coordinates,
\begin{align} \label{JQ.coor}
J=\phi^a\frac{\d}{\d\phi^a}\,,\qquad
q=-\frac{1}{2}\,t^a([t_b,t_c])\phi^b\phi^c\frac{\d}{\d\phi^a}\,,
  \qquad
\theta=\phi^a\otimes t_a.
\end{align}
Notice that $J=J\otimes 1$,\; $q=q\otimes 1$ and $\theta$
satisfy
\begin{align} \label{JQ.relations}
[J,q]=q,\qquad
J\theta=\theta,\qquad
[q,q]=0,\qquad
q\theta+\frac{1}{2}\,[\theta,\theta]=0.
\end{align}

Let $W$ be a $\mf{g}$-module.
If $J,q,\theta$ are regarded as operators on
$\mc{O}(\Pi\mf{g})\otimes W$ and $Q=q+\theta$, then by
(\ref{JQ.relations}) they satisfy $[J,Q]=Q$ and $[Q,Q]=0$.
The Chevalley-Eilenberg complex of $\mf{g}$ with
coefficients in $W$ can be written as
\begin{align*}
\Big(\mc{O}(\Pi\mf{g})\otimes W,\,J,\,Q\Big)
\end{align*}
where $J$ is the grading operator and $Q$ is
the differential.
\end{subsec}

\begin{subsec} \label{sec.Feigin}
{\bf The Feigin complex.}
Given any invariant symmetric bilinear form $\lambda$ on
$\mf{g}$, recall the centrally extended loop algebra
$\hat{\mf{g}}_\lambda$ (see \S\ref{sec.Gmfld}) as well as
the vertex algebra $V_\lambda(\mf{g})$ (see Example
\ref{VAoid.Lie}).
Also consider the algebra of CDOs $\CDO(\Pi\mf{g})$, which
is a fermionic version of \S\ref{algCDO}.
\cite{myCDO}.
Now regard
\begin{align*}
J,q,\theta\t{ as elements of }
\CDO(\Pi\mf{g})\otimes V_\lambda(\mf{g})
\t{ of weight }1.
\end{align*}
Here are some computations with these elements:
\begin{align*}
&J_0 q
=[J,q]+\{J,q\}_\Omega=q+0=q \\
&J_0\theta
=(J_0\otimes 1)(\phi^a\otimes t_a)
=J\phi^a\otimes t_a=\phi^a\otimes t_a=\theta \\
&q_0 q
=[q,q]+\{q,q\}_\Omega
=0-t^a([t_b,t_c])\,t^b([t_a,t_d])\,\phi^d d\phi^c
=-\Killing(t_c,t_d)\,\phi^d d\phi^c
\quad(\t{see }\S\ref{conventions}) \\
&\hspace{-0.1in}
\left.\begin{array}{l}
q_0\theta=(q_0\otimes 1)(\phi^a\otimes t_a)
  \rule[-0.1in]{0in}{0in} \\
\theta_0 q=(\phi^a_1\otimes t_{a,-1})(q\otimes 1)
\end{array}\right\}
\textstyle
=q\phi^a\otimes t_a
=-\frac{1}{2}\phi^b\phi^c\otimes[t_b,t_c] \\
&\theta_0\theta
=(\phi^a_0\otimes t_{a,0}+\phi^a_{-1}\otimes t_{a,1})
  (\phi^b\otimes t_b)
=\phi^a\phi^b\otimes[t_a,t_b]
-\lambda(t_a,t_b)\,\phi^b d\phi^a
\end{align*}
using (\ref{freeVA.comm}), (\ref{JQ.coor}) and the super
version of (\ref{CDO.VAoid456}).
It follows that $J$ and $Q=q+\theta$ satisfy
\begin{align*}
&
&&J_0 Q=Q,
&&Q_0 Q=-(\Killing+\lambda)(t_a,t_b)\,\phi^a d\phi^b & \\
&\Rightarrow
&&[J_0,Q_0]=Q_0,
&&[Q_0,Q_0]=(\Killing+\lambda)(t_a,t_b)\cdot
  \sum_{n\in\bb{Z}}n\phi^a_{-n}\phi^b_n &
\end{align*}
In particular, $Q_0^2=0$ if and only if $\lambda=-\Killing$.

Let $W$ be a $\hat{\mf{g}}_{-\Killing}$-module.
Regard $J$ and $Q$ as elements of
$\CDO(\Pi\mf{g})\otimes V_{-\Killing}(\mf{g})$.
The Feigin complex of $\hat{\mf{g}}_{-\Killing}$ with
coefficients in $W$ is
\begin{align*}
\Big(\CDO(\Pi\mf{g})\otimes W,\,J_0,\,Q_0\Big)
\end{align*}
where $J_0$ is the grading operator and
$Q_0$ is the differential.
\end{subsec}

\begin{subsec} \label{sec.semiinf}
{\bf Semi-infinite cohomology.}
For any $\hat{\mf{g}}_{-\Killing}$-module $W$ as above, let
\begin{align} \label{semiinf}
H^{\frac{\infty}{2}+*}(\hat{\mf{g}}_{-\Killing},W)
:=H^*\Big(\CDO(\Pi\mf{g})\otimes W,\,Q_0\Big).
\end{align}
Notice that if $W$ is a vertex algebra and the
$\hat{\mf{g}}_{-\Killing}$-action is inner, i.e.~induced by
a map of vertex algebras $V_{-\Killing}(\mf{g})\rightarrow W$,
then (\ref{semiinf}) has the structure of a $\bb{Z}$-graded
vertex algebra.
(In this case, $J$ and $Q$ will also denote their images in
$\CDO(\Pi\mf{g})\otimes W$.)
\end{subsec}

{\it Remarks.}
(i) Unlike Lie algebra cohomology, the grading on
semi-infinite cohomology is neither bounded above nor below.
More precisely, the restriction of $J_0$ to
$\CDO(\Pi\mf{g})_k\otimes W$ takes values between $-k$
and $\dim\mf{g}+k$.
(ii) This is only a special case of semi-infinite cohomology.
For expositions in more general settings, see the references
mentioned above.

\begin{lemma}  \label{lemma.semiinf.ops}
Let $W$ be a $\hat{\mf{g}}_{-\Killing}$-module.

(a) The $\hat{\mf{g}}_{-\Killing}$-invariant operators on
$W$ induce grading-preserving operators on (\ref{semiinf}).
Moreover, if $W$ is a vertex algebra and its
$\hat{\mf{g}}_{-\Killing}$-action is inner, then there
is a map of vertex algebras from the centralizer subalgebra
$W^{\hat{\mf{g}}}=C(W,V_{-\Killing}(\mf{g}))$
\cite{Kac,FB-Z} to the zeroth gradation of  (\ref{semiinf}).

(b) If there is a Virasoro action on $W$ of central charge $c$
such that the $\hat{\mf{g}}_{-\Killing}$-action is primary,
then it induces a grading-preserving Virasoro action on
(\ref{semiinf}) of central charge $c-2\dim\mf{g}$.
Moreover, if $W$ is a conformal vertex algebra and its
$\hat{\mf{g}}_{-\Killing}$-action is inner and primary, then
(\ref{semiinf}) is also a conformal vertex
algebra with a conformal vector in the zeroth gradation.
\end{lemma}

\begin{proof}
(a) For the first claim, suppose $U\in\t{End}\,W$ is
$\hat{\mf{g}}_{-\Killing}$-invariant.
Since on $\CDO(\Pi\mf{g})\otimes W$ we have
\begin{align*}
[Q_0,1\otimes U]
=[q_0\otimes 1+\phi^a_{-n}\otimes t_{a,n},\,1\otimes U]
=\phi^a_{-n}\otimes[t_{a,n},U]
=0,
\end{align*}
the operator $1\otimes U$ is well-defined on (\ref{semiinf}).
Clearly it preserves the grading.
For the second claim, let $u\in W^{\hat{\mf{g}}}$,
i.e.~$u\in W$ such that $A_n u=0$ for $A\in\mf{g}$ and $n\geq 0$.
Since in $\CDO(\Pi\mf{g})\otimes W$ we have
\begin{align*}
Q_0(\mb{1}\otimes u)
=(q_0\otimes 1+\phi^a_{-n}\otimes t_{a,n})(\mb{1}\otimes u)
=\sum_{n\geq 0}\phi^a_{-n}\mb{1}\otimes t_{a,n}u
=0,
\end{align*}
the element $\mb{1}\otimes u$ represents a class
$[\mb{1}\otimes u]$ in (\ref{semiinf}).
Clearly $u\mapsto[\mb{1}\otimes u]$ is a map of
vertex algebras and the image is contained in the zeroth
gradation.

(b) The graded vertex algebra $\CDO(\Pi\mf{g})$ has a conformal
vector $\nu^{\Pi\mf{g}}$ of central charge $-2\dim\mf{g}$
(see \S\ref{sec.ass.C});
denote its Virasoro operators by $L^{\Pi\mf{g}}_n$, $n\in\bb{Z}$.
Since $q$ and $\phi^a$ are primary, we have
\begin{align*}
[L^{\Pi\mf{g}}_n,q_0]=0,\qquad
[L^{\Pi\mf{g}}_n,\phi^a_m]=-(n+m)\phi^a_{n+m},\qquad
n,m\in\bb{Z}.
\end{align*}
Since $\nu^{\Pi\mf{g}}$ belongs to the zeroth gradation,
every $L^{\Pi\mf{g}}_n$ preserves the grading.

For the first claim, suppose $L^W_n\in\t{End}\,W$ for
$n\in\bb{Z}$ define a Virasoro action of central charge $c$
and satisfy $[L^W_n,A_m]=-mA_{n+m}$ for $A\in\mf{g}$ and
$n,m\in\bb{Z}$.
Then $L^{\Pi\mf{g}}_n\otimes 1+1\otimes L^W_n$ for
$n\in\bb{Z}$ define a grading-preserving Virasoro action on
$\CDO(\Pi\mf{g})\otimes W$ of central charge $c-2\dim\mf{g}$.
Also, since
\begin{align*}
[Q_0,\,L^{\Pi\mf{g}}_n\otimes 1+1\otimes L^W_n]
&=[q_0,L^{\Pi\mf{g}}_n]\otimes 1
  +[\phi^a_{-m},L^{\Pi\mf{g}}_n]\otimes t_{a,m}
  +\phi^a_{-m}\otimes[t_{a,m},L^W_n] \\
&=0+(n-m)\phi^a_{n-m}\otimes t_{a,m}
  +m\phi^a_{-m}\otimes t_{a,n+m} \\
&=0
\end{align*}
the said Virasoro action is well-defined on
(\ref{semiinf}) as well.

For the second claim, suppose $\nu^W\in W$ is
a conformal vector whose Virasoro operators $L^W_n$,
$n\in\bb{Z}$, satisfy $[L^W_n,A_m]=-mA_{n+m}$ for
$A\in\mf{g}$ and $n,m\in\bb{Z}$.
Then $\nu^{\Pi\mf{g}}\otimes\mb{1}+\mb{1}\otimes\nu^W$
is a conformal vector of $\CDO(\Pi\mf{g})\otimes W$
belonging to the zeroth gradation.
Also, since we have
\begin{align*}
Q_0(\nu^{\Pi\mf{g}}\otimes\mb{1}
  +\mb{1}\otimes\nu^W)
&=Q_0(L^{\Pi\mf{g}}_{-2}\otimes 1
  +1\otimes L^W_{-2})(\mb{1}\otimes\mb{1}) \\
&=\big[Q_0,\,L^{\Pi\mf{g}}_{-2}\otimes 1
  +1\otimes L^W_{-2}\big](\mb{1}\otimes\mb{1}) \\
&=0
\end{align*}
by the previous computation,
$\nu^{\Pi\mf{g}}\otimes\mb{1}+\mb{1}\otimes\nu^W$
represents a conformal vector of (\ref{semiinf}) as well.
\end{proof}

The main purpose of this section is to introduce
the following construction.
As before, $G$ is a compact connected Lie group and
$\mf{g}$ its Lie algebra.
Recall Definition \ref{formalLG.action} and
Corollary \ref{cor.formalLG.P}.

\begin{defn} \label{defn.ass}
Let $\pi:P\rightarrow M$ be a smooth principal $G$-bundle
and $\lambda,\lambda'$ invariant symmetric bilinear forms
on $\mf{g}$, such that $\lambda+\lambda'=-\Killing$ and
$\lambda'(P)=-8\pi^2 p_1(M)$.
For any principal $(\hat{\mf{g}}_\lambda,G)$-algebra $\CDO(P)$
and positive-energy $(\hat{\mf{g}}_{\lambda'},G)$-module $W$,
we define
\begin{align*}
\ass(\pi,W)
:=H^{\frac{\infty}{2}+0}\big(\hat{\mf{g}}_{-\Killing},
  \CDO(P)\otimes W\big).
\end{align*}
Notice that if $W$ is a vertex algebra and its
$(\hat{\mf{g}}_{\lambda'},G)$-action is inner, then $\ass(\pi,W)$
also has the structure of a vertex algebra (see
\S\ref{sec.semiinf}).
\end{defn}

{\it Remarks.}
(i) The positive-energy condition means that there is
a diagonalizable operator $L^W_0$ on $W$ such that
$[L^W_0,A_n]=-nA_n$ for $A\in\mf{g}$, $n\in\bb{Z}$, and the
eigenvalues of $L^W_0$ are bounded below.
Whenever $W$ admits a Virasoro action, it will be understood that
$L^W_0$ coincides with the zeroth Virasoro operator.
For consistency, the eigenvalues of $L^W_0$ will also be called
weights.
Let $L^{\Pi\mf{g}}_0$ (resp.~$L^P_0$) be the weight
operator on $\CDO(\Pi\mf{g})$ (resp.~$\CDO(P)$).
Notice that $L^{\Pi\mf{g}}_0+L^P_0+L^W_0$ commutes with $Q_0$
(see the proof of Lemma \ref{lemma.semiinf.ops}b), so that
$\ass(\pi,W)$ inherits the notion of weights.

(ii) The above definition still makes sense without
the integrability and/or the positive-energy conditions.
On the other hand, under these two conditions, we can
say that the lowest-weight component $W_0$ of $W$ is
a $G$-representation and then the lowest-weight component of
$\ass(\pi,W)$ is $(C^\infty(P)\otimes W_0)^G$,
i.e. the space of sections of the associated vector
bundle $P\times_G W_0\rightarrow M$.

\begin{lemma} \label{lemma.ass}
Consider again the data in Definition \ref{defn.ass}.

(a) For any $W$ as described, $\ass(\pi,W)$ is a module
over $\CDO(P)^{\hat{\mf{g}}}$ (see \S\ref{sec.invCDO}).
Moreover, if $W$ is a vertex algebra and its
$(\hat{\mf{g}}_{\lambda'},G)$-action is inner, then there is
a map of vertex algebras
$\CDO(P)^{\hat{\mf{g}}}\rightarrow\ass(\pi,W)$.

(b) If there is a Virasoro action on $W$ of central charge
$c$ such that the $(\hat{\mf{g}}_{\lambda'},G)$-action
is primary, then it induces a Virasoro action on $\ass(\pi,W)$
of central charge $c+2\dim M$.
Moreover, if $W$ is in fact a conformal vertex algebra such
that its $(\hat{\mf{g}}_{\lambda'},G)$-action is inner and
primary, then $\ass(\pi,W)$ is also a conformal vertex
algebra.
\end{lemma}

\begin{proof}
This follows from Lemma \ref{lemma.semiinf.ops}, together
with the fact that $\CDO(P)$ is a conformal vertex algebra
with central charge $2\dim P$ and a primary inner
$(\hat{\mf{g}}_\lambda,G)$-action (see Theorem
\ref{thm.formalLG.action}).
\end{proof}

\newpage
\setcounter{equation}{0}
\section{Example: Recovering Algebras of CDOs}
\label{CDOP.CDOM}

In this section, we analyze the zeroth semi-infinite
cohomology of a particular type of principal algebra (in the
sense of \S\ref{sec.CDOP}) and as a result identify it with
an algebra of CDOs.
This provides a more conceptual description of a general
algebra of CDOs.

\begin{subsec} \label{sec.ass.C}
{\bf Goal:~a new description of algebras of CDOs.}
Consider the special case of Definition \ref{defn.ass}
associated to a principal frame
$(\hat{\mf{g}}_{-\Killing},G)$-algebra
$\CDO_{\Theta,H}(P)$ and the trivial
$(\hat{\mf{g}}_0,G)$-module $\bb{C}$:
\begin{align} \label{ass.C}
\ass(\pi,\bb{C})
=H^{\frac{\infty}{2}+0}\big(\hat{\mf{g}}_{-\Killing},
  \CDO_{\Theta,H}(P)\big)
=H^0\Big(\CDO(\Pi\mf{g})\otimes\CDO_{\Theta,H}(P),
  Q_0\Big).
\end{align}
By Lemma \ref{lemma.ass}, $\ass(\pi,\bb{C})$ is a conformal
vertex algebra of central charge $2\dim M$.
For convenience, let us recall some of the data involved.
\begin{itemize}
\item[$\centerdot$]
Let $(t_1,t_2,\cdots)$ be a basis of $\mf{g}$;
$(t^1,t^2,\cdots)$ the dual basis of $\mf{g}^\vee$;
$(\phi^1,\phi^2,\cdots)$ the corresponding coordinates
of the supermanifold $\Pi\mf{g}$; and
$(\d_1,\d_2,\cdots)$ their coordinate vector fields.
\vspace{-0.05in}
\item[$\centerdot$]
For the detailed definition of the vertex superalgebra
$\CDO(\Pi\mf{g})$, see e.g.~\cite{myCDO}.
Let us mention that, as a fermionic analogue of \S\ref{algCDO},
it is generated by such elements as $\phi^1,\phi^2,\cdots$ and
$\d_1,\d_2,\cdots$, and a conformal vector of central charge
$-2\dim\mf{g}$ is given by
\begin{align} \label{oddg.conformal}
\nu^{\Pi\mf{g}}=-\d_{a,-1}d\phi^a.
\end{align}
\vspace{-0.2in}
\item[$\centerdot$]
For the detailed definition of the vertex algebra
$\CDO_{\Theta,H}(P)$, see Theorem \ref{thm.CDOP} with
$\lambda=-\Killing$ in mind.
Let us mention that it is defined using:
a smooth principal $G$-bundle $\pi:P\rightarrow M$ with
an isomorphism $P\times_\rho\bb{R}^d\cong TM$ for some
representation $\rho:G\rightarrow SO(\bb{R}^d)$;
a connection $\Theta$ on $\pi$ that induces the Levi-Civita
connection on $TM$;
and a basic $3$-form $H$ on $P$ that satisfies
$dH=\lambda_\rho(\Omega\wedge\Omega)$, where $\Omega$ is the
curvature of $\Theta$.
Also, it is generated by such elements as $f\in C^\infty(P)$,
$A^P\in\T_v(P)$ for $A\in\mf{g}$ and $\tau_i\in\T_h(P)$ for
$i=1,\ldots,d$, and a conformal vector of central charge
$2\dim P$ is given by
\begin{align*}
\textstyle
\nu^P=t^P_{a,-1}\Theta^a+\tau_{i,-1}\tau^i
  -\frac{1}{2}\lambda_\rho(\Theta_{-1}\Theta).
\end{align*}
For the meaning of various notations, see \S\ref{conventions},
\S\ref{sec.Gmfld} and \S\ref{sec.prin.bdl2}.
\vspace{-0.05in}
\item[$\centerdot$]
For details of the Feigin complex, see \S\ref{sec.Feigin}.
Let us mention that the grading operator $J_0$ is
determined by $J_0(\phi^a\otimes u)=\phi^a\otimes u$ and
$J_0(\d_a\otimes u)=-\d_a\otimes u$ for any $u$;
and the differential $Q_0=(q+\theta)_0$ is induced by
the elements
\begin{align} \label{Q.theta}
\textstyle
q=q\otimes\mb{1}
=-\frac{1}{2}\,t^a([t_b,t_c])\phi^b\phi^c\d_a\otimes\mb{1},
\qquad
\theta=\phi^a\otimes t_a^P.
\end{align}
The conformal vector
$\nu^{\Pi\mf{g}}\otimes\mb{1}+\mb{1}\otimes\nu^P$
belongs to the zeroth gradation and is
$Q_0$-closed.
\vspace{-0.05in}
\end{itemize}
By the assumption on $\pi$, $M$ is an oriented
Riemannian manifold.
Let $\nabla$ be the Levi-Civita connection and $R$
the Riemannian curvature of $M$.
The assumption on $H$ can then be written as
$dH=\Tr(R\wedge R)$.
By Theorem \ref{thm.globalCDO}, the data $(\nabla,H)$
determine an algebra of CDOs $\CDO_{\nabla,H}(M)$
with a conformal vector $\nu=\nu^0$ of central
charge $2d=2\dim M$.
The goal of this section is to prove that
\begin{align*}
\ass(\pi,\bb{C})
\cong\CDO_{\nabla,H}(M)
\end{align*}
as conformal vertex algebras.
In fact, we will analyze $\ass(\pi,\bb{C})$ without using any
prior knowledge of $\CDO_{\nabla,H}(M)$, and effectively
rediscover the latter in the end.
\end{subsec}

Throughout this section we identify $\Omega^*(M)$ with
the basic subspace of $\Omega^*(P)$.

\begin{subsec} \label{sec.ass.C.0}
{\bf The component of weight zero.}
Consider the Feigin complex that defines (\ref{ass.C}):
\begin{align} \label{Feigin.C}
\Big(\CDO(\Pi\mf{g})\otimes\CDO_{\Theta,H}(P),\,
Q_0\Big).
\end{align}
Since its weight-zero component is simply the
Chevalley-Eilenberg complex
\begin{align*}
\Big(\mc{O}(\Pi\mf{g})\otimes C^\infty(P),\,
  Q_0\Big)
\end{align*}
(see \S\ref{sec.CE}), the weight-zero component of (\ref{ass.C}) is
\begin{align} \label{ass.C.0}
\ass(\pi,\bb{C})_0
=H^0\big(\mf{g},C^\infty(P)\big)
=C^\infty(P)^G
=C^\infty(M).
\end{align}
Understanding the rest of $\ass(\pi,\bb{C})$ requires more work.
\end{subsec}

\begin{lemma} \label{lemma.G}
The element $\gamma=\d_a\otimes\Theta^a$ in
$\CDO(\Pi\mf{g})_1\otimes\CDO_{\Theta,H}(P)_1$ satisfies
\begin{align*}
Q_0\gamma
&=\nu^{\Pi\mf{g}}\otimes\mb{1}
  +\mb{1}\otimes t^P_{a,-1}\Theta^a,\hspace{-0.8in}
&&\;\;\gamma_0\gamma=0\\
\Rightarrow\qquad
[Q_0,\gamma_0]
&=L^{\Pi\mf{g}}_0\otimes 1
  +1\otimes(t^P_{a,-1}\Theta^a)_0,\hspace{-0.8in}
&&[\gamma_0,\gamma_0]=0 \qquad\qquad
\end{align*}
\end{lemma}

\begin{proof}
Recall the element $Q=q+\theta$ from (\ref{Q.theta}).
Let us first write
\begin{align*}
Q_0\gamma
=\big(q_0\otimes 1+\phi^a_{-1}\otimes t^P_{a,1}
  +\phi^a_0\otimes t^P_{a,0}
  +\phi^a_1\otimes t^P_{a,-1}\big)(\d_b\otimes\Theta^b).
\end{align*}
Then each of these terms is computed as follows:
\begin{align*}
&q_0\d_b\otimes\Theta^b
=\big([q,\d_b]+\{q,\d_b\}_\Omega\big)\otimes\Theta^b
=-t^c([t_b,t_d])\phi^d\d_c\otimes\Theta^b \\
&\phi^a_{-1}\d_b\otimes t^P_{a,1}\Theta^b
=-\d_{b,-1}d\phi^a\otimes\Theta^b(t^P_a)
=-\d_{a,-1}d\phi^a\otimes\mb{1} \\
&\phi^a_0\d_b\otimes t^P_{a,0}\Theta^b
=\phi^a\d_b\otimes L_{t^P_a}\Theta^b
=-\phi^a\d_b\otimes t^b([t_a,t_c])\Theta^c \\
&\phi^a_1\d_b\otimes t^P_{a,-1}\Theta^b
=\mb{1}\otimes t^P_{a,-1}\Theta^a
\end{align*}
using (\ref{freeVA.comm}), the super version of
(\ref{Weyl})--(\ref{CDO.VAoid456}) and the $G$-invariance of
$\Theta=\Theta^a\otimes t_a$.
In view of (\ref{oddg.conformal}), this proves
the claimed expression for $Q_0\gamma$.
On the other hand, we have
\begin{align*}
\gamma_0\gamma
=(\d_{a,1}\otimes\Theta^a_{-1})(\d_b\otimes\Theta^b)
=\d_{a,1}\d_b\otimes\Theta^a_{-1}\Theta^b
=0
\end{align*}
by the super version of (\ref{Weyl}).
\end{proof}

\begin{lemma} \label{lemma.Feigin.C.1}
The weight-one component of the Feigin complex (\ref{Feigin.C})
is quasi-isomorphic to a subcomplex
\begin{align*}
\Big(\mc{O}(\Pi\mf{g})\otimes\CDO_h(P)_1,
  \,Q_0\Big),
\end{align*}
where $\CDO_h(P)_1$ is the direct sum of the following subspaces
of $\Omega^1(P)\oplus\T(P)$:
\begin{align*}
\Omega^1_h(P),\qquad
\big\{\ms{X}-(\tau_i\ms{X}^j)\rho(\Theta)_{ij}:
  \ms{X}=\ms{X}^i\tau_i\in\T_h(P)\big\}.
\end{align*}
\end{lemma}

\begin{proof}
First we compute the operator $(t^P_{a,-1}\Theta^a)_0$ on
$\CDO_{\Theta,H}(P)_k$, $k=0,1$.
Notice that (\ref{freeVA.comm})--(\ref{freeVA.NOP})
will be used repeatedly.
For $f\in C^\infty(P)$ and $\alpha\in\Omega^1(P)$, we have
\begin{align*}
&(t^P_{a,-1}\Theta^a)_0 f
=\Theta^a_0\,t^P_{a,0}f
=0 \\
&(t^P_{a,-1}\Theta^a)_0\alpha
=\big(\Theta^a_0\,t^P_{a,0}+\Theta^a_{-1}t^P_{a,1}\big)\alpha
=0+\alpha(t^P_a)\Theta^a
=\alpha_v
\end{align*}
where $\alpha_v$ means the vertical part of $\alpha$.
For $\ms{X}=\ms{X}^i\tau_i\in\T_h(P)$, we have
\begin{align*}
(t^P_{a,-1}\Theta^a)_0\ms{X}
&=\big(t^P_{a,-1}\Theta^a_1+\Theta^a_0\,t^P_{a,0}
  +\Theta^a_{-1}t^P_{a,1}\big)\ms{X} \\
&=0-\iota_{[t^P_a,\ms{X}]}d\Theta^a
  +\rho(t_a)_{ij}(\tau_i\ms{X}^j)\Theta^a \\
&=-\iota_{[t^P_a,\ms{X}]}\Omega^a
  +(\tau_i\ms{X}^j)\rho(\Theta)_{ij}
\end{align*}
using Lemma \ref{lemma.1form.0mode} and the computation
of $\{A^P,\ms{X}\}$ in \S\ref{sec.invCDO.can}.
For $\ms{Y}=\ms{Y}^a t^P_a\in\T_v(P)$, we have
\begin{align*}
(t^P_{a,-1}\Theta^a)_0\ms{Y}
&=\big(t^P_{a,-1}\Theta^a_1
  +\Theta^a_0\,t^P_{a,0}
  +\Theta^a_{-1}t^P_{a,1}\big)\ms{Y} \\
&=\Theta^a(\ms{Y})t^P_a
  -\iota_{[t^P_a,\ms{Y}]}d\Theta^a
  +\big([t_a,t_b]^P\ms{Y}^b-\Killing(t_a,\ms{Y})\big)
    \Theta^a \\
&=\ms{Y}
  -\Theta^b([t^P_a,\ms{Y}])\iota_{t^P_b}d\Theta^a
  +\Theta^b\big([[t_a,t_b]^P,\ms{Y}]\big)\Theta^a \\
&=\ms{Y}
  +\Theta^b([t^P_a,\ms{Y}])t^a([t_b,t_c])\Theta^c
  +t^c([t_a,t_b])\Theta^b([t^P_c,\ms{Y}])\Theta^a \\
&=\ms{Y}
\end{align*}
using Lemma \ref{lemma.CDOP.VAoidnew},
Lemma \ref{lemma.1form.0mode} and
the computation of $\{A^P,\ms{Y}\}$ in
\S\ref{sec.invCDO.can} (with $\lambda=-\Killing$ but
without the assumption that $\ms{Y}$ is $G$-invariant).

Now consider the operator on the Feigin complex
(\ref{Feigin.C}) given by either of the following
expressions, which are all equal by Lemma \ref{lemma.G}:
\begin{align*}
\big[Q_0,\,\gamma_0 Q_0 \gamma_0\big]
=\big[Q_0,\gamma_0\big]^2
=\Big(L_0^{\Pi\mf{g}}\otimes 1
  +1\otimes(t^P_{a,-1}\Theta^a)_0\Big)^2.
\end{align*}
Let $e$ denote its restriction to the weight-one component.
By the calculations above, we have
\begin{align*}
\begin{array}{ll}
e(u_1\otimes f)=u_1\otimes f, &
u_1\in\CDO(\Pi\mf{g})_1,\; f\in C^\infty(P) \\
e(u_0\otimes\alpha)=u_0\otimes\alpha_v, &
u_0\in\mc{O}(\Pi\mf{g}),\;\alpha\in\Omega^1(P) \\
e(u_0\otimes\ms{X})=u_0\otimes(\tau_i\ms{X}^j)
  \rho(\Theta)_{ij},\quad &
u_0\in\mc{O}(\Pi\mf{g}),\;\ms{X}\in\T_h(P) \\
e(u_0\otimes\ms{Y})=u_0\otimes\ms{Y}, &
u_0\in\mc{O}(\Pi\mf{g}),\;\ms{Y}\in\T_v(P)
\end{array}
\end{align*}
Notice that $e$ is a null-homotopic idempotent.
Therefore the image of $1-e_1$ is a quasi-isomorphic subcomplex.
This is the subcomplex stated in the lemma.
\end{proof}

\begin{subsec} \label{sec.ass.C.1}
{\bf The component of weight one.}
By Lemma \ref{lemma.Feigin.C.1}, the weight-one component
of (\ref{ass.C}) is
\begin{align}
\ass(\pi,\bb{C})_1
&\cong H^0\big(\mf{g},\CDO_h(P)_1\big)
=\CDO_h(P)_1^G \nonumber \\
&=\Omega^1_h(P)^G\oplus
  \big\{\ms{X}-(\tau_i\ms{X}^j)\rho(\Theta)_{ij}:
  \ms{X}\in\T_h(P)^G\big\} \nonumber \\
&=\Omega^1(M)\oplus
  \big\{\wt X-(\tau_i\wt X^j)\rho(\Theta)_{ij}:
  X\in\T(M)\big\} \label{ass.C.1}
\end{align}
where $\wt X\in\T_h(P)^G$ denotes the horizontal lift of
any $X\in\T(M)$.
\footnote{
By Lemma \ref{lemma.ass}, there is a map of vertex algebras
$\CDO_{\Theta,H}(P)^{\hat{\mf{g}}}\rightarrow\ass(\pi,\bb{C})$.
Comparing (\ref{invCDOP.1}) and (\ref{ass.C.1}), we see
that the weight-one component of the said map is surjective
and its kernel is $\T_v(P)^G$ (since $\lambda=-\Killing$).
}
Notice that by (\ref{tau.brackets}) the $G$-invariance
of $\wt X=\wt X^i\tau_i$ means that
\begin{align} \label{lift.inv}
A^P\wt X^i=-\rho(A)_{ij}\wt X^j\quad\t{for }A\in\mf{g}.
\end{align}
Also notice that the operator
$\nabla X\in\Gamma(\t{End}\,TM)$ lifts horizontally to
\begin{align} \label{DX.lift}
\wt{\nabla X}
=\big(d\wt X^j+\rho(\Theta)_{jk}\wt X^k\big)\otimes\tau_j
=(\tau_i\wt X^j)\,\tau^i\otimes\tau_j
\end{align}
where the first equality simply expresses the relation
between $\nabla$ and $\Theta$, and the second equality
follows from (\ref{lift.inv}).
The odd-looking term in (\ref{ass.C.1}) can be given
a global expression using (\ref{DX.lift}).
\end{subsec}

\begin{subsec} \label{sec.ass.C.VAoid}
{\bf The associated vertex algebroid.}
In view of (\ref{ass.C.0}) and (\ref{ass.C.1}),
the extended Lie algebroid associated to $\ass(\pi,\bb{C})$
is $(C^\infty(M),\Omega^1(M),\T(M))$ with the
usual structure maps (see \S\ref{VA.VAoid}).
More precisely, we are identifying each $X\in\T(M)$ with
\begin{align*}
\ell X:=\,\t{class of }\wt X-(\tau_i\wt X^j)\rho(\Theta)_{ij}
\;\in\ass(\pi,\bb{C})_1.
\end{align*}
\footnote{
For example, it follows from (\ref{horlift.nonint}) that
$[\ell X,\ell Y]=\ell[X,Y]$, i.e.~the Lie bracket in the
said extended Lie algebroid indeed agrees with the usual
Lie bracket on $\T(M)$.
}
Consider the vertex algebroid associated to
$\ass(\pi,\bb{C})$ (see \S\ref{VA.VAoid} again):
\begin{align*}
\big(
  C^\infty(M),\Omega^1(M),\T(M),
  \bullet,\{\;\},\{\;\}_\Omega
\big).
\end{align*}
Explicit expressions of $\bullet,\{\;\},\{\;\}_\Omega$
are given in the lemma below.
Let $V^f$ be the vertex algebra freely generated by
this vertex algebroid, and
\begin{align*}
\Psi:V^f\rightarrow\ass(\pi,\bb{C})
\end{align*}
the resulting universal map (see \S\ref{VAoid.VA}).
By construction, $\Psi$ is an isomorphism in the two
lowest weights.
In fact, we will see that $\Psi$ is an isomorphism in all
weights.
\end{subsec}

\begin{prop} \label{prop.ass.C.VAoid}
For $f\in C^\infty(M)$ and $X,Y\in\T(M)$, we have
\begin{align*}
\ell X\bullet f \quad
&=(\nabla X)f \rule[-0.07in]{0in}{0in} \\
\{\ell X,\ell Y\}\;\;
&=-\Tr(\nabla X\cdot\nabla Y) \\
\{\ell X,\ell Y\}_\Omega
&\textstyle
=\Tr\Big(-\nabla(\nabla X)\cdot\nabla Y
  +\nabla X\cdot\iota_Y R
  -\iota_X R\cdot\nabla Y\Big)
  +\frac{1}{2}\iota_X\iota_Y H
\end{align*}
\end{prop}

\begin{proof}
For the calculations below, keep in mind
(\ref{freeVA.comm})--(\ref{freeVA.NOP}) and Lemma
\ref{lemma.CDOP.VAoidnew}.
By definition (\ref{VA.VAoid.456}), $\ell X\bullet f$ is
represented in $\CDO_{\Theta,H}(P)$ by
\begin{align*}
\big(\wt X-\tau_i\wt X^j\,&\rho(\Theta)_{ij}\big)_{-1}f
-\big(\wt{fX}-\tau_i(\wt{fX})^j\rho(\Theta)_{ij}\big) \\
&=\big(\tau_{i,-1}\wt X^i_0+\wt X^i_{-1}\tau_{i,0}\big)f
  -f\wt X
  +\big(\tau_i(f\wt X^j)-f\tau_i\wt X^j\big)\rho(\Theta)_{ij}
\qquad\qquad \\
&=f\wt X+(\tau_i f)d\wt X^i
  -f\wt X
  +(\tau_i f)\wt X^j\rho(\Theta)_{ij} \\
&=(\tau_i f)\big(d\wt X^i+\rho(\Theta)_{ij}\wt X^j\big)
\end{align*}
This proves the first claim according to (\ref{DX.lift}).
By definition (\ref{VA.VAoid.456}), $\{\ell X,\ell Y\}$ is
represented by
\begin{align*}
\big(\wt X-\tau_i\wt X^j\,&\rho(\Theta)_{ij}\big)_1
\big(\wt Y-\tau_k\wt Y^\ell\,\rho(\Theta)_{k\ell}\big) \\
&=\big(\wt X^i_1\tau_{i,0}+\wt X^i_0\tau_{i,1}\big)
  \tau_{k,-1}\wt Y^k \\
&=-[\tau_{k,-1},\wt X^i_1]\,\tau_{i,0}\wt Y^k
  -[[\tau_{i,0},\tau_{k,-1}],\wt X^i_1]\,\wt Y^k
  +\wt X^i_0\,[\tau_{i,1},\tau_{k,-1}]\,\wt Y^k \\
&=-(\tau_k\wt X^i)(\tau_i\wt Y^k)
  +\wt Y^k\Omega(\tau_i,\tau_k)^P\wt X^i
  -\wt X^i\Omega(\tau_i,\tau_k)^P\wt Y^k
  +2\t{Ric}_{ik}\wt X^i\wt Y^k \\
&=-(\tau_k\wt X^i)(\tau_i\wt Y^k)
\end{align*}
where we have used (\ref{lift.inv}).
This proves the second claim again thanks to (\ref{DX.lift}).
Let us only sketch the calculation of $\{\ell X,\ell Y\}_\Omega$.
First of all, by definition (\ref{VA.VAoid.456}) it is
represented by
\begin{align*}
\big(\wt X-\tau_i\wt X^j\,&\rho(\Theta)_{ij}\big)_0
\big(\wt Y-\tau_k\wt Y^\ell\,\rho(\Theta)_{k\ell}\big)
-\big(\wt{[X,Y]}-\tau_i\wt{[X,Y]}{}^j\rho(\Theta)_{ij}\big) \\
&=\big(\tau_{i,-1}\wt X^i_1
  +\wt X^i_0\tau_{i,0}
  +\wt X^i_{-1}\tau_{i,1}\big)\tau_{k,-1}\wt Y^k
  -\wt{[X,Y]} \\
&\qquad
  +\iota_{\wt Y}d\big(\tau_i\wt X^j\,\rho(\Theta)_{ij}\big)
  -L_{\wt X}\big(\tau_i\wt Y^j\,\rho(\Theta)_{ij}\big)
  +\tau_i\wt{[X,Y]}{}^j\rho(\Theta)_{ij} \quad\qquad
\end{align*}
where we have used Lemma \ref{lemma.1form.0mode}.
With some work, we can rewrite the first line as
\begin{align*}
\textstyle
-\Omega(\wt X,\wt Y)^P
-\tau_i\wt Y^k\,d(\tau_k\wt X^i)
+\frac{1}{2}\iota_{\wt X}\iota_{\wt Y}H
+\lambda^*(\Omega(\wt X,\wt Y),\Theta)
\end{align*}
and the second line, thanks to the first Bianchi identity, as
\begin{align*}
\big((\tau_i\wt X^k)(\tau_k\wt Y^j)
  -(\tau_i\wt Y^k)(\tau_k\wt X^j)\big)\rho(\Theta)_{ij}
+\tau_i\wt X^j\,\iota_{\wt Y}\rho(\Omega)_{ij}
-\tau_i\wt Y^j\,\iota_{\wt X}\rho(\Omega)_{ij}
-\lambda_\rho(\Omega(\wt X,\wt Y),\Theta).
\end{align*}
Notice that $\lambda^*=\lambda_\rho$ in our current
setting (see \S\ref{sec.ass.C}) and by the proof of
Lemma \ref{lemma.Feigin.C.1} we may ignore the term
$-\Omega(\wt X,\wt Y)^P$.
Then it follows from (\ref{tau.brackets}) and (\ref{lift.inv})
that $\{\ell X,\ell Y\}_\Omega$ is also represented by
\begin{align*}
\textstyle
-(\tau_\ell\tau_k\wt X^i)(\tau_i\wt Y^k)\tau^\ell
+\tau_i\wt X^j\,\iota_{\wt Y}\rho(\Omega)_{ij}
-\tau_i\wt Y^j\,\iota_{\wt X}\rho(\Omega)_{ij}
+\frac{1}{2}\iota_{\wt X}\iota_{\wt Y}H.
\end{align*}
In view of (\ref{DX.lift}), this proves the last claim.
\end{proof}

{\it Remark.}
This result recovers the vertex algebroid described in
Theorem \ref{thm.globalCDO}a (as $\nabla$ is now
torsion-free).
In other words, we have $V^f=\CDO_{\nabla,H}(M)$.

\begin{subsec} \label{sec.ass.C.conformal}
{\bf The conformal vector.}
According to \S\ref{sec.ass.C}, the vertex algebra
$\ass(\pi,\bb{C})$ has a conformal vector $\nu^M$ of central
charge $2d=2\dim M$, represented by the element
$\nu^{\Pi\mf{g}}\otimes\mb{1}+\mb{1}\otimes\nu^P$ of
(\ref{Feigin.C}).
By Lemma \ref{lemma.G}, $\nu^M$ can also be represented by 
\begin{align} \label{ass.C.conformal}
\textstyle
\nu^{\Pi\mf{g}}\otimes\mb{1}+\mb{1}\otimes\nu^P
  -Q_0\gamma
=\mb{1}\otimes\Big(\tau_{i,-1}\tau^i
  -\frac{1}{2}\lambda_\rho(\Theta_{-1}\Theta)\Big).
\end{align}
The key to understanding the entire structure of
$\ass(\pi,\bb{C})$ is the observation that $\nu^M$ is generated
by the associated vertex algebroid (in a certain way).
\end{subsec}

\begin{prop} \label{prop.ass.C.conformal}
The vertex algebra $V^f$ has a conformal vector $\nu^f$ with
$\Psi(\nu^f)=\nu^M$ (see \S\ref{sec.ass.C.VAoid}).
Given an open subset $U\subset M$ and a smooth
section $\sigma:U\rightarrow\pi^{-1}(U)\subset P$ of
$\pi$, there is a local expression
\begin{align*}
\textstyle
\nu^f|_U
=(\ell\sigma_i)_{-1}\sigma^i
+\frac{1}{2}\Tr\big(\Gamma^\sigma_{-1}\Gamma^\sigma)
+\sigma^i([\sigma_j,\sigma_k])\,
  \sigma^k_{-1}\Gamma^\sigma_{ji}
\end{align*}
\footnote{ \label{underlying.sheaf}
Recall from \S\ref{sec.review} that, like many other constructions
in this paper, $V^f$ is the space of global sections of some
underlying sheaf (on $M$).
}
where $(\sigma_1,\ldots,\sigma_d)$ is the
$C^\infty(U)$-basis of $\T(U)$ induced by $\sigma$;
$(\sigma^1,\ldots,\sigma^d)$ the dual basis of
$\Omega^1(U)$;
and $\Gamma^\sigma=\rho(\sigma^*\Theta)$.
In particular, $\nu^f$ belongs to
$\mc{F}_{\preceq(-1;-1)}V^f$ (see \S\ref{freeVA.PBW}).
\end{prop}

\begin{proof}
By assumption, $\pi:P\rightarrow M$ is a lifting of the usual
frame bundle of $TM$ (see \S\ref{sec.ass.C}), so that any local
section of $\pi$ indeed induces a local framing of $TM$.
Consider the smooth map $g:\pi^{-1}(U)\rightarrow G$ defined by
$\sigma(\pi(p))=p\cdot g(p)$ for $p\in\pi^{-1}(U)$.
To ease notations, let us also write
$\rho(g):\pi^{-1}(U)\rightarrow SO(\bb{R}^d)$ simply as $g$.
Then we have
\begin{align} \label{frame.lift}
\wt\sigma_i=g_{ri}\tau_r,\qquad
\sigma^i=g_{ri}\tau^r,\qquad
\tau_r=g_{ri}\wt\sigma_i,\qquad
\tau^r=g_{ri}\sigma^i.
\end{align}
Since $\nabla\sigma_i=\Gamma^\sigma_{ji}\otimes\sigma_j$,
it follows from (\ref{DX.lift}) and (\ref{frame.lift}) that
\begin{align} \label{dg}
dg+\rho(\Theta)\cdot g
=(\tau_k g)\tau^k
=g\cdot\Gamma^\sigma
\end{align}
where $\cdot$ denotes matrix multiplication.

Consider the representative of $\nu^M$ in (\ref{ass.C.conformal}).
Our main task is to express that element entirely in terms of
$\sigma$.
First we can write
\begin{align*}
\tau_{r,-1}\tau^r
&=\tau_{r,-1}g_{ri,0}\sigma^i \\
&=(\tau_{r,-1}g_{ri})_{-1}\sigma^i
-\big(g_{ri,-1}\tau_{r,0}+g_{ri,-2}\tau_{r,1}\big)
  g_{si}\tau^s \\
&\textstyle
=\wt\sigma_{i,-1}\sigma^i
-(\tau_r g_{si})\tau^s_{-1}dg_{ri}
-\Tr\big(g^{-1}\cdot\rho(\Theta)_{-1}dg\big)
-\frac{1}{2}\Tr\big(g^{-1}\cdot Tdg\big)
\end{align*}
using (\ref{frame.lift}), Corollary \ref{cor.fund.normal.comm}
and the Lie derivative $L_{\tau_r}\tau^s=\rho(\Theta)_{sr}$
implied by (\ref{tau.brackets}).
Then we work on each term separately, with repeated use of
(\ref{frame.lift}) and (\ref{dg}):
\begin{align*}
\t{2nd term}
&=(\tau_r g_{si})\tau^s_{-1}
  \big(\rho(\Theta)_{rt}g_{ti}
  -g_{rj}\Gamma^\sigma_{ji}\big) \\
&=-g_{si}(\tau_r g_{ti})\tau^s_{-1}\rho(\Theta)_{rt}
  +g_{si} g_{rj}(\tau_r g_{sk})\sigma^k_{-1}
    \Gamma^\sigma_{ji} \\
&=-(\tau_r\wt\sigma_i^t)\sigma^i_{-1}\rho(\Theta)_{rt}
  +g_{si} g_{rk}(\tau_r g_{sj})\sigma^k_{-1}
    \Gamma^\sigma_{ji}
  +\sigma^i([\wt\sigma_j,\wt\sigma_k])\sigma^k_{-1}
    \Gamma^\sigma_{ji} \\
&=-(\tau_r\wt\sigma_i^t)\sigma^i_{-1}\rho(\Theta)_{rt}
  +\Tr\big(\Gamma^\sigma_{-1}\Gamma^\sigma\big)
  +\sigma^i([\sigma_j,\sigma_k])\sigma^k_{-1}
    \Gamma^\sigma_{ji}  \\
\t{3rd term}
&=\lambda_\rho(\Theta_{-1}\Theta)
-\Tr\big(g^{-1}\cdot\rho(\Theta)_{-1}\cdot
  g\cdot\Gamma^\sigma\big) \\
\t{4th term}
&\textstyle
=-\frac{1}{2}\Tr\big(g^{-1}\cdot dg_{-1}
  \cdot g^{-1}\cdot dg\big) \\
&\textstyle
=-\frac{1}{2}\lambda_\rho(\Theta_{-1}\Theta)
+\Tr\big(g^{-1}\cdot\rho(\Theta)_{-1}\cdot
  g\cdot\Gamma^\sigma\big)
-\frac{1}{2}\Tr(\Gamma^\sigma_{-1}\Gamma^\sigma)
\end{align*}
These calculations together yield an identity
in $\CDO_{\Theta,H}(\pi^{-1}(U))$:
\begin{align} \label{ass.C.conformal.local}
\textstyle
\tau_{r,-1}\tau^r
-\frac{1}{2}\lambda_\rho(\Theta_{-1}\Theta)
=\big(\wt\sigma_i
  -\tau_r\wt\sigma^t_i\,\rho(\Theta)_{rt}\big)_{-1}
  \sigma^i
+\frac{1}{2}\Tr(\Gamma^\sigma_{-1}\Gamma^\sigma)
  +\sigma^i([\sigma_j,\sigma_k])\sigma^k_{-1}
    \Gamma^\sigma_{ji}.
\end{align}
The left hand side is defined globally on $P$ and represents
$\nu^M\in\ass(\pi,\bb{C})$.
Since the right hand side is manifestly generated by the subspace
(\ref{ass.C.1}), it defines an element $\nu^f\in V^f$ such that
$\Psi(\nu^f)=\nu^M$.

It remains to show that $\nu^f$ is a conformal
vector of $V^f$.
According to \cite{Kac,FB-Z}, this amounts to
checking:
(i) $\nu^f_{-1}=T$,
(ii) $\nu^f_0=L_0$, and
(iii) $\nu^f_2\nu^f\in\bb{C}$.
For (i) and (ii) it suffices to check them on
\begin{align*}
C^\infty(U)\cup\{\ell\sigma_1,\ldots,\ell\sigma_d\}
\end{align*} 
because $V^f$ is locally generated by these elements.
In fact, since $\Psi(\nu^f)=\nu^M$ is known to be conformal
and $\Psi$ is an isomorphism in weights $0$ and $1$, everything
we need to check automatically holds except for the equation
$\nu^f_{-1}\ell\sigma_i=T(\ell\sigma_i)$ in weight $2$.
This equation can be verified by a straightforward calculation,
which we omit.
\end{proof}

{\it Remark.}
Since we have already observed that $V^f=\CDO_{\nabla,H}(M)$,
this result recovers the conformal vector $\nu^0$ described
in Theorem \ref{thm.globalCDO}b (as $\Gamma^\sigma$ is
now traceless).

\begin{corollary} \label{cor.ass.C.freeVA}
The map of vertex algebras $\Psi:V^f\rightarrow\ass(\pi,\bb{C})$ 
(see \S\ref{sec.ass.C.VAoid}) is an isomorphism.
\end{corollary}

\begin{proof}
By Proposition \ref{prop.ass.C.conformal},
the conformal vector $\nu^M\in\ass(\pi,\bb{C})$ belongs to
$\mc{F}_{\preceq(-1;-1)}$.
Then Lemma \ref{lemma.conf.gen} applies so that
$\Psi$ is surjective.
By construction, the ideal $\ker\Psi\subset V^f$ is trivial in
weights $0$ and $1$.
Then by Proposition \ref{prop.ass.C.conformal} again and Lemma
\ref{lemma.conf.ideal}, $\ker\Psi$ is in fact trivial in all weights.
\end{proof}

Propositions \ref{prop.ass.C.VAoid},
\ref{prop.ass.C.conformal} and Corollary \ref{cor.ass.C.freeVA}
together show that $\ass(\pi,\bb{C})\cong\CDO_{\nabla,H}(M)$
as conformal vertex algebras, thus fulfilling the goal of this
section.
Let us summarize our work as follows.

\begin{theorem} \label{thm.CDO.semiinf}
Suppose $\pi:P\rightarrow M$ is a smooth prinicpal
$G$-bundle and $\rho:G\rightarrow SO(\bb{R}^d)$ is
a representation such that there is an isomorphism
$P\times_\rho\bb{R}^d\cong TM$.
Given a principal frame $(\hat{\mf{g}}_{-\Killing},G)$-algebra
$\CDO_{\Theta,H}(P)$ (see Theorem \ref{thm.CDOP}), the zeroth
semi-infinite cohomology
\begin{align*}
\ass(\pi,\bb{C})
=H^{\frac{\infty}{2}+0}\big(\hat{\mf{g}}_{-\Killing},
  \CDO_{\Theta,H}(P)\big)
\end{align*}
is an algebra of CDOs on $M$.
Up to isomorphism, every algebra of CDOs on $M$ arises
this way.
This vertex algebra is freely generated by its weight-zero and
weight-one components, which are represented bijectively by the
following subspaces of $\CDO_{\Theta,H}(P)$:
\begin{align*}
\pi^*C^\infty(M),\qquad
\pi^*\Omega^1(M)\oplus
  \big\{\wt X-\tau_i\wt X^j\,\rho(\Theta)_{ij}:
  X\in\T(M)\big\}.
\end{align*}
(For details on the weight-one component, see Lemma
\ref{lemma.Feigin.C.1}.)
Moreover, $\ass(\pi,\bb{C})$ has a conformal vector
of central charge $2d=2\dim M$, represented in
$\CDO_{\Theta,H}(P)$ by
\begin{align*}
\textstyle
\qquad
\tau_{i,-1}\tau^i
-\frac{1}{2}\lambda_\rho(\Theta_{-1}\Theta).
\qquad\qedsymbol
\end{align*}
\end{theorem}

To conclude this section, let us describe an extension of
Theorem \ref{thm.CDO.semiinf}.

\begin{subsec} \label{CDO.cs}
{\bf Generalization to supermanifolds.}
Suppose $\pi:P\rightarrow M$ and
$\rho:G\rightarrow SO(\bb{R}^d)$ are the same as above;
also let $\rho':G\rightarrow U(\bb{C}^r)$ be another
representation, $E=P\times_{\rho'}\bb{C}^r$ the associated
vector bundle and $\Pi E$ the corresponding cs-manifold.
\cite{QFS.susy}
The $G$-action on
$\mc{O}(\Pi\bb{R}^r)\otimes\bb{C}=\wedge^*(\bb{C}^r)^\vee$
induced by $\rho'$ lifts to an inner
$(\hat{\mf{g}}_{\lambda_{\rho'}},G)$-action on
$\CDO(\Pi\bb{R}^r)$, a fermionic analogue of
\S\ref{algCDO} (see Definition \ref{formalLG.action}).
By Theorem \ref{thm.CDOP}, there exists a principal frame
$(\hat{\mf{g}}_\lambda,G)$-algebra $\CDO_{\Theta,H}(P)$ with
$\lambda+\lambda_{\rho'}=-\Killing$ if and only if
\begin{align*}
\lambda^*(P)=(\lambda_\rho-\lambda_{\rho'})(P)=0
\qquad\iff\qquad
p_1(M)-ch_2(E)=0.
\end{align*}
In this case, we can apply Definition \ref{defn.ass}
to construct a vertex superalgebra
\begin{align*}
\ass(\pi,\CDO(\Pi\bb{R}^r))
=H^{\frac{\infty}{2}+0}\Big(\hat{\mf{g}}_{-\Killing}\,,\,
  \CDO_{\Theta,H}(P)\otimes\CDO(\Pi\bb{R}^r)\Big).
\end{align*}
Moreover, the $(\hat{\mf{g}}_{\lambda_{\rho'}},G)$-action
on $\CDO(\Pi\bb{R}^r)$ is primary if and only if
\begin{align*}
\rho'(\mf{g})\subset\mf{su}_r
\qquad\iff\qquad
c_1(E)=0.
\end{align*}
In this case, $\ass(\pi,\CDO(\Pi\bb{R}^r))$
has a conformal vector of central charge $2(d-r)$ by Lemma
\ref{lemma.ass}b.
It follows from a similar analysis that $\ass(\pi,\CDO(\Pi\bb{R}^r))$
is an algebra of CDOs on $\Pi E$ in the sense of \cite{myCDO}.
In particular, in the case $\rho'=\rho_{\bb{C}}$
(i.e. $E=TM_{\bb{C}}$), both obstructions are trivial and
$\ass(\pi,\CDO(\Pi\bb{R}^d))$ is the chiral de Rham algebra of $M$.
\end{subsec}

\newpage
\setcounter{equation}{0}
\section{Example: Spinor Module over CDOs}
\label{sec.FL.spinor}

In this section, our object of study is another example of
the construction in \S\ref{sec.ass} called the
\emph{spinor module}.
Following a similar strategy as in \S\ref{CDOP.CDOM}, we
analyze its structure in order to give an explicit description
of generating data and relations.
It is hoped that this spinor module has an interpretation
in terms of a ``spinor bundle with connection on the formal
loops of a string manifold'', and a deeper understanding of it,
including the identification of an appropriate Dirac operator,
will lead to a useful geometric theory of the Witten genus.
(This was in fact the original motivation of the paper.)

Recall that the normalized Killing form on $\mf{so}_d$ is given by
$\lambda_0(A,B)=\frac{1}{2}\Tr AB$.
Notice that $\lambda_\rho=2\lambda_0$ for the standard
representation $\rho$ and $\Killing=(2d-4)\lambda_0$
(see \S\ref{conventions}).
For $k\in\bb{C}$, we will write
$(\widehat{\mf{so}}_d)_{k\lambda_0}$ more simply as
$(\widehat{\mf{so}}_d)_k$ (see \S\ref{sec.Gmfld}).

\begin{subsec} \label{sec.Cl}
{\bf The Ramond Clifford algebra.}
Let $C\ell$ be the unital $\bb{Z}/2\bb{Z}$-graded
associative $\bb{C}$-algebra with the following generators
and relations
\begin{align} \label{Cl}
e_{i,n}\t{ odd},\;\;i=1,\ldots,d,\;\;n\in\bb{Z},
\qquad
[e_{i,n},e_{j,m}]
=e_{i,n}e_{j,m}+e_{j,m}e_{i,n}
=-2\delta_{ij}\delta_{n+m,0}.
\end{align}
\footnote{
The notation and the normalization of the generators
are chosen to resemble those in e.g.~\cite{spin}.
}
Suppose $W$ is a $C\ell$-module with the property that
any $w\in W$ is annihilated by $e_{i,n}$ for sufficiently
large $n$.
The operators given by
\begin{align} \label{Cl.so}
A^{C\ell}_n
=\frac{1}{4}A_{ji}\,e_{i,n-r}e_{j,r}
\quad\t{for }A\in\mf{so}_d,\;n\in\bb{Z}
\end{align}
define an $(\widehat{\mf{so}}_d)_1$-action on $W$.
On the other hand, the operators given by
\begin{align} \label{Cl.Virasoro}
L^{C\ell}_n
=-\frac{1}{8}\sum_{r\geq 0}(2r-n)e_{i,n-r}e_{i,r}
  +\frac{1}{8}\sum_{r<0}(2r-n)e_{i,r}e_{i,n-r}
  +\frac{d}{16}\delta_{n,0}\quad
\t{for }n\in\bb{Z}
\end{align}
define a Virasoro action on $W$ of central charge
$d/2$.
This is in fact the Sugawara construction associated to
the above $(\widehat{\mf{so}}_d)_1$-action, and
accordingly satisfies
\begin{align*}
[L^{C\ell}_n,A^{C\ell}_m]=-mA^{C\ell}_{n+m}\quad
\t{for }n,m\in\bb{Z}.
\end{align*}
The eigenvalues of $L^{C\ell}_0$ are called weights (as usual)
and $e_{i,n}$ changes weights by $-n$.
For more detailed explanations, see e.g.~\cite{Fuchs}
\end{subsec}

For the rest of the section, $d=2d'$ is even.

\begin{subsec} \label{sec.S}
{\bf The spinor representation of $\widehat{\mf{so}}_{2d'}$.}
Let $C\ell_0$ (resp.~$C\ell_+$) be the subalgebra of $C\ell$
generated by those $e_{i,n}$ with $n=0$ (resp.~$n\geq 0$).
The finite-dimensional Clifford algebra $C\ell_0$ has a unique
irreducible $\bb{Z}/2\bb{Z}$-graded representation $S_0$.
Regarding $S_0$ as a $C\ell_+$-module on which $\{e_{i,n}\}_{n>0}$
act trivially, we define a $C\ell$-module by
\begin{align*}
S=C\ell\otimes_{C\ell_+}S_0.
\end{align*}
By (\ref{Cl}), $S$ as a vector space is spanned by elements of
the form
\begin{align*}
e_{i_p,n_p}\cdots e_{i_1,n_1}s,\qquad
n_1<0,\;\;
(n_p,i_p)<\cdots<(n_1,i_1),\;\;
s\in S_0
\end{align*}
where the indicated pairs are ordered lexicographically.
According to \S\ref{sec.Cl}, $S$ admits
a $((\widehat{\mf{so}}_{2d'})_1,\t{Spin}_{2d'})$-action
together with an intertwining Virasoro action of central
charge $d'$.
Notice that the element displayed above has weight
\begin{align*}
\frac{d'}{8}+|n_1|+\ldots+|n_p|.
\end{align*}
For convenience, we will write $S_k\subset S$ for
the component of weight $(d'/8)+k$.
\end{subsec}

The following is our main object of study in this section.

\begin{subsec} \label{sec.ass.S}
{\bf The spinor module over CDOs.}
Consider the special case of Definition \ref{defn.ass}
associated to a principal frame
$((\hat{\mf{so}}_{2d'})_{3-4d'},\t{Spin}_{2d'})$-algebra
$\CDO_{\Theta,H}(P)$ and the
$((\hat{\mf{so}}_{2d'})_1,\t{Spin}_{2d'})$-module $S$:
\begin{align}
\ass(\pi,S)
&=H^{\frac{\infty}{2}+0}\Big(
  (\widehat{\mf{so}}_{2d'})_{4-4d'}\,,
  \,\CDO_{\Theta,H}(P)\otimes S\Big) \nonumber \\
&=H^0\Big(
  \CDO(\Pi\mf{so}_{2d'})\otimes\CDO_{\Theta,H}(P)
  \otimes S\,,\,Q_0
\Big) \label{ass.S}
\end{align}
By Lemma \ref{lemma.ass} and \S\ref{sec.S},
$\ass(\pi,S)$ is a module over the vertex algebra
$\CDO_{\Theta,H}(P)^{\widehat{\mf{so}}_{2d'}}$ and
admits a Virasoro action of central charge
$5d'=\frac{5}{2}\dim M$.
For convenience, let us recall some of the data involved.
\begin{itemize}
\item[$\centerdot$]
Let $(t_1,t_2,\ldots)$ be a basis of $\mf{so}_{2d'}$;
$(t^1,t^2,\ldots)$ the dual basis of $(\mf{so}_{2d'})^\vee$;
$(\phi^1,\phi^2,\ldots)$ the corresponding coordinates of
the supermanifold $\Pi\mf{so}_{2d'}$;
and $(\d_1,\d_2,\ldots)$ their coordinate vector fields.
\vspace{-0.05in}
\item[$\centerdot$]
For comments on the conformal vertex superalgebra
$\CDO(\Pi\mf{so}_{2d'})$, see \S\ref{sec.ass.C}.
\vspace{-0.05in}
\item[$\centerdot$]
For the detailed definition of the vertex algebra
$\CDO_{\Theta,H}(P)$, see Theorem \ref{thm.CDOP} with
$\lambda=(3-4d')\lambda_0$ in mind.
Let us mention that it is defined using:
a principal $\t{Spin}_{2d'}$-frame bundle $\pi:P\rightarrow M$;
the Levi-Civita connection $\Theta$ on $\pi$;
and a basic $3$-form $H$ on $P$ that satisfies
$dH=\lambda_0(\Omega\wedge\Omega)$, where $\Omega$ is the
curvature of $\Theta$.
Also, it is generated by such elements as $f\in C^\infty(P)$,
$A^P\in\T_v(P)$ for $A\in\mf{so}_{2d'}$ and $\tau_i\in\T_h(P)$
for $i=1,\ldots,2d'$, and a conformal vector of central charge
$2\dim P$ is given by
\begin{align*}
\textstyle
\nu^P=t^P_{a,-1}\Theta^a+\tau_{i,-1}\tau^i
  -\frac{1}{2}\lambda_0(\Theta_{-1}\Theta).
\end{align*}
For the meaning of various notations, see \S\ref{sec.Gmfld} and
\S\ref{sec.prin.bdl2}.
\vspace{-0.05in}
\item[$\centerdot$]
For details on the Feigin complex, see \S\ref{sec.Feigin}.
Let us mention that its differential is given by
\begin{align} \label{ass.S.Q.theta}
Q_0
=q_0\otimes 1\otimes 1
+\phi^a_{-r}\otimes t^P_{a,r}\otimes 1
+\phi^a_{-r}\otimes 1\otimes t^{C\ell}_{a,r}
\end{align}
where $q$ is the odd vector field in (\ref{JQ.coor}).
Also, the Virasoro operators
\begin{align} \label{ass.S.Vir}
\qquad\qquad
L^{\Pi{\mf{so}}}_n\otimes 1\otimes 1
+1\otimes L^P_n\otimes 1
+1\otimes 1\otimes L^{C\ell}_n,
\qquad n\in\bb{Z}
\end{align}
preserve the gradation and commute with $Q_0$.
\vspace{-0.05in}
\end{itemize}
Let us slightly rephrase our geometric ingredients:
$M$ is a Riemannian manifold with a spin structure and a $3$-form
$H$ that satisfies $dH=\frac{1}{2}\Tr(R\wedge R)$, where $R$ is
the Riemannian curvature.
This can be viewed as the de Rham version of
a \emph{string structure}.
In the sequel, we will describe $\ass(\pi,S)$ more
explicitly in terms of generating data (i.e. a subspace and
some fields) and their relations.
\end{subsec}

Throughout this section we identify $\Omega^*(M)$ with
the basic subspace of $\Omega^*(P)$.

\begin{subsec} \label{sec.ass.S.0}
{\bf The component of the lowest weight.}
Consider the Feigin complex that defines (\ref{ass.S}):
\begin{align} \label{Feigin.S}
\Big(
  \CDO(\Pi\mf{so}_{2d'})\otimes\CDO_{\Theta,H}(P)
  \otimes S\,,\,Q_0
\Big).
\end{align}
Since its component of weight $d'/8$ is simply
the Chevalley-Eilenberg complex
\begin{align*}
\Big(\mc{O}(\Pi\mf{so}_{2d'})\otimes C^\infty(P)
  \otimes S_0\,,\,Q_0\Big)
\end{align*}
(see \S\ref{sec.CE}), the component of (\ref{ass.S}) of weight
$d'/8$ is
\begin{align} \label{ass.S.0}
\ass(\pi,S)_0
=H^0\big(\mf{so}_{2d'},C^\infty(P)\otimes S_0\big)
=(C^\infty(P)\otimes S_0)^{\mf{so}_{2d'}}
=S(M)
\end{align}
i.e.~the space of smooth sections of the spinor bundle on $M$.
Understanding the rest of $\ass(\pi,S)$ requires more work.
\end{subsec}

\begin{lemma} \label{lemma.G.S}
Let $\gamma$ be the same element as in Lemma \ref{lemma.G}
(with $\mf{g}=\mf{so}_{2d'}$) and
$\ms{Y}=\ms{Y}^a t^P_a\in\T_v(P)^{\mf{so}_{2d'}}$.
For $n\in\bb{Z}$ we have the following equations of operators
on the Feigin complex (\ref{Feigin.S}):
\begin{align*}
&[Q_0,\,\gamma_n\otimes 1]
=L^{\Pi\mf{so}}_n\otimes 1\otimes 1
+1\otimes(t^P_{a,-1}\Theta^a)_n\otimes 1
+1\otimes\Theta^a_{n-r}\otimes t^{C\ell}_{a,r} \\
&[Q_0,\,
  (\gamma_0(1\otimes\ms{Y}))_n\otimes 1]
=1\otimes\ms{Y}_n\otimes 1
+1\otimes\ms{Y}^a_{n-r}\otimes t^{C\ell}_{a,r}
\end{align*}
\end{lemma}

\begin{proof}
Recall the differential $Q_0$ from (\ref{ass.S.Q.theta}).
The first anticommutator is computed as follows
\begin{align*}
[Q_0,\,\gamma_n\otimes 1]
&=[(q\otimes 1+\phi^a\otimes t^P_a)_0,\,\gamma_n]\otimes 1
+[\phi^a_{-r},\d_{b,n-s}]\otimes\Theta^b_s
  \otimes t^{C\ell}_{a,r} \\
&=L^{\Pi\mf{so}}_n\otimes 1\otimes 1
+1\otimes(t^P_{a,-1}\Theta^a)_n\otimes 1
+1\otimes\Theta^a_{n-r}\otimes t^{C\ell}_{a,r}
\end{align*}
using Lemma \ref{lemma.G} and the fermionic analogue of
(\ref{Weyl}).
On the other hand, since
$\gamma_0(1\otimes\ms{Y})=\d_b\otimes\ms{Y}^b$,
the second anticommutator can be written as
\begin{align*}
[Q_0,(\d_b\otimes\ms{Y}^b)_n\otimes 1]
&=[(q\otimes 1+\phi^a\otimes t^P_a)_0,
  (\d_b\otimes\ms{Y}^b)_n]\otimes 1
+[\phi^a_{-r},\d_{b,n-s}]\otimes\ms{Y}^b_s
  \otimes t^{C\ell}_{a,r} \\
&=\big((q\otimes 1+\phi^a\otimes t^P_a)_0
  (\d_b\otimes\ms{Y}^b)\big)_n\otimes 1
+1\otimes\ms{Y}^a_{n-r}\otimes t^{C\ell}_{a,r}
\end{align*}
using again the analogue of (\ref{Weyl}).
It remains to compute the element
$(q\otimes 1+\phi^a\otimes t^P_a)_0
(\d_b\otimes\ms{Y}^b)$, which is the sum of the
following:
\begin{align*}
&(q_0\d_b)\otimes\ms{Y}^b
=[q,\d_b]\otimes\ms{Y}^b
=-t^c([t_b,t_d])\phi^d\d_c\otimes\ms{Y}^b \\
&(\phi^a_0\otimes t^P_{a,0})(\d_b\otimes\ms{Y}^b)
=\phi^a\d_b\otimes t^P_a\ms{Y}^b
=-t^b([t_a,t_c])\phi^a\d_b\otimes\ms{Y}^c \\
&(\phi^a_1\otimes t^P_{a,-1})(\d_b\otimes\ms{Y}^b)
=\delta^a_b\otimes t^P_{a,-1}\ms{Y}^b
=1\otimes\ms{Y}
\end{align*}
Indeed, the first line is similar to a computation
in the proof of Lemma \ref{lemma.G};
the second follows from the
$\mf{so}_{2d'}$-invariance of $\ms{Y}$;
and the last follows from the analogue of
(\ref{Weyl}) and Corollary
\ref{cor.fund.normal.comm}.
\end{proof}

{\it Preparation.}
Let $C\ell(M)=(C^\infty(P)\otimes C\ell_0)^{\mf{so}_{2d'}}$,
which can be viewed as the algebra of $C^\infty(M)$-linear
endomorphisms of $S(M)$.
For $X\in\T(M)$, we denote its horizontal lift by
$\wt X=\wt X^i\tau_i\in\T_h(P)^{\mf{so}_{2d'}}$ and
its Clifford action by $cX=\wt X^i\otimes e_{i,0}\in C\ell(M)$.
Also, since each
$\ms{Y}=\ms{Y}^a t^P_a\in\T_v(P)^{\mf{so}_{2d'}}$ satisfies
\begin{align*}
\ms{Y}\otimes 1+\ms{Y}^a\otimes t_a=0
\quad\t{on}\quad
(C^\infty(P)\otimes S_0)^{\mf{so}_{2d'}}=S(M),
\end{align*}
it represents an endomorphism
$c\ms{Y}=-\ms{Y}^a\otimes t_a\in C\ell(M)$.

\begin{subsec} \label{sec.ass.S.fields}
{\bf Fields of low weights.}
Recall Lemma \ref{lemma.semiinf.ops}a for the special case
(\ref{ass.S}):
the $\widehat{\mf{so}}_{2d'}$-invariant fields
on $\CDO_{\Theta,H}(P)\otimes S$, including
the vertex operators of 
$\CDO_{\Theta,H}(P)^{\widehat{\mf{so}}_{2d'}}$,
induce fields on $\ass(\pi,S)$.
Let us now describe a collection of such fields (or rather their
Fourier modes);
we will later see that their actions on the subspace
$\ass(\pi,S)_0=S(M)$ generate the entire space $\ass(\pi,S)$.

For $f\in C^\infty(M)$, $X\in\T(M)$ and $n\in\bb{Z}$,
define the following operators on $\ass(\pi,S)$:
\begin{align}
f_n
&:=\t{the operator induced by }f_n\otimes 1
\nonumber \\
\ell X_n
&:=\t{the operator induced by }
  (\wt X-\tau_i\wt X^j\,\Theta_{ij})_n\otimes 1
\label{ass.S.gen.fields} \\
c X_n
&:=\t{the operator induced by }
  \wt X^i_{n-r}\otimes e_{i,r}
\nonumber
\end{align}
Indeed, the first two are well-defined because both $f$ and
$\wt X-\tau_i\wt X^j\,\Theta_{ij}$ belong to
$\CDO_{\Theta,H}(P)^{\widehat{\mf{so}}_{2d'}}$
according to \S\ref{sec.invCDO.can};
and so is the last one because for $A\in\mf{so}_{2d'}$ and
$m\in\bb{Z}$ we have
\begin{align*}
\big[A^P_m\otimes 1+1\otimes A^{C\ell}_m,\,
  \wt X^i_{n-r}\otimes e_{i,r}\big]
&=(A^P\wt X^i)_{m+n-r}\otimes e_{i,r}
  +\wt X^i_{n-r}\otimes(A_{ji}e_{j,m+r}) \\
&=-A_{ij}\wt X^j_{m+n-r}\otimes e_{i,r}
  +A_{ji}\wt X^i_{n-r}\otimes e_{j,m+r}\;=0
\end{align*}
thanks to (\ref{freeVA.comm}), the $\mf{so}_{2d'}$-invariance
of $\wt X$ and (\ref{Cl})--(\ref{Cl.so}).
These three sequences of operators should be regarded as
the Fourier modes of certain fields of weights $0$, $1$ and
$\frac{1}{2}$ respectively.

It will be convenient to also consider some other operators
that are generated by those in (\ref{ass.S.gen.fields}).
For $\alpha\in\Omega^1(M)$, $\ms{Y}\in\T_v(P)^{\mf{so}_{2d'}}$
and $n\in\bb{Z}$, define the following operators on
$\ass(\pi,S)$:
\begin{align} \label{ass.S.other.fields}
\begin{array}{rl}
\alpha_n
&:=\t{the operator induced by }\alpha_n\otimes 1 \\
c\ms{Y}_n
&:=\t{the operator induced by }
  (\ms{Y}+\lambda_0(\Theta,\ms{Y}))_n\otimes 1
\ph{\Big)}
\end{array}
\end{align}
These are well-defined because both $\alpha$ and
$\ms{Y}+\lambda_0(\Theta,\ms{Y})$ belong to
$\CDO_{\Theta,H}(P)^{\widehat{\mf{so}}_{2d'}}$ according to
\S\ref{sec.invCDO.can}.
To see that $\alpha_n$ can be expressed in terms of
(\ref{ass.S.gen.fields}), it suffices to assume $\alpha=fdg$
and notice that
\begin{align} \label{ass.S.1formfield}
(fdg)_n=-\sum_{r\in\bb{Z}}r f_{n-r}g_r
\end{align}
by (\ref{freeVA.comm})--(\ref{freeVA.NOP}).
The next lemma describes how $c\ms{Y}_n$ can be expressed in
terms of (\ref{ass.S.gen.fields}).
\end{subsec}

\begin{lemma} \label{lemma.verfield}
The operator $c\ms{Y}_n$ defined above equals a sum of operators
of the form
\begin{align} \label{cX.cY.normal}
\sum_{r\geq 0}cX_{n-r}\,cY_r
-\sum_{r<0}cY_r\,cX_{n-r}
-2\langle X,\nabla Y\rangle_n
+\langle X,Y\rangle_n,\quad
X,Y\in\T(M)
\end{align}
where $\langle\;\rangle$ and $\nabla$ denote the Riemannian
metric and Levi-Civita connection on $TM$.
\end{lemma}

\begin{proof}
For the calculations below, keep in mind
(\ref{freeVA.comm})--(\ref{freeVA.NOP}).
Let us write $\ms{Y}=\ms{Y}^a t_a^P$ and
$\ms{Y}_{ij}=\ms{Y}^a(t_a)_{ij}$.
By Lemma \ref{lemma.G.S} and (\ref{Cl.so}), we can also
represent $c\ms{Y}_n$ on $\CDO_{\Theta,H}(P)\otimes S$ by
\begin{align*}
\textstyle
-\ms{Y}^a_{n-r}\otimes t^{C\ell}_{a,r}
+\lambda_0(\Theta,\ms{Y})_n\otimes 1
=\frac{1}{4}\,\big(\ms{Y}_{ij,n-r}\otimes 1\big)
  \big(1\otimes e_{i,r-s}e_{j,s}
  -2\Theta_{ij,r}\otimes 1\big).
\end{align*}
Since $c\ms{Y}=\frac{1}{4}\,\ms{Y}_{ij}\otimes e_{i,0}e_{j,0}$
corresponds to a $2$-vector field under the isomorphism
$C\ell(M)\cong\Gamma(\wedge^* TM)$, the above operator can
be written as a sum of operators of the form
\begin{align} \label{Cl2.field}
\textstyle
\qquad
\frac{1}{2}\,
\big((\wt X^i\wt Y^j-\wt Y^i\wt X^j)_{n-r}
  \otimes 1\big)
\big(1\otimes e_{i,r-s}e_{j,s}
  -2\Theta_{ij,r}\otimes 1\big),\quad
X,Y\in\T(M)
\end{align}
where the factor of $\frac{1}{2}$ is included only for
convenience.
It remains to show that (\ref{Cl2.field}) represents
(\ref{cX.cY.normal}).

Consider the first two terms in (\ref{cX.cY.normal}).
By definition (\ref{ass.S.gen.fields}) they are represented by
\begin{align*}
\sum_{r\geq 0}(\wt X^i_{n-r-s}\otimes e_{i,s})
  (\wt Y^j_{r-t}\otimes e_{j,t})
-\sum_{r<0}(\wt Y^j_{r-t}\otimes e_{j,t})
  (\wt X^i_{n-r-s}\otimes e_{i,s}).
\end{align*}
In view of (\ref{Cl}), let us split this sum into three parts.
The first part consists of terms with $i\neq j$:
\begin{align*}
\sum_r\sum_{i\neq j}\sum_{s,t}
  \wt X^i_{n-r-s}\wt Y^j_{r-t}\otimes e_{i,s}e_{j,t}
=\sum_{i\neq j}\sum_{s,t}
  (\wt X^i\wt Y^j)_{n-s-t}\otimes e_{i,s}e_{j,t}\,.
\end{align*}
The second part consists of terms with $i=j$ and
$s+t\neq 0$, which vanishes by symmetry:
\begin{align*}
\sum_r\sum_i\sum_{s+t\neq 0}
  \wt X^i_{n-r-s}\wt Y^i_{r-t}\otimes e_{i,s}e_{i,t}
=\sum_i\sum_{s+t\neq 0}
  (\wt X^i\wt Y^i)_{n-s-t}\otimes e_{i,s}e_{i,t}
=0\,.
\end{align*}
The last part consists of terms with $i=j$ and $s+t=0$,
which is computed as follows:
\begin{align*}
\sum_{r\geq 0}\sum_{i,s}
  \wt X^i_{n-r-s}&\wt Y^i_{r+s}\otimes e_{i,s}e_{i,-s}
-\sum_{r<0}\sum_{i,s}
  \wt X^i_{n-r-s}\wt Y^i_{r+s}\otimes e_{i,-s}e_{i,s} \\
&=-\sum_{i,u}(2u+1)\wt X^i_{n-u}\wt Y^i_u\otimes 1 \\
&=-(\wt X^i\wt Y^i)_n\otimes 1
  +2(\wt X^i d\wt Y^i)_n\otimes 1 \\
&=-\langle X,Y\rangle_n\otimes 1
  +2\langle X,\nabla Y\rangle_n\otimes 1
  -2\big(\wt X^i\wt Y^j\Theta_{ij}\big)_n\otimes 1
\end{align*}
where we have used (\ref{DX.lift}) in the last equality.
It follows from these calculations and a little reorganization
that (\ref{cX.cY.normal}) is indeed represented by
(\ref{Cl2.field}).
\end{proof}

Next we record the relations between the lowest-weight
component $\ass(\pi,S)_0=S(M)$ and the fields introduced in
\S\ref{sec.ass.S.fields}, as well as the relations between
the fields themselves.

\begin{prop} \label{prop.S.fields.ground}
For $f\in C^\infty(M)$, $X\in\T(M)$ and $s\in S(M)$, we have
\begin{align*}
f_0 s=fs,\quad\;\;
cX_0 s=cX\cdot s,\quad\;\;
\ell X_0 s=\nabla_X s,\quad\;\;
f_n s=cX_n s=\ell X_n s=0\;\t{ for }n>0
\end{align*}
where $\cdot$ denotes Clifford multiplication and
$\nabla$ the Levi-Civita connection.
\end{prop}

\begin{proof}
The first three equations follow immediately from (\ref{ass.S.0})
and (\ref{ass.S.gen.fields}).
The rest are true simply because $S(M)\subset\ass(\pi,S)$ is
the component of the lowest weight.
\end{proof}

\begin{prop} \label{prop.S.fields.relations}
Recall the maps $\bullet$, $\{\;\}$, $\{\;\}_\Omega$ in
Proposition \ref{prop.ass.C.VAoid}.
For $f,g\in C^\infty(M)$, $X,Y\in\T(M)$ and
$n,m\in\bb{Z}$, we have the normal-ordered
expansions
\begin{align*}
(fg)_n=f_{n-r}g_r,\quad
c(fX)_n=f_{n-r}\,cX_r,\quad
\ell(fX)_n=\sum_{r\geq 0} f_{n-r}\,\ell X_r
  +\sum_{r<0}\ell X_r f_{n-r}
  -(\ell X\bullet f)_n
\end{align*}
as well as the supercommutation relations
\begin{align*}
&[f_n,g_m]=0,\qquad
[cX_n,f_m]=0,\qquad
[\ell X_n,f_m]=(Xf)_{n+m} \\
&[\ell X_n,cY_m]=c(\nabla_X Y)_{n+m},\qquad
[cX_n,cY_m]=-2\langle X,Y\rangle_{n+m} \\
&[\ell X_n,\ell Y_m]=\ell[X,Y]_{n+m}
  -c\,\Omega(\wt X,\wt Y)^P_{n+m}
  +(\{\ell X,\ell Y\}_\Omega)_{n+m}
  +n\{\ell X,\ell Y\}_{n+m}
\end{align*}
where $\langle\;\;\rangle$ and $\nabla$ denote
the Riemannian metric and Levi-Civita connection on $TM$.
\end{prop}

\begin{proof}
Let us verify the relations at the level of
$\CDO_{\Theta,H}(P)\otimes S$.
For the calculations below, keep in mind
(\ref{ass.S.gen.fields}) and
(\ref{freeVA.comm})--(\ref{freeVA.NOP}).
The first two normal-ordered expansions are
easy, e.g.
\begin{align*}
(\wt{fX})^i_{n-s}\otimes e_{i,s}
=f_{n-s-t}\wt X^i_t\otimes e_{i,s}
=(f_{n-r}\otimes 1)(\wt X^i_{r-s}\otimes e_{i,s}).
\end{align*}
The remaining normal-ordered expansion follows immediately from
the equation
\begin{align*}
\wt{fX}-\tau_i(\wt{fX})^j\Theta_{ij}
=(\wt X-\tau_i\wt X^j\,\Theta_{ij})_{-1}f
-\ell X\bullet f
\end{align*}
which is the result of the first calculation in the proof
of Proposition \ref{prop.ass.C.VAoid}.

For the supercommutators, the first three are easy.
The next two are computed as follows, using (\ref{DX.lift})
and (\ref{Cl}) respectively:
\begin{align*}
&\big[(\wt X-\tau_i\wt X^j\Theta_{ij})_n\otimes 1,\,
  \wt Y^k_{m-r}\otimes e_{k,r}\big] \\
&\qquad\qquad
=\big[\wt X_n,\wt Y^k_{m-r}]
  \otimes e_{k,r}
=(\wt X^i\tau_i\wt Y^k)_{n+m-r}\otimes e_{k,r}
=(\wt{\nabla_X Y})^k_{n+m-r}\otimes e_{k,r} \\
&\big[\wt X^i_{n-r}\otimes e_{i,r},\,
  \wt Y^j_{m-s}\otimes e_{j,s}\big] \\
&\qquad\qquad
=\wt X^i_{n-r}\wt Y^j_{m-s}\otimes[e_{i,r},e_{j,s}]
=-2\wt X^i_{n-r}\wt Y^i_{m+r}\otimes 1
=-2\langle X,Y\rangle_{n+m}\otimes 1
\end{align*}
The last (super)commutator follows immediately from the
equations
\begin{align*}
&\big(\wt X-\tau_i\wt X^j\,\Theta_{ij}\big)_1
\big(\wt Y-\tau_k\wt Y^\ell\,\Theta_{k\ell}\big)
=\{\ell X,\ell Y\} \\
&\big(\wt X-\tau_i\wt X^j\,\Theta_{ij}\big)_0
\big(\wt Y-\tau_k\wt Y^\ell\,\Theta_{k\ell}\big)\\
&\qquad\qquad
=\big(\wt{[X,Y]}-\tau_i\wt{[X,Y]}{}^j\Theta_{ij}\big)
-\Omega(\wt X,\wt Y)^P
-\lambda_0(\Omega(\wt X,\wt Y),\Theta)
+\{\ell X,\ell Y\}_\Omega
\end{align*}
together with (\ref{ass.S.other.fields});
these equations are the results of the second and third
calculations in the proof of Proposition \ref{prop.ass.C.VAoid},
but applied to our current setting where $\lambda^*=\lambda_0$ 
and $\lambda_\rho=2\lambda_0$ (see \S\ref{sec.ass.S}).
\end{proof}

The following construction is manufactured using
precisely the information about $\ass(\pi,S)$ we have 
gathered so far:
the subspace $S(M)$,
the fields in \S\ref{sec.ass.S.fields} and
their relations in Propositions \ref{prop.S.fields.ground}
and \ref{prop.S.fields.relations}.

\begin{subsec} \label{sec.ass.S.free}
{\bf Comparing with a generators-and-relations construction.}
Let $\mc{U}$ be a unital $\bb{Z}/2\bb{Z}$-graded associative
algebra with generators
\begin{align*}
\t{even}:\quad
f_n,\;\,
\ell X_n,\;\,
\alpha_n,\;\,
c\ms{Y}_n;\qquad\quad
\t{odd}:\quad
cX_n
\end{align*}
for $f\in C^\infty(M)$, $X\in\T(M)$,
$\alpha\in\Omega^1(M)$, $\ms{Y}\in\T_v(P)^{\mf{so}_{2d'}}$
and $n\in\bb{Z}$, such that
(i) $1_n=\delta_{n,0}$,
(ii) $f\mapsto f_n$, $X\mapsto\ell X_n$, $\cdots$ are linear and
(iii) they satisfy the supercommutation relations in Proposition
\ref{prop.S.fields.relations}.
The subalgebra $\mc{U}_+\subset\mc{U}$ generated by
$\{f_n,\,\ell X_n,\,cX_n\}_{n\geq 0}$ has an action on $S(M)$ as
described in Proposition \ref{prop.S.fields.ground}.
Let $\wt W^f=\mc{U}\otimes_{\mc{U}_+}S(M)$ and
$W^f=\wt W^f/\sim$ be the quotient obtained by imposing the
normal-ordered expansions in (\ref{ass.S.1formfield}),
Lemma \ref{lemma.verfield} and Proposition
\ref{prop.S.fields.relations}.
Define an operator $L^f_0$ on $W^f$ by
\begin{align} \label{freemod.wt}
L^f_0|_{S(M)}=\frac{d'}{8},\qquad
[L^f_0,f_n]=-nf_n,\qquad
[L^f_0,\ell X_n]=-n\,\ell X_n,\qquad
[L^f_0,cX_n]=-n\,cX_n
\end{align}
which are consistent with all the above relations;
its eigenvalues are called weights.

By construction, there is a unique linear map
\begin{align*}
\Psi:W^f\rightarrow\ass(\pi,S)
\end{align*}
that restricts to the identity on $S(M)$
\footnote{
Since the composition
$\wt W^f\twoheadrightarrow W^f\rightarrow\ass(\pi,S)$ restricts
to the identity on $1\otimes S(M)\subset\wt W^f$, the quotient
map must be injective there.
Hence we may indeed identify $S(M)$ as a subspace of $W^f$.
}
and respects all the above operators.
In fact, we will see that $\Psi$ is an isomorphism.
\end{subsec}

\begin{subsec} \label{sec.ass.S.Virasoro}
{\bf The Virasoro action.}
According to \S\ref{sec.ass.S}, the operators on $\ass(\pi,S)$ 
induced by (\ref{ass.S.Vir}) define a Virasoro action of
central charge $5d'$;
let us denote them by $L_n^M$ for $n\in\bb{Z}$.
By Lemma \ref{lemma.G.S}, each $L_n^M$ can also be represented on
$\CDO_{\Theta,H}(P)\otimes S$ by
\begin{align} \label{ass.S.Virasoro}
\textstyle
(\tau_{i,-1}\tau^i)_n\otimes 1
\,+\,1\otimes L^{C\ell}_n
\,-\,\Theta^a_{n-r}\otimes t^{C\ell}_{a,r}
\,-\,\frac{1}{2}\lambda_0(\Theta_{-1}\Theta)_n\otimes 1.
\end{align}
The key to understanding the entire structure of $\ass(\pi,S)$ is
the observation that $L_0^M$ is generated by the fields described
in \S\ref{sec.ass.S.fields} (in a certain way).
\end{subsec}

\begin{prop} \label{prop.ass.S.wt}
Let $U\subset M$ be an open subset and
$\sigma:U\rightarrow\pi^{-1}(U)\subset P$ a smooth
section of $\pi$.
Both of the weight operators $L^f_0$ on $W^f$ (see
\S\ref{sec.ass.S.free}) and $L^M_0$ on $\ass(\pi,S)$
have the local expression:
\begin{align}  
&\sum_{r\geq 0}\sigma^i_{-r}(\ell\sigma_i)_r
+\sum_{r<0}(\ell\sigma_i)_r\sigma^i_{-r}
+\frac{3}{4}\,\Tr(\Gamma^\sigma_{-r}\Gamma^\sigma_r)
+\sigma^i([\sigma_j,\sigma_k])_{-r}\,\sigma^k_{-s}
  (\Gamma^\sigma_{ji})_{r+s} \nonumber \\
&\qquad\quad
-\frac{1}{2}\sum_{r>0}r(c\sigma_i)_{-r}(c\sigma_i)_r
+\frac{1}{4}(c\sigma_i)_{-r}(c\sigma_j)_{-s}
  (\Gamma^\sigma_{ij})_{r+s}
+\frac{d'}{8}
\label{ass.S.wt.fields}
\end{align}
\footnote{
Recall from \S\ref{sec.review} that, like many other
constructions in this paper, both $W^f$ and $\ass(\pi,S)$
are the spaces of global sections of some underlying sheaves.
}
where $(\sigma_1,\ldots,\sigma_{2d'})$ is the
$C^\infty(U)$-basis of $\T(U)$ induced by $\sigma$;
$(\sigma^1,\ldots,\sigma^{2d'})$ the dual basis of
$\Omega^1(U)$;
and $\Gamma^\sigma=\sigma^*\Theta$.
\end{prop}

\begin{proof}
By assumption, $\pi:P\rightarrow M$ is a lifting of the usual
frame bundle of $TM$ (see \S\ref{sec.ass.S}), so that any local
section of $\pi$ indeed induces a local framing of $TM$.
Let us first study $L^M_0$.
According to \S\ref{sec.ass.S.Virasoro}, $L_0^M$ is represented
by the operator $\ell'+\ell''$ on $\CDO_{\Theta,H}(P)\otimes S$,
where
\begin{align*}
\textstyle
\ell'
=\big(\tau_{i,-1}\tau^i
  -\frac{1}{2}\,\Tr(\Theta_{-1}\Theta)\big)_0\otimes 1,
\qquad
\ell''
=1\otimes L^{C\ell}_0
-\Theta^a_{-r}\otimes t^{C\ell}_{a,r}
+\frac{1}{4}\,\Tr(\Theta_{-1}\Theta)_0\otimes 1.
\end{align*}
Since the calculation in the proof of Proposition
\ref{prop.ass.C.conformal} is valid for any principal frame
algebra, $\ell'$ can be expressed as the zeroth mode of
(\ref{ass.C.conformal.local}).
Our main task is to express $\ell''$ entirely in terms of
$\sigma$ as well.

For the calculations below, adopt again the notations in the
proof of Proposition \ref{prop.ass.C.conformal}.
Let $\theta=g^{-1}dg$.
Also let us introduce a (slight abuse of) notation:
\begin{align} \label{frame.Cl}
(c\sigma_i)_r
=\wt\sigma_{i,r-s}^j\otimes e_{j,s}
=g_{ji,r-s}\otimes e_{j,s},\qquad
i=1,\ldots,2d',\;\;r\in\bb{Z}.
\end{align}
By (\ref{Cl}) these operators satisfy the anticommutation
relations
\begin{align} \label{frame.Cl.comm}
[(c\sigma_i)_r,(c\sigma_j)_s]
=-2\delta_{ij}\delta_{r+s,0}.
\end{align}
Now the first term of $\ell''$ can be written as follows
\begin{align*}
1\otimes L^{C\ell}_0
&=\frac{1}{4}\bigg(
  -\sum_{r\geq 0}r\,(c\sigma_j)_s(c\sigma_k)_t
  +\sum_{r<0}r\,(c\sigma_k)_t(c\sigma_j)_s\bigg)
  \big(g_{ij,-r-s}g_{ik,r-t}\otimes 1\big)
+\frac{d'}{8} \\
&=\frac{1}{4}\bigg(
  -\sum_{t\geq 0}r\,(c\sigma_j)_s(c\sigma_k)_t
  +\sum_{t<0}r\,(c\sigma_k)_t(c\sigma_j)_s
  -\sum_{r\geq 0>t}
    r\,[(c\sigma_j)_s,(c\sigma_k)_t]  \\
&\qquad\qquad\qquad
+\sum_{r<0\leq t}
  r\,[(c\sigma_j)_s,(c\sigma_k)_t]\bigg)
  \big(g_{ij,-r-s}g_{ik,r-t}\otimes 1\big)
+\frac{d'}{8}  \\
&=-\frac{1}{2}\sum_{t>0}t(c\sigma_j)_{-t}(c\sigma_j)_t
+\frac{1}{4}(c\sigma_j)_s(c\sigma_k)_t
  (\theta_{jk,-s-t}\otimes 1)
+\frac{1}{4}\Tr(\theta_{-1}\theta)_0\otimes 1
+\frac{d'}{8}
\end{align*}
The first step follows from (\ref{Cl.Virasoro}) and
(\ref{frame.Cl});
the second step is a rearrangement that is valid because every
sum in sight is finite when applied to arbitrary elements;
and the last step is the result of some algebra and
an application of (\ref{frame.Cl.comm}).
The second term of $\ell''$ can be written as follows
\begin{align*}
-\Theta^a_{-r}\otimes t^{C\ell}_{a,r}
&=\frac{1}{4}\Theta_{ij,-r}\otimes e_{i,r-s}e_{j,s}
=\frac{1}{4}\sum_{s\geq 0}
  \Theta_{ij,-r}\otimes e_{i,r-s}e_{j,s}
-\frac{1}{4}\sum_{s<0}
  \Theta_{ij,-r}\otimes e_{j,s}e_{i,r-s}  \\
&=\frac{1}{4}\bigg(
  \sum_{s\geq 0}(c\sigma_k)_t(c\sigma_\ell)_u
  -\sum_{s<0}(c\sigma_\ell)_u(c\sigma_k)_t\bigg)
  \big((g^{-1}\Theta)_{kj,-s-t}\,g_{j\ell,s-u}
    \otimes 1\big)  \\
&=\frac{1}{4}\,(c\sigma_k)_t(c\sigma_\ell)_u
  \big((g^{-1}\Theta\,g)_{k\ell,-t-u}\otimes 1\big)
+\frac{1}{2}\,\Tr\big((g^{-1}\Theta\,g)_{-1}\theta\big)_0
  \otimes 1
\end{align*}
The first step uses (\ref{Cl.so});
the third step follows from (\ref{frame.Cl}) and is valid thanks
to the normal ordering of $e_{i,r-s}e_{j,s}$;
and the last step is largely similar to the above treatment of
$1\otimes L^{C\ell}_0$.
These calculations, together with (\ref{dg}), yield the following
expression:
\begin{align*}
\ell''
=-\frac{1}{2}\sum_{t>0}
  t(c\sigma_j)_{-t}(c\sigma_j)_t
+\frac{1}{4}(c\sigma_j)_s(c\sigma_k)_t
  (\Gamma^\sigma_{jk,-s-t}\otimes 1)
+\frac{1}{4}\Tr(\Gamma^\sigma_{-1}\Gamma^\sigma)_0
  \otimes 1
+\frac{d'}{8}\,.
\end{align*}
If we combine the earlier comment on $\ell'$ with this
expression of $\ell''$ and recall
(\ref{ass.S.gen.fields})--(\ref{ass.S.other.fields}),
then we find that $\ell'+\ell''$ represents the operator
shown in (\ref{ass.S.wt.fields}).
This proves the claim for $L^M_0$.

Since the expression in (\ref{ass.S.wt.fields}) is generated
by the operators in
(\ref{ass.S.gen.fields})--(\ref{ass.S.other.fields}),
it determines an operator $L'_0$ on $W^f$.
To show $L'_0=L^f_0$, we need to verify that $L'_0$ satisfies
(\ref{freemod.wt}) using only the relations in
(\ref{ass.S.1formfield}), Lemma \ref{lemma.verfield}
and Propositions
\ref{prop.S.fields.ground}--\ref{prop.S.fields.relations}.
The calculations, which we omit, are straightforward.
\end{proof}

\begin{corollary} \label{cor.ass.S.free}
The linear map $\Psi:W^f\rightarrow\ass(\pi,S)$
(see \S\ref{sec.ass.S.free}) is an isomorphism.
\end{corollary}

\begin{proof}
Let $f\in C^\infty(M)$, $\alpha\in\Omega^1(M)$
and $X\in\T(M)$.
By definition, $\Psi$ is an isomorphism on the lowest weight
$d'/8$.
Assume that $\Psi$ is an isomorphism on all weights up to
$(d'/8)+k-1$ for some $k>0$.
Let $u\in\ass(\pi,S)_k$.
For reason of weight as well as (\ref{ass.S.1formfield}),
elements of the form
\begin{align} \label{lowered}
f_n u,\;\ell X_n u,\;cX_n u\;\t{ for }n>0,\qquad
\alpha_n u\;\t{ for }n\geq 0
\end{align}
are in the image of $\Psi$;
then by Proposition \ref{prop.ass.S.wt}, so is
$ku=(L^M_0-d'/8)u$.
This proves the surjectivity of $\Psi$ on weight $(d'/8)+k$.
On the other hand, let $u\in W^f_k\cap\ker\Psi$.
Since the elements (\ref{lowered}) are also in the kernel of $\Psi$,
for reason of weight and (\ref{ass.S.1formfield}) again they must
be trivial;
then by Proposition \ref{prop.ass.S.wt}, so is $ku=(L^f_0-d'/8)u$.
This proves the injectivity of $\Psi$ on weight $(d'/8)+k$.
By induction, $\Psi$ is an isomorphism on all weights.
\end{proof}

The following summarizes our analysis of $\ass(\pi,S)$.

\begin{theorem} \label{thm.ass.S}
Suppose $M^{2d'}$ is a Riemannian manifold with
a spin structure $\pi:P\rightarrow M$ and
a $3$-form $H$ satisfying $dH=\frac{1}{2}\Tr(R\wedge R)$,
where $R$ is the Riemannian curvature.
Let $\CDO_{\Theta,H}(P)$ be the associated principal
frame $(\widehat{\mf{so}}_{2d'},\mathrm{Spin}_{2d'})$-algebra
(see Theorem \ref{thm.CDOP}), $S$ the spinor representation
of $\widehat{\mf{so}}_{2d'}$, and
\begin{align*}
\ass(\pi,S)
=H^{\frac{\infty}{2}+0}\big(
  \widehat{\mf{so}}_{2d'}\,,
  \,\CDO_{\Theta,H}(P)\otimes S\big).
\end{align*}
The lowest-weight component of $\ass(\pi,S)$ is $S(M)$,
the space of sections of the spinor bundle on $M$.
For $f\in C^\infty(M)$ and $X\in\T(M)$, there are associated
fields $\Phi_f(z)$, $\Phi_{\ell X}(z)$ and $\Phi_{cX}(z)$
on $\ass(\pi,S)$ of weights $0$, $1$ and $\frac{1}{2}$
respectively;
their Fourier modes are represented on
$\CDO_{\Theta,H}(P)\otimes S$ by
\begin{align*}
\qquad\qquad
f_n\otimes 1,\qquad
(\wt X-\tau_i\wt X^j\Theta_{ij})_n\otimes 1,\qquad
\wt X^i_{n-r}\otimes e_{i,r},\qquad
\t{for }n\in\bb{Z}.
\end{align*}
The actions of these fields on $S(M)$ generate the entire
space $\ass(\pi,S)$, subject only to the relations stated
in Propositions \ref{prop.S.fields.ground} and 
\ref{prop.S.fields.relations}.
Moreover, there is a Virasoro field on $\ass(\pi,S)$ of central
charge $5d'$, whose Fourier modes are represented on
$\CDO_{\Theta,H}(P)\otimes S$ by
\begin{align*}
\textstyle
\qquad
\qquad
(\tau_{i,-1}\tau^i)_n\otimes 1
+1\otimes L^{C\ell}_n
-\Theta^a_{n-r}\otimes t^{C\ell}_{a,r}
-\frac{1}{4}\Tr(\Theta_{-1}\Theta)_n\otimes 1,\qquad
\t{for }n\in\bb{Z}.
\qquad
\qedsymbol
\end{align*}
\end{theorem}

This description of $\ass(\pi,S)$ allows us to define
a natural filtration as follows.

\begin{subsec} \label{ass.S.PBW}
{\bf PBW filtration.}
Given an increasing sequence of negative integers 
$\n=\{n_1\leq\cdots\leq n_s<0\}$, possibly empty,
let us write
\begin{align*}
|\n|=|n_1|+\cdots+|n_s|\quad (0\t{ if }\n=\{\}),\qquad
\n(i)=\t{number of times }i\t{ appears in }\n
\end{align*}
and regard $\n$ as a partition of $-|\n|$.
For any nonnegative integer $w$, let $\ms{I}_w$ be the set
of triples $(\n;\m;\p)$ of such sequences with
$|\n|+|\m|+|\p|=w$ and $\m(i)\leq 2d'$ for all $i<0$.
Define a partial ordering on $\ms{I}_w$ by declaring
that $(\n;\m\;\p)\prec(\n';\m';\p')$ if one of the
following is true:
\vspace{-0.05in}
\begin{itemize}
\item[$\cdot$]
$|\n|<|\n'|$, or $|\n|=|\n'|$ and $|\m|<|\m'|$
\vspace{-0.08in}
\item[$\cdot$]
$\n'$ is a proper subpartition of $\n$, $\m=\m'$
and $\p=\p'$
\vspace{-0.08in}
\item[$\cdot$]
$\n=\n'$, $\m=\m'$ and $\p$ is a proper subpartition
of $\p'$
\vspace{-0.05in}
\end{itemize}
For example in $\ms{I}_3$, a particular ascending chain is
\begin{align*}
(;;-2,-1)
\prec(;;-3)
\prec(;-1,-1;-1)
\prec(-2;-1;)
\prec(-1,-1;-1;)
\end{align*}
while $(;;-1,-1,-1)$ and $(-1,-1,-1;;)$ are
the unique minimal and maximal elements.

Given a sequence $\n=\{n_1\leq\cdots\leq n_s<0\}$
as above and an $s$-tuple
$\bs{\alpha}=(\alpha_1,\ldots,\alpha_s)$ in
$\Omega^1(M)$, or an $s$-tuple
$\bs{X}=(X_1,\ldots,X_s)$ in $\T(M)$, let us introduce
the notations
\begin{align*}
\bs{\alpha}_\n=\alpha_{1,n_1}\cdots\alpha_{s,n_s},\quad
\bs{cX}_\n=(cX_1)_{n_1}\cdots(cX_s)_{n_s},\quad
\bs{\ell X}_\n=(\ell X_1)_{n_1}\cdots(\ell X_s)_{n_s}\quad
(1\t{ if }\n=\{\})
\end{align*}
which are operators on $\ass(\pi,S)$ (see
\S\ref{sec.ass.S.fields}).
It follows from Theorem \ref{thm.ass.S} that for $w\geq 0$
we have
\begin{align*}
\ass(\pi,S)_w
=\t{span}\Big\{
  \bs{\ell X}_\n\bs{cY}_\m\bs{\alpha}_\p\, s:
  (\n;\m\;\p)\in\ms{I}_w;\,
  \t{all suitable }\bs{\alpha},\bs{X},\bs{Y};\,
  s\in S(M)
\Big\}.
\end{align*}
Indeed, as $f_n=-\frac{1}{n}(df)_n$ for
$f\in C^\infty(M)$ and $n\neq 0$, Propositions
\ref{prop.S.fields.ground} and \ref{prop.S.fields.relations}
allow us to express every element of $\ass(\pi,S)$ in the
indicated form.
For $(\n;\m;\p)\in\ms{I}_w$ consider the subspaces
\begin{align*}
\mc{F}_{\preceq(\n;\m;\p)}
&=\t{span}\big\{
  \bs{\ell X}_{\n'}\bs{cY}_{\m'}\bs{\alpha}_{\p'} s:
  (\n';\m';\p')\preceq(\n;\m;\p)
\big\}\subset\ass(\pi,S)_w \\
\mc{F}_{\prec(\n;\m;\p)}
&=\t{span}\big\{
  \bs{\ell X}_{\n'}\bs{cY}_{\m'}\bs{\alpha}_{\p'} s:
  (\n';\m';\p')\prec(\n;\m;\p)
\big\}\subset\ass(\pi,S)_w
\end{align*}
For the next statement, let $O$ (and $O'$)
stand for one of the operators of the form
$\alpha_n$, $cX_n$ or $\ell X_n$ with $n<0$, and
$fO$ the corresponding operator $(f\alpha)_n$,
$c(fX)_n$ or $\ell(fX)_n$, where
$f\in C^\infty(M)$.
The subspaces $\mc{F}_{\preceq(\n;\m;\p)}$ and
$\mc{F}_{\prec(\n;\m;\p)}$ have the following
properties:
\begin{align*}
\begin{array}{crcrl}
\t{(i)}&\;\;
\cdots OO'\cdots\in\mc{F}_{\preceq(\n;\m;\p)}&
\;\;\Rightarrow&
\cdots[O,O']\cdots\in\mc{F}_{\prec(\n;\m;\p)}\ph{,}& \\
\t{(ii)}&
\cdots Os\in\mc{F}_{\preceq(\n;\m;\p)}&
\;\;\Rightarrow& \ph{\Big(}
\cdots\big((fO)s-O(fs)\big)\in\mc{F}_{\prec(\n;\m;\p)},&
\quad s\in S(M)
\end{array}
\end{align*}
Indeed (i) follows from the supercommutation
relations in Proposition \ref{prop.S.fields.relations}
and (ii) from the normal-ordered expansions there as
well as Proposition \ref{prop.S.fields.ground}.
Consequently, there is a natural isomorphism
\begin{align*}
&\mc{F}_{\preceq(\n;\m;\p)}/\mc{F}_{\prec(\n;\m;\p)} \\
&\qquad\qquad
\cong
\bigg(\bigotimes_{i<0}\t{Sym}^{\p(i)}\Omega^1(M)\bigg)\otimes
\bigg(\bigotimes_{i<0}\t{Sym}^{\n(i)}\T(M)\bigg)\otimes
\bigg(\bigotimes_{i<0}\wedge^{\m(i)}\T(M)\bigg)\otimes S(M)
\end{align*}
where all tensor, symmetric and exterior products
are over $C^\infty(M)$.
Let $q$ be a formal variable.
When all $(\n;\m\;\p)\in\ms{I}_w$ and all $w\geq 0$
are considered, we obtain
\begin{align*}
&\bigoplus_{w\geq 0}\bigg(
  q^w\bigoplus_{(\n;\m;\p)\in\ms{I}_w}
  \mc{F}_{\preceq(\n;\m;\p)}/\mc{F}_{\prec(\n;\m;\p)}
\bigg) \\
&\qquad\qquad\quad\cong
\bigg(\bigotimes_{k\geq 1}\t{Sym}_{q^k}\Omega^1(M)\bigg)
\otimes\bigg(\bigotimes_{k\geq 1}\t{Sym}_{q^k}\T(M)\bigg)
\otimes\bigg(\bigotimes_{k\geq 1}\wedge_{q^k}\T(M)\bigg)
\otimes S(M)
\end{align*}
where $\t{Sym}_t=\sum_{n=0}^\infty t^n\t{Sym}^n$
and $\wedge_t=\sum_{n=0}^{2d'}t^n\wedge^n$ as usual.
\end{subsec}

\newpage

\titleformat{\section}[block]
  {\sc\large\filcenter}
  {Appendix \S\thesection.}{.5em}{}

\appendix

\setcounter{equation}{0}
\section{Vertex Algebroids} 
\label{app.VAoid}

The notion of a vertex algebroid captures the part of
structure of a vertex algebra that involves only its two
lowest weights (see \S\ref{conventions}).
Even though its definition is rather complicated, it serves as
a convenient tool for handling the vertex algebras in this paper.
This appendix reviews, mostly without proof, the category
of vertex algebroids and an adjoint pair of functors between
vertex algebras and vertex algebroids.
For more details, the reader is referred to the original work
\cite{GMS2}.

\begin{defn} \label{ELA}
An {\bf extended Lie algebroid} $(A,\Omega,\T)$ consists of:
\vspace{-0.08in}
\begin{itemize}
\item[$\cdot$]
a commutative, associative $\bb{C}$-algebra with unit $(A,\mb{1})$
\vspace{-0.08in}
\item[$\cdot$]
an $A$-module $\Omega$, together with an $A$-derivation
$d:A\rightarrow\Omega$ such that $\Omega=A\cdot dA$
\vspace{-0.08in}
\item[$\cdot$]
another $A$-module $\T$, equipped with a Lie bracket $[\;\;]$
\vspace{-0.08in}
\item[$\cdot$]
an $A$-linear map of Lie algebras $\T\rightarrow\t{End}\,A$,
denoted by $X\mapsto X$
\vspace{-0.08in}
\item[$\cdot$]
a $\bb{C}$-linear map of Lie algebras 
$\T\rightarrow\t{End}\,\Omega$, denoted by $X\mapsto L_X$
\vspace{-0.08in}
\item[$\cdot$]
an $A$-bilinear pairing $\Omega\times\T\rightarrow A$, denoted 
by $(\alpha,X)\mapsto\alpha(X)$
\vspace{-0.08in}
\end{itemize}
Furthermore, it is required that:
\vspace{-0.08in}
\begin{itemize}
\item[$\cdot$]
the $\T$-actions on $A$ and $\Omega$ commute with $d$
\vspace{-0.08in}
\item[$\cdot$]
the $\T$-actions on $A,\Omega$ and $\T$ (via $[\;\;]$) satisfy
the Leibniz rule with respect to $A$-multiplication
\vspace{-0.08in}
\item[$\cdot$]
$df(X)=Xf$ for $f\in A$ and $X\in\T$
\end{itemize}
\end{defn}

\begin{defn} \label{ELA.mor}
A {\bf map of extended Lie algebroids} 
$\varphi:(A,\Omega,\T)\rightarrow(A',\Omega',\T')$
is simply a map of ordered triples that respects
the extended Lie algebroid structures.
Composition of maps is defined in the obvious way.
\end{defn}

\begin{defn} \label{VAoid}
A {\bf vertex algebroid}
$(A,\Omega,\T,\bullet,\{\;\},\{\;\}_\Omega)$ consists of
an extended Lie algebroid $(A,\Omega,\T)$ together with three
$\bb{C}$-bilinear maps
\begin{align*}
\bullet:\T\times A\rightarrow\Omega,\qquad
\{\;\}:\T\times\T\rightarrow A,\qquad
\{\;\}_\Omega:\T\times\T\rightarrow\Omega
\end{align*}
that satisfy the following identities:
\vspace{-0.08in}
\begin{itemize}
\item[$\cdot$]
$\{X,Y\}=\{Y,X\}$
\vspace{-0.08in}
\item[$\cdot$]
$d\{X,Y\}=\{X,Y\}_\Omega+\{Y,X\}_\Omega$
\vspace{-0.08in}
\item[$\cdot$]
$X\bullet(fg)-(gX)\bullet f-f(X\bullet g)=-(Xf)dg$
\vspace{-0.08in}
\item[$\cdot$]
$\{X,fY\}-f\{X,Y\}=-(Y\bullet f)(X)+[X,Y]f$
\vspace{-0.08in}
\item[$\cdot$]
$\{X,fY\}_\Omega-f\{X,Y\}_\Omega=
-L_X(Y\bullet f)+[X,Y]\bullet f+Y\bullet(Xf)$
\vspace{-0.08in}
\item[$\cdot$]
$X\{Y,Z\}-\{[X,Y],Z\}-\{Y,[X,Z]\}=
\{X,Y\}_\Omega(Z)+\{X,Z\}_\Omega(Y)$
\vspace{-0.08in}
\item[$\cdot$]
$L_X\{Y,Z\}_\Omega-L_Y\{X,Z\}_\Omega+L_Z\{X,Y\}_\Omega
+\{X,[Y,Z]\}_\Omega-\{Y,[X,Z]\}_\Omega-\{[X,Y],Z\}_\Omega$ \\
\ph{WW}$=d\,\Big(\{X,Y\}_\Omega(Z)\Big)$
\vspace{-0.08in}
\end{itemize}
for $f,g\in A$ and $X,Y,Z\in\T$.
\end{defn}

{\it Remark.}
This definition is equivalent to, but slightly different
from both the original one in \cite{GMS2} with the
notations $(\gamma,\langle\;\rangle,c)$, and the one in
\cite{myCDO} with the notations
$(\ast,\{\;\},\{\;\}_\Omega)$.
The various notations are related as follows:
\begin{align*}
&X\bullet f=-\gamma(f,X)+dXf=f\ast X+dXf& \\
&\ds\{X,Y\}=\langle X,Y\rangle,\qquad
\{X,Y\}_\Omega=-c(X,Y)+\frac{1}{2}\langle X,Y\rangle.&
\end{align*}
In this paper we adopt the above definition to simplify
the description of certain vertex algebras.

\begin{defn} \label{VAoid.mor} 
A {\bf map of vertex algebroids} 
\begin{align*}
(\varphi,\Delta):(A,\Omega,\T,\bullet,\{\;\},\{\;\}_\Omega)\rightarrow
(A',\Omega',\T',\bullet',\{\;\}',\{\;\}'_\Omega)
\end{align*}
consists of a map of extended Lie algebroids 
$\varphi:(A,\Omega,\T)\rightarrow(A',\Omega',\T')$ 
together with a $\bb{C}$-linear map 
$\Delta:\T\rightarrow\Omega'$ such that
\vspace{-0.08in}
\begin{itemize}
\item[$\cdot$]
$\varphi X\bullet'\varphi f-\varphi(X\bullet f)=
\Delta(fX)-(\varphi f)\Delta(X)$
\vspace{-0.08in}
\item[$\cdot$]
$\{\varphi X,\varphi Y\}'-\varphi\{X,Y\}=
-\Delta(X)(\varphi Y)-\Delta(Y)(\varphi X)$
\vspace{-0.08in}
\item[$\cdot$]
$\{\varphi X,\varphi Y\}'_\Omega-\varphi\{X,Y\}_\Omega=
-L_{\varphi X}\Delta(Y)+L_{\varphi Y}\Delta(X)
-d\big(\Delta(X)(\varphi Y)\big)+\Delta([X,Y])$
\vspace{-0.05in}
\end{itemize}
for $f\in A$ and $X,Y\in\T$.
Composition of maps is defined by
\begin{align*}
(\varphi',\Delta')\circ(\varphi,\Delta)=
(\varphi'\varphi,\,\varphi'\Delta+\Delta'\varphi|_\T).
\end{align*}
\end{defn}

\begin{subsec} \label{VA.VAoid}
{\bf From vertex algebras to vertex algebroids:~objects.}
Given a vertex algebra $(V,\bs{1},T,Y)$, consider the
following subquotient spaces
\begin{align*}
A:=V_0,\qquad
\Omega:=A_0(TA),\qquad
\T:=V_1/\Omega.
\end{align*}
Choose a splitting $s:\T\rightarrow V_1$ of the quotient
map to obtain an identification of vector spaces
\begin{align} \label{V1.iden}
\Omega\oplus\T\cong V_1,\qquad (\alpha,X)\mapsto\alpha+s(X).
\end{align}
The part of vertex algebra structure on $V$ involving
only the two lowest weights consists of an element
$\mathbf 1\in V_0$, a linear map $T:V_0\rightarrow V_1$,
and eight bilinear maps
\begin{align*}
V_i\times V_j\rightarrow V_k,\quad
(u,v)\mapsto u_{j-k}v,\qquad
\t{for }i,j,k=0,1
\end{align*}
satisfying a set of (Borcherds) identities.
All these data, when rephrased in terms of the
identification (\ref{V1.iden}), are equivalent to a vertex
algebroid $(A,\Omega,\T,\bullet,\{\;\},\{\;\}_\Omega)$.
The extended Lie algebroid $(A,\Omega,\T)$ consists of precisely those 
data that are independent of the choice of $s$, namely
\begin{align*}
\begin{array}{lll}
fg:=f_0 g &
f\alpha:=f_0\alpha &
fX:=f_0 s(X)\;\t{mod}\;\Omega \vss \\
Xf:=s(X)_0 f\ph{aa} &
L_X\alpha:=s(X)_0 \alpha\ph{aa} &
[X,Y]:=s(X)_0 s(Y)\;\t{mod}\;\Omega \vss \\
df:=Tf &
\alpha(X):=\alpha_1 s(X) &
\end{array}
\end{align*}
\footnote{
For example, the definition of $Xf$ is indeed independent
of $s$ because $\alpha_0 f=0$ for $f\in A$ and $\alpha\in\Omega$.
}
for $f,g\in A$, $\alpha\in\Omega$ and $X,Y\in\T$.
The rest of the vertex algebroid structure is given by
\begin{align}
X\bullet f\;\;\;&:=s(X)_{-1}f-s(fX) \nonumber \\
\{X,Y\}\;\;&:=s(X)_1 s(Y)  \label{VA.VAoid.456} \\
\{X,Y\}_\Omega&:=s(X)_0 s(Y)-s([X,Y]) \nonumber
\end{align}
for $f\in A$ and $X,Y\in\T$.
\end{subsec}

\begin{subsec}
{\bf From vertex algebras to vertex algebroids:~morphisms.}
Let $\Phi:V\rightarrow V'$ be a map of vertex algebras.
Then let $(A,\Omega,\T,\cdots)$ be the vertex algebroid
associated to $V$ and a splitting $s:\T\rightarrow V_1$;
and similarly $(A',\Omega',\T',\cdots)$ associated to $V'$
and $s':\T'\rightarrow V'_1$.
The part of data of $\Phi$ involving only the two lowest weights,
when rephrased in terms of identifications like
(\ref{V1.iden}), are equivalent to a map of vertex algebroids
$(\varphi,\Delta)$.
It consists of the obvious map of ordered triples 
$\varphi:(A,\Omega,\T)\rightarrow(A',\Omega',\T')$ induced by $\Phi$,
and a map $\Delta:\T\rightarrow\Omega'$ defined by
\begin{align*}
\Delta(X)=\Phi s(X)-s'(\varphi X).
\end{align*}
\end{subsec}

\begin{subsec} \label{VAoid.VA}
{\bf From vertex algebroids to vertex algebras:~objects.}
Let $(A,\Omega,\T,\bullet,\{\;\},\{\;\}_\Omega)$ be a vertex
algebroid.
In this discussion, $f,g$ (resp.~$\alpha,\beta$) (resp.~$X,Y$)
always denote general elements of $A$ (resp.~$\Omega$)
(resp.~$\T$).
Let $\mc{U}$ be a unital associative algebra with generators
$f_n,\,\alpha_n,\,X_n,\,n\in\bb{Z}$, and the relations
\begin{align} \label{freeVA.comm}
\begin{array}{l}
f\mapsto f_n,\;
\alpha\mapsto\alpha_n,\;
X\mapsto X_n\;\;
\t{are linear} \vs \\
\mathbf 1_n=\delta_{n,0},\qquad
(df)_n=-nf_n,\qquad
\lbrack f_n,g_m]
  =[f_n,\alpha_m]
  =[\alpha_n,\beta_m]=0 \vs \\ 
\lbrack X_n,f_m]=(Xf)_{n+m},\qquad
[X_n,\alpha_m]
  =(L_X\alpha)_{n+m}+n\alpha(X)_{n+m} \vss \\
\lbrack X_n,Y_m]
  =[X,Y]_{n+m}+(\{X,Y\}_\Omega)_{n+m}+n\{X,Y\}_{n+m}
\end{array} 
\end{align}
for $n,m\in\bb{Z}$.
The subalgebra $\mc{U}_+\subset\mc{U}$ generated by
$\{f_n\}_{n>0}$ and $\{\alpha_n,X_n\}_{n\geq 0}$
has a trivial action on $\bb{C}$.
Let $\wt{V}:=\mc{U}\otimes_{\mc{U}_+}\bb{C}$ be the
induced $\mc{U}$-module and $V:=\wt{V}/\sim$ the
quotient module obtained by imposing the following
relations for $v\in\wt{V}$:
\begin{align} \label{freeVA.NOP}
\begin{array}{lcl}
(fg)_n v & \sim & \sum_{k\in\bb{Z}} f_{n-k}g_k v \vss \\
(f\alpha)_n v & \sim & \sum_{k\in\bb{Z}} f_{n-k}\alpha_k v \vss \\
(fX)_n v & \sim & 
  \sum_{k\geq 0}f_{n-k}X_k v+\sum_{k<0}X_k f_{n-k}v-(X\bullet f)_n v 
\end{array}
\end{align}
Notice that the summations are always finite.
It follows from the axioms of a vertex algebroid that 
(\ref{freeVA.comm})--(\ref{freeVA.NOP}) are consistent.
\footnote{
For example, $[X_n,(fY)_m]$ can be evaluated
by either taking the commutator first or
expanding $(fY)_m$ first.
The resulting identity is already implied by
the vertex algebroid axioms and does not lead
to a new relation.
}
Define a vertex algebra structure on $V$ as follows.
The vacuum $\mathbf 1\in V$ is the coset of
$1\otimes 1\in\wt{V}$.
The infinitesimal translation $T$ and weight operator
$L_0$ are determined by
\begin{align*}
\begin{array}{llll}
T\mathbf 1=0,\ph{aaa} &
[T,f_n]=(1-n)f_{n-1},\ph{aa} &
[T,\alpha_n]=-n\alpha_{n-1},\ph{aa} &
[T,X_n]=-nX_{n-1} \vss \\
L_0\mathbf 1=0, &
[L_0,f_n]=-nf_n, &
[L_0,\alpha_n]=-n\alpha_n, &
[L_0,X_n]=-nX_n
\end{array}
\end{align*}
which are consistent with (\ref{freeVA.comm})--(\ref{freeVA.NOP});
notice that actions of $f_n,\alpha_n,X_n$ change weights by
$-n$.
Identify $A,\Omega,\T$ as subspaces of $V$ via
$f=f_0\mb{1}$, $\alpha=\alpha_{-1}\mb{1}$,
$X=X_{-1}\mb{1}$, and associate to them the
fields
\begin{align*}
\begin{array}{c}
\sum_n f_n z^{-n}\,,\qquad
\sum_n \alpha_n z^{-n-1}\,,\qquad
\sum_n X_n z^{-n-1}
\end{array}
\end{align*}
which are mutually local by (\ref{freeVA.comm}).
Now apply the Reconstruction Theorem \cite{FB-Z}.
\end{subsec}

{\it Remark.}
If $V'$ is another vertex algebra whose associated vertex
algebroid is $(A,\Omega,\T,\bullet,\{\;\},\{\;\}_\Omega)$,
then by construction there is a canonical map of vertex
algebras $V\rightarrow V'$.
If it is surjective (resp.~bijective), then $V'$ is
said to be generated (resp.~freely generated) by
the vertex algebroid.

\begin{subsec} \label{VAoid.VA.mor} 
{\bf From vertex algebroids to vertex algebras:~morphisms.}
A map of vertex algebroids
\begin{align*}
(\varphi,\Delta):(A,\Omega,\T,\cdots)\rightarrow
  (A',\Omega',\T',\cdots)
\end{align*}
induces a map $\Phi:V\rightarrow V'$ between
the freely generated vertex algebras by
\begin{align*}
\begin{array}{lll}
\Phi f=\varphi f, &
\Phi\alpha=\varphi\alpha, & 
\Phi X=\varphi X+\Delta(X) \vss \\
\Phi\circ f_n=(\Phi f)_n\circ\Phi,\qquad &
\Phi\circ\alpha_n=(\Phi\alpha)_n\circ\Phi,\qquad &
\Phi\circ X_n=(\Phi X)_n\circ\Phi
\end{array}
\end{align*}
for $f\in A,\,\alpha\in\Omega,\,X\in\T,n\in\bb{Z}$.
Indeed, these equations are consistent with 
(\ref{freeVA.comm})--(\ref{freeVA.NOP}).
\end{subsec}

\begin{lemma} \label{lemma.newVAoid}
Given a vertex algebroid
$(A,\Omega,\T,\bullet,\{\;\},\{\;\}_\Omega)$,
an isomorphism of extended Lie algebroids 
$\varphi:(A,\Omega,\T)\rightarrow(A',\Omega',\T')$ and
a $\bb{C}$-linear map $\Delta:\T\rightarrow\Omega'$,
if we define
\begin{align*}
\bullet':A'\times\T'\rightarrow\Omega',\qquad
\{\;\}':\T'\times\T'\rightarrow A',\qquad
\{\;\}'_\Omega:\T'\times\T'\rightarrow\Omega'
\end{align*}
by the equations in Definition \ref{VAoid.mor}, then 
$(A',\Omega',\T',\bullet',\{\;\}',\{\;\}'_\Omega)$
is a vertex algebroid and $(\varphi,\Delta)$ is
by construction an isomorphism between the two 
vertex algebroids. $\qedsymbol$
\end{lemma}

\begin{example} \label{VAoid.Lie}
{\bf The vertex algebroids associated to a Lie algebra.}
Consider a Lie algebra $\mf{g}$ over $\bb{C}$ and
a vertex algebroid of the form
$(\bb{C},0,\mf{g},0,\lambda,0)$ with $\mf{g}$ acting
trivially on $\bb{C}$.
The second, fourth and last components are trivial
by necessity.
By Definition \ref{VAoid}, the conditions
on $\lambda:\mf{g}\times\mf{g}\rightarrow\bb{C}$ are
\begin{align*}
\lambda(X,Y)=\lambda(Y,X),\qquad
\lambda([X,Y],Z)+\lambda(Y,[X,Z])=0
\end{align*}
i.e.~it is an invariant symmetric bilinear form on $\mf{g}$.
Let
\begin{align*}
V_\lambda(\mf{g})
  =\t{vertex algebra freely generated by }
  (\bb{C},0,\mf{g},0,\lambda,0).
\end{align*}
In the case $\mf{g}$ is finite-dimensional, simple
and $\lambda$ equals $k$ times the normalized Killing
form, this is the same as the affine vertex algebra
$V_k(\mf{g})$.
\cite{Kac,FB-Z}
\end{example}

\begin{subsec} \label{freeVA.PBW}
{\bf PBW filtration of a freely generated vertex algebra.}
Given an increasing sequence of negative integers 
$\n=\{n_1\leq\cdots\leq n_s<0\}$, possibly empty,
we write
\begin{align*}
|\n|=|n_1|+\cdots+|n_s|\quad (0\t{ if }\n=\{\}),\qquad
\n(i)=\t{number of times }i\t{ appears in }\n
\end{align*}
and regard $\n$ as a partition of $-|\n|$.
For $w\geq 0$, let $\ms I_w$ be the set of pairs
$(\n;\m)$ of such sequences that satisfy
$|\n|+|\m|=w$.
Define a partial ordering on $\ms I_w$ by declaring that
$(\n;\m)\prec(\n';\m')$ if
\vspace{-0.05in}
\begin{itemize}
\item[$\cdot$]
$|\n|<|\n'|$, or
\vspace{-0.08in}
\item[$\cdot$]
$\n'$ is a proper subpartition of $\n$ and $\m=\m'$, or
\vspace{-0.08in}
\item[$\cdot$]
$\n=\n'$ and $\m$ is a proper subpartition of $\m'$
\vspace{-0.05in}
\end{itemize}
For example, in $\ms I_3$ a particular chain is given by
$(;-2,-1)\prec(;-3)\prec(-2;-1)\prec(-1,-1;-1)$,
while $(;-1,-1,-1)$ and $(-1,-1,-1;)$ are
the unique minimal and maximal elements.

Consider the vertex algebra $V$ constructed in
\S\ref{VAoid.VA}.
Given a sequence $\n=\{n_1\leq\cdots\leq n_s<0\}$
as above and an $s$-tuple
$\bs{\alpha}=(\alpha_1,\ldots,\alpha_s)$ in
$\Omega$, or an $s$-tuple
$\bs{X}=(X_1,\ldots,X_s)$ in $\T$, let
\begin{align*}
\bs{\alpha}_\n=\alpha_{1,n_1}\cdots\alpha_{s,n_s},\qquad
\bs{X}_\n=X_{1,n_1}\cdots X_{s,n_s}\qquad
(1\t{ if }\n=\{\})
\end{align*}
as operators on $V$.
It follows from the relations in (\ref{freeVA.comm}) that
\begin{align*}
V_w=\t{span}\Big\{
  \bs{X}_\n\bs{\alpha}_\m f:
  (\n;\m)\in\ms I_w;\,
  \t{all suitable }\bs{\alpha},\bs{X};\,
  f\in A
\Big\}.
\end{align*}
For each $(\n,\m)\in\ms I_w$ define the subspaces
\begin{align*}
\mc{F}_{\preceq(\n;\m)}
&=\t{span}\big\{
  \bs{X}_{\n'}\bs{\alpha}_{\m'}f:
  (\n';\m')\preceq(\n;\m)
\big\}\subset V_w \\
\mc{F}_{\prec(\n;\m)}
&=\t{span}\big\{
  \bs{X}_{\n'}\bs{\alpha}_{\m'}f:
  (\n';\m')\prec(\n;\m)
\big\}\subset V_w
\end{align*}
For the next statement, let $O$ (and $O'$)
stand for an operator of the form
$\alpha_n$ or $X_n$ with $n<0$, and
$fO$ the corresponding operator $(f\alpha)_n$
or $(fX)_n$, where $f\in A$.
Observe that the subspaces just defined have the
following properties:
\begin{align*}
\begin{array}{crcrl}
\t{(i)}&\;\;
\cdots OO'\cdots\in\mc{F}_{\preceq(\n;\m)}&
\;\;\Rightarrow&
\cdots[O,O']\cdots\in\mc{F}_{\prec(\n;\m)}\ph{,}& \\
\t{(ii)}&
\cdots Og\in\mc{F}_{\preceq(\n;\m)}&
\;\;\Rightarrow& \ph{\Big(}
\cdots\big((fO)g-O(fg)\big)\in\mc{F}_{\prec(\n;\m)},&
\quad f,g\in A
\end{array}
\end{align*}
Indeed, (i) follows from (\ref{freeVA.comm}) and
(ii) from (\ref{freeVA.NOP}).
These properties imply that
\begin{align*}
\mc{F}_{\preceq(\n;\m)}/\mc{F}_{\prec(\n;\m)}
\cong
\left(\bigotimes_{i<0}\t{Sym}^{\n(i)}\T\right)\otimes
\left(\bigotimes_{i<0}\t{Sym}^{\m(i)}\Omega\right)
\end{align*}
where the tensor and symmetric products are over $A$.
Let $q$ be a formal variable.
When all $(\n;\m)\in\ms I_w$ and all $w\geq 0$
are considered, we obtain an isomorphism
\begin{align*}
\bigoplus_{w\geq 0}\left(
  q^w\bigoplus_{(\n;\m)\in\ms I_w}
  \mc{F}_{\preceq(\n;\m)}/\mc{F}_{\prec(\n;\m)}
\right)\cong
\left(\bigotimes_{k\geq 1}\t{Sym}_{q^k}\T\right)\otimes
\left(\bigotimes_{k\geq 1}\t{Sym}_{q^k}\Omega\right)
\end{align*}
where $\t{Sym}_t=\sum_{n=0}^\infty t^n\t{Sym}^n$
as usual.
\footnote{
If we extend the partial ordering on $\ms I_w$ to
a total ordering, the latter will induce in
an obvious way a filtration on $V_w$ whose
associated graded space is the coefficient of
$q^w$.
}
The subspaces $\mc{F}_{\preceq(\n;\m)}$,
$\mc{F}_{\prec(\n;\m)}$ and the above isomorphisms
are natural, i.e.~respected by maps described in
\S\ref{VAoid.VA.mor}.
\end{subsec}

{\it Remark.}
The definition of the subspaces $\mc{F}_{\preceq(\n;\m)}$ and
$\mc{F}_{\prec(\n;\m)}$ make sense for any vertex algebra.

\begin{lemma} \label{lemma.conf.gen}
If a vertex algebra $V$ has a conformal vector $\nu$ contained
in $\mc{F}_{\preceq(-1;-1)}\subset V_2$, then it is generated
by its associated vertex algebroid.
\end{lemma}

\begin{proof}
Consider the vertex algebroid $(A,\Omega,\T,\cdots)$ associated
to $V$ and the vertex subalgebra $V'\subset V$ it generates.
Let $f,g\in A$; $\alpha,\beta\in\Omega$; and $X\in\T$.
By definition, $V'_0=V_0$ and $V'_1=V_1$.
Suppose $V'_i=V_i$ for $i\leq k-1$, for some positive $k$, and let
$u\in V_k$.
It suffices to prove that $u\in V'$.

Clearly $\alpha_r u,\,X_r u\in V_{k-r}\subset V'$ for $r>0$.
Also, it follows from
\begin{align} \label{fdg.0}
(fdg)_0
=\sum_{s\in\bb{Z}}f_{-s}(dg)_s
=-\sum_{s>0}sf_{-s}g_s+\sum_{s>0}sg_{-s}f_s
\end{align}
that we have $(fdg)_0 u\in V'$.
Hence in fact
\begin{align*}
\alpha_r u\in V'\;\;\t{for}\;\;r\geq 0,\qquad
X_r u\in V'\;\;\t{for}\;\;r>0.
\end{align*}
This easily implies that $(\alpha_{-1}\beta)_0 u$,
$(\alpha_{-2}\mb{1})_0 u$ and $(X_{-1}\alpha)_0 u$
must all belong to $V'$.
By assumption, $\nu$ is a sum of elements of the form
$\alpha_{-1}\beta$, $\alpha_{-2}\mb{1}$ and $X_{-1}\alpha$,
so that $ku=L_0 u=\nu_0 u\in V'$.
Since $k>0$, we have $u\in V'$ as desired.
\end{proof}

\begin{lemma} \label{lemma.conf.ideal}
If a vertex algebra $V$ has a conformal vector $\nu$
contained in $\mc{F}_{\preceq(-1;-1)}\subset V_2$, then it
has no nontrivial ideal consisting only of positive weights.
\end{lemma}

\begin{proof}
Use the same notations as in the proof of Lemma
\ref{lemma.conf.gen}.
Let $I\subset V$ an ideal with $I_0=0$.
Suppose $I_i=0$ for $i\leq k-1$, for some positive $k$, and let
$u\in I_k$.
It suffices to prove that $u=0$.

Clearly $\alpha_r u,\,X_r u\in I_{k-r}=0$ for $r>0$.
Also, it follows from (\ref{fdg.0}) that $(fdg)_0 u=0$.
Hence in fact
\begin{align*}
\alpha_r u=0\;\;\t{for}\;\;r\geq 0,\qquad
X_r u=0\;\;\t{for}\;\;r>0.
\end{align*}
This easily implies that
$(\alpha_{-1}\beta)_0 u=(\alpha_{-2}\mb{1})_0 u
=(X_{-1}\alpha)_0 u=0$.
By assumption, $\nu$ is a sum of elements of the form
$\alpha_{-1}\beta$, $\alpha_{-2}\mb{1}$ and
$X_{-1}\alpha$, so that $ku=L_0 u=\nu_0 u=0$.
Since $k>0$, we have $u=0$ as desired.
\end{proof}

\newpage

{\footnotesize

\vspace{0.1in}
\rule{2.5in}{.1pt}
\vspace{0.1in}

\noindent
{\sc School of Mathematics \& Statistics, University of Sheffield,
  Hicks Building, Sheffield S3 7RH, U.K.} \\
{\it Email address:}
{\sf p.cheung@sheffield.ac.uk}

}

\end{document}